\documentclass{article}
\usepackage{graphicx} % Required for inserting images
\usepackage{geometry}
\usepackage{ulem}
\usepackage{amsmath}
\usepackage{amssymb}
\usepackage{amsfonts}
\usepackage{multirow}
\usepackage{amsthm}
\usepackage{xcolor}
\usepackage{natbib}
\usepackage{comment}
\usepackage{enumitem}
\usepackage{hyperref}
\usepackage{algorithm}
\usepackage{algpseudocode}
\usepackage{rotating}
\newtheorem{theorem}{Theorem}
\newtheorem{lemma}{Lemma}
\newtheorem{corollary}{Corollary}
\newtheorem{proposition}{Proposition}
\theoremstyle{definition}
\newtheorem{definition}{Definition}
\newtheorem{assumption}{A}
\newtheorem{example}{Example}
\newtheorem{remark}{Remark}

 {\begin{list}{}%
         {\setlength{\leftmargin}{#1}}%
         \item[]%
 }
 {\end{list}}

\newcommand{\iidsim}{\overset{iid}{\sim}} % iid simulated 
\newcommand{\di}{{\rm d}}

\newcommand{\seq}[1]{({#1})_{n \geq 0}}
\graphicspath{{figures/}}
\usepackage{xr}
\title{Bayesian Predictive Inference Beyond Martingales}
 \author{Marco Battiston$^{1}$ and Lorenzo Cappello$^{2,3}$\\
 		$^1$School of Mathematical Sciences, Lancaster University\\
	$^2$ Department of Economics and Business, Universitat Pompeu Fabra\\
$^3$ Data Science Center, Barcelona School of Economics}
\begin{document}

\maketitle

\begin{abstract}%
There is growing interest in the so-called Bayesian Predictive Inference approach, which enables Bayesian inference without specifying a likelihood function or prior distribution and without resorting to MCMC methods. Instead, only a sequence of predictive distributions for the observations is required, and inference on the unknown estimand can be performed, cheaply in parallel, using bootstrap-type schemes. %, which are fast and easily parallelizable.
%There is a growing interest in Bayesian inferential procedures that bypass the need of specifying likelihood and prior of the model, and instead only assume a sequence of predictive distributions for the observations.
%This idea is indeed at the core of the so-called Posterior Resampling approach, in which bootstrap-type schemes can be used to learn the unknown estimand, by specifying only a sequence of predictive rules and without the need of any MCMC. 
Understanding which classes of predictive distributions can be used within this framework, is still a key open question. We relax commonly used probabilistic assumptions on the observations, namely exchangeability and conditionally identical distribution, and on their predictive distributions, which are measure-valued martingales, by introducing the new class of \emph{Almost Conditionally Identically Distributed} (a.c.i.d.) random variables. This class assumes that the predictive distributions are measure-valued almost supermartingales, and is parametrized by a sequence of parameters $(\xi_{n})_{n\geq 0}$, which regulate the decay of conditional dependence among future observations. Under mild summability assumptions on $(\xi_{n})_{n\geq 0}$, the resulting sequence of observations is shown to be asymptotically exchangeable, hence amenable to Bayesian Predictive Inference %Posterior Resampling 
techniques. A.c.i.d. random variables arise naturally in recursive algorithms, and include classic approaches in Statistics and Learning Theory, such as kernel estimators, and more novel ones, such as the parametric Bayesian bootstrap. 
\end{abstract}

\begin{keywords} Almost Supermartingales, Conditional identical distribution, Kernel Methods, Martingale Posterior, Parametric Bootstrap, Predictive Resampling.
\end{keywords}

\section{Introduction} \label{sec:intro}

%\subsection{Motivation}

%\indent 
Prediction lies at the foundation of Bayesian Statistics. Traditionally, much research has focused on the relationship between a specified statistical model and prior, and the resulting predictive distribution of future observations given current information \citep{g93,s10,cc18,e20}; see also \cite{bdlpr26} and \cite{fortini2024exchangeability} for recent reviews. However, when the primary goal is inference, a (Bayesian) statistician usually specifies a statistical model and a prior, and then, given a sample, computes the corresponding posterior distribution. Predictive distributions do not play a role in this process; possibly except for the necessary checks of model fit \citep[e.g. predictive checks,][]{gelman1996posterior}. A recent stream of literature is taking a new stance, suggesting that it is possible to do Bayesian inference by simply specifying a \textit{predictive rule} - how to update the prediction as new information becomes available - fully bypassing the need for models and priors \citep{fortini2020quasi,berti2021class,fong2021martingale,holmes2023statistical,fortini2023prediction}. The peculiarity of this approach is that prediction is done with inference in mind rather than forecasting. This emerging area, often referred to as \textit{Bayesian predictive inference} or \textit{modeling}, is the focus of the paper. 

%Several reasons support the popularity of this growing research area. First, it bypasses the need to specify statistical model and  subjective probabilities on latent parameters, some of which could have no direct interpretation or be high-dimensional. All one needs is a probabilistic learning rule on the observables, which is an assumption that can be more formally evaluated, for example, via predictive checks \citep{gelman1996posterior}.  Second, one is free to do inference using any tools at disposal, including those developed recently to deal with model misspecification. Third, it can be computationally appealing: predictive resampling can avoid Markov chain Monte Carlo methods and be parallelized, making inference substantially faster at times \citep{fong2021martingale}. These computational gains includes predictive rules develop to obtain "fast" approximation to certain Bayesian models, which can be used within this framework \citep{fortini2020quasi}. Lastly, it has the potential to provide a bridge between the two cultures \cite{breiman2001statistical}, statistical inference and prediction. However, the current framework is somewhat failing at exploiting this connection for reasons we discuss next. This paper makes a proposal to partially close this gap.  

Several reasons support the popularity of this growing research area. First, it bypasses the need to specify a statistical model and subjective probabilities on latent parameters, which may lack interpretability or be high-dimensional. Instead, it relies on a probabilistic learning rule on the observables - an assumption that can be formally assessed, and %for example, via predictive checks \citep{gelman1996posterior}. This also
aligns with de Finetti's view that probabilistic statements should be made on observable events. %Second, inference can be performed using any available tools, including recent methods for handling model misspecification. 
Second, it is computationally appealing: one can avoid MCMC methods, and rely on numerical methods that speed up inference \citep{fong2021martingale,fortini2020quasi}. %These gains include predictive rules developed as fast approximations to Bayesian models, which fit naturally within this framework \citep{fortini2020quasi}. 
Lastly, it has the potential to bridge the two cultures —statistical inference and prediction \citep{breiman2001statistical} — by enabling the use of machine learning algorithms for inference. However, the current framework requires stringent assumptions on the predictive rule, and many popular algorithms violate these. This paper proposes a relaxation of these assumptions to enlarge the class of algorithms admissible in predictive inference, thus paving the way for this bridge.

Let $\theta$ be an estimand and $x_{1:n}$ a sample. The key observation underlying predictive inference is that the uncertainty described by a posterior distribution $\mathbb{P}(\theta|x_{1:n})$ % follows from the fact that the complete population is not observed, but rather we have a finite sample. With complete population sampling, there is no estimation uncertainty. The uncertainty in the parameter estimates follows from the uncertainty in the observations that are yet to be collected. The predictive rule describes this second type of uncertainty. This is done
can be equivalently described through the predictive distribution of the yet-to-be-seen observations $\mathbb{P}(X_{n+1:\infty}|x_{1:n})$, or equivalently a sequence of one-step ahead distributions $(\mathbb{P}(X_{n+m}|x_{1:n+m-1}))_{m \geq 1}$. The scheme to move from prediction to inference is called \textit{Predictive Resampling}, which is a generalization of the Bayesian bootstrap \citep{rubin1981bayesian}, where one samples many synthetic data points from $\mathbb{P}(X_{n+1:n+M}|x_{1:n})$ for some large $M$, and then for each synthetic sample computes an estimate $\hat{\theta}=f(x_{1:n+M})$, where $f$ is an appropriately chosen estimator which crucially uses both the original and the synthetic data. The resulting empirical distribution of $\hat{\theta}$ is an approximation to the so-called \textit{Martingale Posterior} \citep{fong2021martingale}; see Subsection~\ref{sec:pred_resampl}. Crucially, under suitable conditions, this posterior is fully equivalent to the one obtained with the usual application of Bayes' rule  starting with a likelihood and a prior. %, which corresponds to $\mathbb{P}(\theta|y_{1:n})$ as $m \to \infty$.  

What conditions the predictive rule needs to satisfy to get a valid martingale posterior is one of the key research questions. Given a filtration $\mathcal{G}=(\mathcal{G}_n)_{n \geq 1}$ (for example, the natural filtration generated by $(X_{n})_{n \geq 1}$), \cite{fong2021martingale} assume that the predictive distributions are measure-valued martingales, that is $(\mathbb{P}(X_{n+1} \in A |\mathcal{G}_n))_{n \geq 0}$ is a martingale for all Borel sets $A$. %, \textit{i.e.}, 
%\begin{equation*}
%    \mathbb{E} (\mathbb{P}(X_{n+2} \in A |\mathcal{G}_{n+1})|\mathcal{G}_{n}) - \mathbb{P}(X_{n+1} \in A |\mathcal{G}_{n})=0.
%\end{equation*}
The same condition appears also in \cite{fortini2020quasi}. The martingale assumption has many desirable implications.  It represents a departure from the assumption of exchangeability on $(X_n)_{n\geq 1}$ to the more general class of \textit{Conditionally Identically Distributed} (c.i.d.) sequences \citep{kal88,berti2004limit}. This enlarges the class of predictive distributions and models that can be considered, moving away from exchangeable Bayesian predictive schemes \citep[e.g.][]{fortini2020quasi}. Such a departure comes at no substantial cost: while the convergence of predictive distributions for exchangeable sequences is given by de Finetti's representation theorem, for c.i.d. sequences the random sequence  $(\mathbb{P}(X_{n+1} \in \cdot |\mathcal{G}_n))_{n \geq 0}$ converges weakly almost surely (a.s.) by martingale arguments \citep[see][]{ald85}. Crucially,  if $(X_n)_{n\geq 1}$ is c.i.d., it is also \textit{Asymptotically Exchangeable}, meaning that it behaves in the limit as an exchangeable sequence \citep{ald85}. This property justifies the validity of the martingale posterior as an approximation for $\mathbb{P}(\theta|x_{1:n})$. %In specific instances, \cite{fortini2020quasi} show that it is even possible to characterize what these asymptotic models and prior are.
What is appealing is that there is no need to specify a model and a prior or even assume exchangeability.

The martingale assumption has other desirable features; see \cite{fong2021martingale}. %It is an elegant framework that links Bayesian predictive inference to applications of martingales developed by \cite{doob1949application}. In addition, \cite{fong2021martingale} use the preservation of expectation of martingales to justify the coherence of the procedure \citep[in the sense of][]{robins2000conditioning}.
Nevertheless, what matters is the almost sure weak convergence of $(\mathbb{P}(X_{n+1} \in \cdot |\mathcal{G}_n))_{n \geq 0}$ and the resulting asymptotic exchangeability of $(X_n)_{n\geq 1}$.  %This emerges in subsequent works and %is stated clearly by \cite{glpr26a}, who argue that martingales are mostly useful for the mathematical argument. % of proving convergence of a general class of random measures. 
It is becoming clear that the martingale assumption could be limiting, and the literature now includes many examples that violate this assumption; for example, the parametric Bayesian bootstrap \citep{holmes2023statistical}, an online logistic regression \citep{fortini2024exchangeability}, and predictive schemes driven by sample moments \citep{glpr26a}. Several algorithms in learning theory and machine learning (such as kernel estimators), if considered as predictive rules, do not satisfy the martingale property either. %The mathematical convenience of working with martingales is lost and the asymptotic exchangeability of the resulting sequence $(X_n)_{n\geq 1}$ needs to be established for every new proposal. This is at times not straightforward, as proven by the involved mathematical arguments in \cite{fong2024asymptotics,glpr26a}. The technical condition we propose in this paper should  facilitate this and covers several examples. 

%Even more importantly, the current framework is failing at exploiting advanced related fields of research, such as learning theory and machine learning, whose primary goal is to develop algorithms for prediction. As we mentioned above, Bayesian predictive inference offers the potential of using these algorithms for inference. However, several of the algorithms we checked (\textit{e.g.} kernel estimators) do not satisfy the martingale property. This seems a missed opportunity given the great development in these areas in the last decades, not only with regards to the use of these algorithms for prediction but recently also for inference; see, Causal Machine Learning  \citep{chernozhukov2024applied}, and prediction-powered inference \citep{angelopoulos2023prediction}.

%At the same This is not surprising given that many of the proposals currently available seems to be inherently tied to the predictive distributions of a Bayesian model built under exchangeability or at least to Bayesian update; for example, many of the algorithms discussed in the aforementioned literature are related to Dirichlet Process mixtures. All such sequences of predictive

\subsection{Preview of Contributions}

In this paper, %we propose a new condition to determine whether a sequence of one-step ahead predictive distributions is suitable for predictive resampling and develop tools to verify it. In particular, 
we relax the martingale assumption on predictive distributions, by assuming instead that  the sequence of one-step ahead predictive distributions $(\mathbb{P}(X_{n+1} \in A |\mathcal{G}_n))_{n \geq 0}$ forms an almost supermartingale, \citep{robbins1971convergence} (see Appendix \ref{sec:app.almostsupermartingales} 
for background), for every set $A$. %satisfies for all $n$ and Borel sets $A$
%\begin{equation}\label{eq:pred_contribution}
%	\mid \mathbb{E} (\mathbb{P}(X_{n+2} \in A |\mathcal{G}_{n+1})|\mathcal{G}_{n}) - \mathbb{P}(X_{n+1} \in A |\mathcal{G}_{n})\mid \, \leq \xi_n \,\,\,\ \text{a.s.},
%\end{equation}
%where $(\xi_n)_{n \geq 0}$ is an appropriately defined sequence $\xi_n \to 0$ a.s..
Informally, it means that the sequence is not a martingale but converges to one asymptotically. %The idea builds on existing generalizations of martingales, such as almost supermartingales , mixingales \citep{mcleish1975invariance}, and quasi-martingales \citep{fisk1965quasi}. 
The most immediate consequence of the departure from martingales is that the resulting sequence $(X_n)_{n \geq 1}$ is not necessarily c.i.d. A similar departure from the martingale property is studied in the concurrent work of \cite{lpr26}.

In Section~\ref{sec:acid}, we define \textit{Almost Conditionally Identically Distributed} (a.c.i.d.) sequences, which generalize both c.i.d. and exchangeable sequences, and discuss a few characterizations. % of a.c.i.d. sequences that offer sufficient conditions on $(\mathbb{P}(X_{n+1} \in \cdot |\mathcal{G}_n))_{n \geq 0}$. 
It is then shown in Theorem~\ref{thm:asym.exch} that, under mild conditions,  an a.c.i.d. sequence is asymptotically exchangeable. This is done by proving that any sequence of random measures that is an almost supermartingale converges weakly a.s. This result generalizes similar results for martingales \citep{ald85}, and may be of general interest beyond the scope of this paper.  An implication is that any a.c.i.d. sequence satisfying these mild conditions can be used in predictive inference.%proving that the corresponding sequence of predictive distributions converges weakly a.s. We deem this result of general interest, beyond the current framework, as 
%  The proofs build on ideas of the aforementioned generalizations of martingales. 

Section~\ref{sec:examples} is devoted to illustrations showcasing the benefits of doing Bayesian predictive inference with a.c.i.d. sequences. %under this new assumption. 
We consider predictive rules built using old ideas in Statistics and Machine Learning, and new ones. 
We actually consider the former class more interesting to illustrate how Bayesian predictive inference can incorporate ideas from fields centered around prediction into a formal inference framework. For example, we show that a sequence of predictive distributions defined by kernel estimates generates an a.c.i.d. sequence. %, both in the case of density estimation and regression using classical kernels, such as Gaussian or Uniform, %Despite known shortcomings of kernel estimators, we consider this result notable, in 
A similar result holds for an online version of the Random Forest, exploiting the link between kernel methods and Random Forest \citep{breiman2000some}.
Lastly, another example is the Bayesian parametric bootstrap \citep{holmes2023statistical,fong2024asymptotics}, which satisfies the property under certain regularity conditions. %-- a recent work by \cite{fong2024asymptotics} studies this predictive rule but the focus is very different, we will elaborate on this.

%The goal of the illustrations included in the paper is not to advocate in favor of one predictive resampling scheme versus another one. As \cite{glpr26a} suggests, this will likely depend on specific problems and the available data. Rather, we want to show that working under \eqref{eq:pred_contribution} could be useful when evaluating whether a predictive rule is suitable for predictive resampling. For example, it is straightforward to show that the sequence in Example~\ref{ex:gauss_intro} satisfies \eqref{eq:pred_contribution}, and so it defines an asymptotically exchangeable sequence, while \cite{glpr26a} require a new strong law of large number for dependent non identically distributed random variables -- even though their setting is more general. 

Although the illustrations in this paper primarily serve to demonstrate examples of a.c.i.d. schemes, many of them are already of independent interest as predictive methods. For instance, kernel-based predictive schemes offer notable advantages: they do not require evaluation on a numerical grid and can be less sensitive to the ordering of data, in contrast to copula-based methods \citep{hahn18,fong2021martingale} and predictive recursion approaches \citep{new98,fortini2020quasi}. In addition, framing kernels as a.c.i.d. procedures may be useful for interpreting these algorithms in a proper statistical framework, as done by \cite{fortini2020quasi} for Newton's predictive recursion. Informally, our result suggests that there is a statistical model underlying kernel methods, thus addressing a question posed by \cite{west1991kernel}. In Section~\ref{sec:simulations}, we explore some numerical properties of these schemes, while leaving a more comprehensive investigation of kernel methods for predictive inference to future work. Section~\ref{sec:conclusion} provides concluding remarks.\\ \hspace{0.2cm}
 
%Similarly, this can be done for a novel class of copula-based predictive algorithms studied by \cite{hahn18,cappello24recursive}. We don't know if such sequences are martingales, but it is clear that checking \eqref{eq:pred_contribution} is mathematically less cumbersome. 

%\subsection*{Notation}
\noindent \textbf{Notation.} In the rest of the paper, all random variables (r.v.s) are assumed to be defined on a common probability space $(\Omega,\mathcal{A},\mathbb{P})$, and $\mathcal{G}=(\mathcal{G}_{n})_{n\geq 0}$ denotes a filtration on it, with $\mathcal{G}_{0}=\{\emptyset,E\}$. $(X_{n})_{n\geq 1}$ denotes a sequence of random variables taking values on a Polish space $(E,\mathcal{E})$. %$\mathcal{M}(E)$ denotes the space of probability measures on $E$. 
For most examples, $(E,\mathcal{E})$ can be taken as $(\mathbb{R}^{p},\mathcal{B}(\mathbb{R}^{p}))$. $X_{1:n}$ is used as shorthand notation for the vector $(X_{1},\ldots,X_{n})$. %When needed in Section \ref{sec:examples}, $\mathcal{G}^{X}=(\mathcal{G}^{X}_{n})_{n\geq 0}$ is used to highlight that the filtration generated by $X$ is considered, i.e. $\mathcal{G}^{X}_{n} = \sigma(X_{1:n})$ for all $n\geq 0$. 
$\stackrel{\text{w}}{\to}$, $\stackrel{\text{d}}{\to}$ and $\stackrel{\text{a.s.}}{\to}$ stand for weak convergence of probability measures, convergence in distribution and almost sure convergence of r.v.s, respectively. $\text{TV}(\alpha_{1},\alpha_{2})= \sup_{A\in \mathcal{E}}|\alpha_{1}(A)-\alpha_{2}(A)|$ denotes the total variation distance between the probability measures $\alpha_{1}$ and $\alpha_{2}$.  %Finally, $\alpha_{n,m}(\omega,A) := \mathbb{P}\left( X_{n+m}\in A | \mathcal{G}_{n} \right)(\omega)$ denotes the $m$-steps ahead predictive distribution, given $\mathcal{G}_{n}$. In $\alpha_{n,m}(\omega,A)$, we will often not highlight the dependence on $\omega$, and simply write $\alpha_{n,m}(A)$. Also, when the $m$ index is missing, it is implicitly assumed to be equal to $1$, i.e. $\alpha_{n}(A):= \alpha_{n,1}(A)$, the $1$-step ahead predictive distribution.
Finally, $\alpha_{n}(\omega,A) := \mathbb{P}\left( X_{n+1}\in A | \mathcal{G}_{n} \right)(\omega)$ denotes the one-step ahead predictive distribution, given $\mathcal{G}_{n}$. In $\alpha_{n}(\omega,A)$, we will often not highlight the dependence on $\omega$, and simply write $\alpha_{n}(A)$, and write $(\alpha_n)_{n \geq 0}$ for the random sequence $(\alpha_n(\cdot,\cdot))_{n \geq 0}$.

\section{Preliminaries}

In this section, we provide some background material. Specifically, Subsection~\ref{sec:cid} reviews two definitions of c.i.d. r.v.s, which will be both generalized to the a.c.i.d. case in Section~\ref{sec:acid}. Subsection~\ref{sec:pred_resampl} reviews the Predictive Resampling algorithm, whose application to a.c.i.d. r.v.s will be then illustrated in Section~\ref{sec:simulations}.

\subsection{Conditionally Identically Distributed Sequences} \label{sec:cid}

Informally, a sequence of r.v.s is c.i.d. if, at any point in time, given the current history of the process $\mathcal{G}_{n}$, future random variables are identically distributed. This form of probabilistic dependence has been thoroughly studied in \cite{kal88}
and \cite{berti2004limit} from a theoretical point of view, and important examples of c.i.d. r.v.s were proposed in \cite{bcl10}, \cite{acblg14}, \cite{czgvl19} and \cite{csz23}. C.i.d. r.v.s can be formally defined as follows, 

 \begin{definition}[C.i.d., Law] \label{def:cid.law}
    A sequence of $\mathcal{G}$-adapted r.v.s. $(X_{n})_{n\geq 1}$ is  \textit{Conditionally Identically Distributed (c.i.d.)} with respect to $\mathcal{G}$ if for all $k>n\geq 0$ and for all $A \in \mathcal{E}$
\begin{equation} \label{eq:cid}
\mathbb{P}\left( X_{k} \in A  |\mathcal{G}_{n}\right) = \mathbb{P} \left( X_{n+1} \in A | \mathcal{G}_{n}\right)  \ \ \ \ \ \ \ \  \text{a.s.}. 
\end{equation}
\end{definition}

In Definition \ref{def:cid.law}, it is easy to check that it is actually enough to require that $\mathbb{P}\left( X_{n+2} \in A  |\mathcal{G}_{n}\right) = \mathbb{P} \left( X_{n+1} \in A | \mathcal{G}_{n}\right)$ for all $n\geq 0$ and $A\in \mathcal{E}$, that is $X_{n+1}$ and $X_{n+2}$ are identically distributed, conditionally on the current history $\mathcal{G}_{n}$. Moreover, using usual approximations with simple functions, it is possible to restate Definition~\ref{def:cid.law} with condition~\eqref{eq:cid} holding for expectations over all integrable functions $f:E\to \mathbb{R}$, that is, $\mathbb{E}\left( f(X_{k})|\mathcal{G}_{n}\right) = \mathbb{E} \left( f(X_{n+1}) | \mathcal{G}_{n}\right)$, instead of just for indicator functions, $f(x)=\mathbb{I}(x\in A)$. Another, possibly less obvious, equivalent definition of c.i.d. r.v.s is through the martingale property of their predictive distributions.%, which can be formulated as follows 
\begin{definition}[C.i.d., Predictive]\label{def:cid.pred}
 A sequence of $\mathcal{G}$-adapted r.v.s. $(X_{n})_{n\geq 1}$ is \textit{c.i.d.} with respect to $\mathcal{G}$ if the sequence
 %   \begin{equation} \label{eq:cid.mart}
   $ \left( \mathbb{P}\left(X_{n+1} \in A |\mathcal{G}_{n} \right) \right)_{n\geq 0}$
%\end{equation}
forms a $\mathcal{G}$-martingale, for each $A \in \mathcal{E}$.
\end{definition}

Informally, we will often write that a sequence of predictive distributions is c.i.d. to mean that the corresponding sequence $(X_{n})_{n\geq 1}$ is c.i.d.
%DO WE NEED ALSO ASYMPTOTIC EXCHANGEABILITY?

\subsection{Predictive Resampling}\label{sec:pred_resampl}

Predictive resampling generalizes the Bayesian bootstrap \citep{rubin1981bayesian} to arbitrary predictive schemes. As in the Introduction, let $\theta$ denote an estimand of interest, which may be either finite- or infinite-dimensional, and let $f$ be an estimator such that, given a sample $x_{1:n}$, an estimate for $\theta$ is given by $\hat{\theta}_n = f(x_{1:n})$. Both $\theta$ and $f$ are arbitrary. \cite{fong2021martingale} define the finite martingale posterior as follows:

\begin{definition} (Finite Martingale Posterior)
The finite martingale posterior is defined as
\begin{equation}\label{eq:martingale_posterior} \mathbb{P}(\hat{\theta}_{n+M} \in \cdot \mid \mathcal{G}_n) = \int \mathbb{I}(f(x_{1:n+M}) \in \cdot ) \mathbb{P} (\di x_{n+1:n+M} \mid \mathcal{G}_n). \end{equation} \end{definition}

The martingale posterior is obtained as the limit $M\to \infty$. Under the assumption that $(X_{n})_{n \geq 1}$ is exchangeable and $(\mathbb{P}(X_{n+1} \in \cdot \mid \mathcal{G}_{n}))_{n \geq 0}$ the corresponding sequence of predictive distributions, one can show that the martingale posterior coincides with the standard Bayesian posterior. This argument traces back to \cite{doob1949application}. If $(X_{n})_{n \geq 1}$ is c.i.d., such equivalence may not hold in finite samples but only asymptotically. Informally, this happens because $(X_{n})_{n \geq 1}$ behaves in the limit as an exchangeable sequence. We will  return to this point in the next section when we define asymptotic exchangeability.

In Bayesian predictive inference, rather than specifying a model and prior, one directly defines the sequence $(\mathbb{P}(X_{n+1} \in \cdot \mid \mathcal{G}_{n}))_{n \geq 0}$ and uses it to approximate the martingale posterior. %, allowing us to bypass the assumption of exchangeability. Early works \citep{fortini2020quasi,fong2021martingale} assume that $(\mathbb{P}(X_{n+1} \in \cdot \mid \mathcal{G}_{n}))_{n \geq 1}$ forms a martingale and so the underlying sequence of random variables is c.i.d.. The specification of the predictive sequence fully determines the martingale posterior. In practice, one computes the finite martingale posterior, which we will refer to simply as the martingale posterior.
Predictive resampling is a Monte Carlo method to do that. Specifically, one draws $B$ samples of size $M$, $X^{(i)}_{n+1:n+M} \sim \mathbb{P}(X_{n+1:n+M} \in \cdot \mid \mathcal{G}_n)$, for $i=1,\ldots,B$, and computes $(\hat{\theta}^{(i)}_{n+M})_{1:B}$, where $\hat{\theta}^{(i)}_{n+M}=f(x_{1:n+M}^{(i)})$. The empirical distribution of $(\hat{\theta}^{(i)}_{n+M})_{1:B}$ then serves as an approximation of $\mathbb{P}(\hat{\theta}_{n+M} \in \cdot \mid \mathcal{G}_n)$. Sampling $X_{n+1:n+M}$ can be performed sequentially, one step at a time, using the one-step ahead predictive. If $(\mathbb{P}(X_{n+1} \in \cdot \mid \mathcal{G}_{n}))_{n \geq 0}$ is the empirical distribution, this scheme coincides with the Bayesian bootstrap. The key novelty is that we can move beyond the discreteness of the empirical distribution and employ continuous updates; for instance, using the Gaussian copula algorithm of \cite{hahn18}. The same numerical scheme can be used to scrutinize the underlying model and prior; see the application to Newton’s algorithm in \cite{fortini2020quasi} for details.

\section{Almost Conditionally Identically Distributed Random Variables } \label{sec:acid}

%Almost Conditionally Identically Distributed (a.c.i.d.) r.v.s are %generalizations of  c.i.d. r.v.s, %which themselves can be thought of as generalizations of exchangeable sequences. %Therefore, after introducing the necessary notation, we will start by briefly recalling the definition of c.i.d. r.v.s, and then introduce a.c.i.d.s r.v.s and present some of their properties.  
%A.c.i.d. r.v.s naturally generalize 
We relax Definition \ref{def:cid.pred} of c.i.d. r.v.s by allowing the sequence of predictive distributions 
to be an almost supermartingale \citep{robbins1971convergence}. We remark that a similar generalization has recently been considered in the concurrent work, \cite{lpr26}, in which the predictive distribution is allowed to be a quasi-martingale \citep{fisk1965quasi}.  Almost Conditionally Identically Distributed (a.c.i.d.) r.v.s are defined as follows,  %In Section \ref{sec:examples}, we will see that this generalization is indeed quite natural and includes common predictive schemes. %Definition \ref{def:cid.pred} generalizes to the a.c.i.d. case as follows, 

%\red{$\xi_n$ has to change to put the conditions on having it random. then we need this conditions on teh series of expectations.}

\begin{definition}[A.c.i.d., Predictive] \label{def:acid.pred}
    A sequence of $\mathcal{G}$-adapted r.v.s. $(X_{n})_{n\geq 1}$ is  $(\xi_{n})_{n\geq 0}$-a.c.i.d. with respect to $\mathcal{G}$  if the sequence
%\begin{equation*} 
  $  \left( \mathbb{P}\left(X_{n+1} \in A |\mathcal{G}_{n} \right) \right)_{n\geq 0}$
%\end{equation*}
    forms an almost $\mathcal{G}$-supermartingale  for all $A\in \mathcal{E}$, that is, for all $n \geq 0$ and $A\in \mathcal{E}$, 
\begin{align} \label{eq:acid.pred}
    \mathbb{E}\left( \mathbb{P}(X_{n+2} \in A |\mathcal{G}_{n+1}) | \mathcal{G}_{n} \right) & \leq \mathbb{P}(X_{n+1} \in A |\mathcal{G}_{n}) +\xi_{n}  \ \ \ \ \ \ \ \  \text{a.s.,}
\end{align}
 where $(\xi_{n})_{n\geq 0}$ is a sequence of $\mathcal{G}$-adapted non-negative r.v.s.
\end{definition}

When $(\xi_{n})_{n\geq 0}$ and/or $\mathcal{G}$ are clear from the context, we will simply write that the sequence $(X_{n})_{n\geq 1}$ is a.c.i.d., instead of $(\xi_{n})_{n\geq 0}$-a.c.i.d. with respect to $\mathcal{G}$. Informally, we will often write that a sequence of predictive distributions is a.c.i.d. to mean that the corresponding $(X_{n})_{n\geq 1}$ is a.c.i.d.

Notice that, because $\mathcal{G}_{n}\subset \mathcal{G}_{n+1}$, formula \eqref{eq:acid.pred} reduces to $\mathbb{P}(X_{n+2} \in A |\mathcal{G}_{n})  \leq \mathbb{P}(X_{n+1} \in A |\mathcal{G}_{n}) +\xi_{n}$. Since condition \eqref{eq:acid.pred} must hold for all $A\in \mathcal{E}$, by applying it to $A^{c}$, we obtain that it must also hold that $\mathbb{P}(X_{n+2} \in A |\mathcal{G}_{n})  \geq \mathbb{P}(X_{n+1} \in A |\mathcal{G}_{n}) - \xi_{n}$ for all $A\in \mathcal{E}$. %Hence, in an a.c.i.d. sequence, the sequence of predictive distributions, $\left( \mathbb{P}\left(X_{n+1} \in A |\mathcal{G}_{n} \right) \right)_{n\geq 0}$, needs to be both an almost supermartingale and almost submartingale. 
Finally, using condition \eqref{eq:acid.pred} on both $A$ and $A^{c}$, it follows that, for $A\in \mathcal{E}$,  $| \mathbb{P}(X_{n+2} \in A |\mathcal{G}_{n})  - \mathbb{P}(X_{n+1} \in A |\mathcal{G}_{n}) | \leq \xi_{n}$, which then provides a natural generalization of Definition \ref{def:cid.law} to the a.c.i.d. case, in terms of total variation distance between $\mathbb{P}(X_{n+2} \in \cdot |\mathcal{G}_{n})$ and $\mathbb{P}(X_{n+1} \in \cdot |\mathcal{G}_{n})$.

\begin{definition}[A.c.i.d., Law] \label{def:acid.law}
    A sequence of $\mathcal{G}$-adapted r.v.s. $(X_{n})_{n\geq 1}$ is  $(\xi_{n})_{n\geq 0}$-a.c.i.d.  with respect to $\mathcal{G}$ if, for all $n\geq 0$, a.s.
    \begin{align*}
     \text{TV}\left(\mathbb{P}(X_{n+2} \in \cdot |\mathcal{G}_{n}),  \mathbb{P}(X_{n+1} \in \cdot |\mathcal{G}_{n}) \right) % =\sup_{A\in\mathcal{A}}  | \mathbb{P}(X_{n+2} \in A |\mathcal{G}_{n})  - \mathbb{P}(X_{n+1} \in A |\mathcal{G}_{n}) | 
     \leq \xi_{n}.
    \end{align*}
\end{definition}
%This characterization is useful also because it  allows one to use bounds on other distances or divergences between probability measures, which might be more easily manageable in the specific example, that upper bound the total variation distance in order to prove that the sequence is a.c.i.d..

The same argument that led to Definition~\ref{def:acid.law} justifies working with almost supermartingales:  requiring the predictive distributions to form either a supermartingale or a submartingale cannot generalize the definition of c.i.d. r.v.s, because in that case they would be both, and thus a martingale. Definition~\ref{def:acid.law} also highlights the difference with the condition of \cite{lpr26}, where the upper bound can depend on a given Borel set.

As already mentioned in the Introduction, the main requirement for a sequence of predictive distributions to apply the posterior resampling approach of \cite{fong2021martingale} is \textit{asymptotic exchangeability}. Theorem~\ref{thm:asym.exch} provides sufficient conditions on $(\xi_{n})_{n\geq 0}$ that ensure an a.c.i.d. sequence is asymptotically exchangeable. This is achieved by showing that any sequence of random measures forming an almost supermartingale converges weakly a.s., extending analogous results for martingales \citep[e.g. see][]{ald85,berti2004limit}. %, and may be of general interest beyond the scope of this paper. %In Section~\ref{sec:examples}, we show that these conditions are easily satisfied by some natural classes of predictive distributions. 
Given a.s. weak convergence, asymptotic exchangeability follows from Lemma 8.2 of \cite{ald85}.
The proof is in Section~\ref{sec:Proofs} of the Appendix.

\begin{theorem} \label{thm:asym.exch}
    Let $(X_{n})_{n\geq 1}$ be a $(\xi_{n})_{n\geq 0}$-a.c.i.d. sequence  and $\seq{\alpha_n(\cdot,\cdot)}$ the corresponding sequence of 1-step ahead predictive distributions. If  the sequence $(\xi_{n})_{n\geq 0}$  satisfies\begin{equation} \label{eq:xi.summ}
       \mathbb{E}(\sum_{n=0}^\infty  \xi_n)  < \infty %\ \ \ \ \ \\text{a.s.},
    \end{equation} then there exists a random probability measure $\alpha:\Omega \times \mathcal{A} \to [0,1]$ such that  \begin{equation*}
      \mathbb{P}(\omega\in \Omega: \alpha_n(\omega,\cdot)\stackrel{\text{w}}{\to}  \alpha(\omega,\cdot))=1.
      \end{equation*}
Moreover, $(X_{n})_{n\geq 1}$ is asymptotically exchangeable, that is, as $n\to\infty$,
\begin{equation*}
    (X_{n},X_{n+1},\ldots) \stackrel{\text{d}}{\to} (Z_{1},Z_{2},\ldots),
\end{equation*}
where $Z=(Z_{1},Z_{2},\ldots)$ is an exchangeable sequence directed by $\alpha$.
\end{theorem}

In most applications, the sequence $(\xi_{n})_{n\geq 0}$ is deterministic, in which case condition \eqref{eq:xi.summ} reduces to the summability condition $\sum_{n=0}^{\infty} \xi_{n} < \infty$. Informally, Theorem~\ref{thm:asym.exch} states that the distribution of $(X_n)_{n\geq N}$  can be approximated for large $N$ by that of an exchangeable sequence whose predictive distributions converge to the same random limit $\alpha(\cdot)$. %As such, thanks to de Finetti's theorem, there must exist a (random) marginal distribution $\alpha(\cdot)$, conditionally to which observations become i.i.d..
From a Bayesian perspective, the random measure $\alpha(\cdot)$ represents the unknown statistical model, and its law corresponds to the prior distribution. Given a sample, whose information is summarized by the sigma-algebra $\mathcal{G}_n$, $\mathbb{P}(\alpha(\cdot)\in \cdot | \mathcal{G}_n)$ is the posterior distribution, and  $\mathbb{E}[\alpha(\cdot) | \mathcal{G}_n]$ the predictive distribution. %, where the expectation is taken with respect to the posterior. 
Theorem~\ref{thm:asym.exch} does not provide a characterization of the asymptotic model and its properties, such as whether it is absolutely continuous. Such characterizations are available for martingale random measures through the tools developed in \cite{berti2013exchangeable}, who also link the absolute continuity of the random limit to convergence in total variation, but remain unavailable for almost supermartingale random measures. We return to this point in the discussion.

If the sequence $(\alpha_n)_{n \geq 0}$ forms a martingale (Definition~\ref{def:cid.pred}), standard properties of martingales imply $\alpha_n(\cdot)=\mathbb{E}[\alpha(\cdot) | \mathcal{G}_n]$. That is, the predictive distribution is an unbiased estimator of the limiting measure; in this context, the statistical model. Notably, this result depends solely on the martingale property and does not require exchangeability; see the notion of predictive coherence in \cite{fong2021martingale}. However, under the a.c.i.d. condition (Definition~\ref{def:acid.pred}), the sequence $(\alpha_n)_{n \geq 0}$ fails to be a martingale, and in general, $\alpha_n(\cdot) \neq \mathbb{E}[\alpha(\cdot)| \mathcal{G}_n]$. This naturally raises the question of whether this bias vanishes or can be reduced. The following corollary to Theorem~\ref{thm:asym.exch} provides a partial answer.

\begin{corollary}\label{corollary1}
Let $(X_{n})_{n\geq 1
}$ be an (a.c.i.d.) asymptotically exchangeable sequence with directing measure $\alpha(\cdot)$. Then, for $M\to \infty$.
\begin{equation} \label{eq:conv.ind}
    \mathbb{P}\left(X_{n+M}\in \cdot  |\mathcal{G}_{n}\right) \stackrel{\text{w}}{\to} \mathbb{E}\left(\alpha(\cdot) | \mathcal{G}_{n} \right) \ \ \ \ \ \text{a.s.}.
\end{equation}
\end{corollary}

In words, Corollary~\ref{corollary1} states that, if we are predicting future observations using the sequence of predictive distributions of an asymptotically exchangeable sequence, the $M$-step ahead predictive distribution will become closer (in a distance metrizing weak convergence) to corresponding predictive distribution obtained using the Bayesian posterior distribution as $M$ increases. %If one deems the unbiasdness of the predictive distribution important, Corollary~\ref{corollary1} suggests a recipe to work $(\alpha_n)_{n\geq 1}$ 

\section{Examples of  a.c.i.d. Sequences} \label{sec:examples}

Throughout the section, $\mathcal{G}$ will be the natural filtration%with $\mathcal{G}_{0}=\{\emptyset ,\mathbb{R}\}$
, $(E,\mathcal{E})=(\mathbb{R}^{p},\mathcal{B}(\mathbb{R}^{p}))$, and $(\eta_{n})_{n \geq 1}$ an appropriately chosen deterministic sequence such that $\eta_n \to 0$. When possible, the predictive distributions $(\alpha_n)_{n \geq 0}$ will be specified recursively. This is not a requirement, but is convenient for implementing predictive resampling.  Examples in Sections~\ref{sec:gauss_pred} and \ref{sec:param_boot} are considered by \cite{holmes2023statistical} and related ones are studied mathematically by \cite{glpr26a} and \cite{fong2024asymptotics}. Here, the goal is to illustrate the benefit of working with our new technical condition. Examples in Sections~\ref{sec:kernel} and \ref{sec:kernel_regr} are new, and will define usable new predictive resampling schemes. 

\subsection{Gaussian Sample Mean Driven Sequence}\label{sec:gauss_pred}

%We generalize Example~\ref{ex:gauss_intro} to the multivariate setting. Specifically, l
Let $\text{N}(\mu,\Sigma)$ denote a $p$-variate Gaussian distribution with mean $\mu$ and covariance $\Sigma$, and $\phi(x;\mu,\Sigma)$ the corresponding density function. We can specify the generative process of a sequence $(X_n)_{n\geq 1}$ as follows: 
choose $\hat{\theta}_{0} \in \mathbb{R}^p$ and $\Sigma$ a valid variance-covariance matrix, and %set $\mathcal{G}_{0}=\{\emptyset ,\mathbb{R}\}$, then 
sample inductively 
%\begin{equation*}
    $X_{n+1}|\mathcal{G}_{n} \sim \text{N}(\hat{\theta}_{n},\Sigma),$
%\end{equation*}
where %$\mathcal{G}_{n}=\sigma(X_{1:n})$ and 
$\hat{\theta}_{n}$ is updated recursively as $\hat{\theta}_{n} = (1-\eta_{n}) \hat{\theta}_{n-1} + \eta_{n} X_{n}$. %, for an appropriately chosen deterministic sequence $(\eta_{n})_{n \geq 1}$.
 $(\text{N}(\hat{\theta}_{n},\Sigma))_{n \geq 0}$ is a sequence of  parametric predictive random measures, where the randomness is driven by $\hat{\theta}_n$. When $\eta_n=\frac{1}{n}$, $\hat{\theta}_n$ corresponds to the sample mean (the MLE). %Such sequence has been proposed by \cite{holmes2023statistical}, and studied mathematically by \cite{glpr26a}. \\
 
It is easy to show that $(X_n)_{n \geq 1}$ is in general not c.i.d.; instead, it is an a.c.i.d. sequence. Indeed, since $\alpha_{n}(A) = \int_{A} \phi(x;\hat{\theta}_{n}, \Sigma) \di x$, it holds that
%\begin{align*}
%    \mathbb{E}\left[\alpha_{n+1}(A) | \mathcal{G}_{n}\right]
 %   & = \mathbb{E}\left[ \int_{A} \phi(x;\hat{\theta}_{n+1}, \Sigma ) \di x | \mathcal{G}_{n}\right] \\
 %   & = \int_{\mathbb{R}^p} \int_{A} \phi(x;(1-\eta_{n+1}) \hat{\theta}_{n} + \eta_{n+1} x_{n+1}, \Sigma ) \phi(x_{n+1};\hat{\theta}_{n},\Sigma ) \di x \di x_{n+1} \\
 %    & =  \int_A \phi(x;\hat{\theta}_{n-1},(1+\eta_n^2) \Sigma) \di x,
%\end{align*}
\begin{align*}
   \mathbb{E}\left(\alpha_{n+1}(A)  | \mathcal{G}_{n}\right)
     & = \int_{\mathbb{R}^p} \int_{A} \phi(x; \hat{\theta}_{n+1}, \Sigma )  \di x \phi(x_{n+1};\hat{\theta}_{n},\Sigma )  \di x_{n+1}   
   %  & =\int_{\mathbb{R}^p} \int_{A} \phi(x;(1-\eta_{n+1}) \hat{\theta}_{n} + \eta_{n+1} x_{n+1}, \Sigma ) \phi(x_{n+1};\hat{\theta}_{n},\Sigma ) \di x \di x_{n+1}   \\
     = \int_A \phi(x;\hat{\theta}_{n},(1+\eta_{n+1}^2) \Sigma) \di x,
\end{align*}
where the second equality follows from the recursion formula for $\hat{\theta}_{n+1}$, %while the  last one from changing the order of integration
and solving a convolution of Gaussian variables. To check Definition~\ref{def:acid.law}, we note that 
\begin{align*}
 \text{TV}\left(\mathbb{P}(X_{n+2} \in \cdot |\mathcal{G}_{n}), \mathbb{P}(X_{n+1} \in \cdot |\mathcal{G}_{n}) \right) & =  \text{TV}\left(\text{N}(\hat{\theta}_{n},(1+\eta_{n+1}^2) \Sigma), \text{N}(\hat{\theta}_{n},\Sigma) \right) \leq \frac{3}{2} p \eta_{n+1}^2,
\end{align*}
where the inequality follows from Theorem 1.1 of \cite{devroye2018total} (see Appendix Section~\ref{sec:background_tv_bounds}). Then, $(X_n)_{n\geq 1}$ is a.c.i.d. with parameter $\xi_n = \frac{3}{2} p \eta_{n+1}^2$. The fact that $\mathbb{P}(X_{n+2}\in\cdot|\mathcal{G}_{n})=\text{N}(\hat{\theta}_{n},(1+\eta_{n+1}^2) \Sigma)$ also implies that there is no $\eta_{n+1}$ that makes the sequence c.i.d., except for the degenerate case where $\eta_{n+1}=0$, which corresponds to %$i.i.d.$ sampling (or exchangeable if $\hat{\theta}_{0}$ is randomized) and
no parameter update. 

From Theorem~\ref{thm:asym.exch}, it suffices to assume that $\sum_{n=1}^\infty \eta_n^2 < \infty$ in order to prove that $(\text{N}(\hat{\theta}_{n},\Sigma))_{n \geq 0}$ converges weakly a.s. and that $(X_n)_{n\geq 1}$ is asymptotically exchangeable. The proof is straightforward, relying only on convolutions of Gaussian distributions and a standard bound on the total variation distance. It would be of interest to see if our argument applies to the more general sequence of \cite{glpr26a}, who study sequences driven by the sample mean and the sample variance and prove almost sure convergence in total variation, or even their follow-up work \citep{glpr26b}, where they consider a mixture of such sequences.

\subsection{Bayesian Parametric Bootstrap}\label{sec:param_boot}

Let $P_{\theta}$ denote a probability distribution parametrized by an unknown parameter $\theta \in \Theta$; $p(x;\theta)$ denote the corresponding density function (or probability mass function if the random variable is discrete), $p'(x;\theta)$ and $p''(x;\theta)$ the first and second derivatives with respect to $\theta$, $s(x;\theta):= \frac{\partial}{\partial \theta} \log p(x;\theta)$ the score function, and $\mathcal{I}(\theta)=\mathbb{E}(s^{2}(x;\theta))$ the Fisher information. Let $\hat{\theta}_{0}$ be an initialization, we inductively sample for $n\geq 1$
\begin{equation}\label{eq:param_sequence}
    X_{n}|\mathcal{G}_{n-1} \sim P_{\hat{\theta}_{n-1}},
\end{equation}
where $\hat{\theta}_{n}$ is updated recursively with Stochastic Gradient Descent (SGD)
\begin{equation}\label{eq:sgd_natural}
\hat{\theta}_{n}= \hat{\theta}_{n-1} + \eta_n Z(x_n, \hat{\theta}_{n-1}),
\end{equation}
where $Z(x_n, \hat{\theta}_{n-1}):=\mathcal{I}(\hat{\theta}_{n-1})^{-1} s(x_n, \hat{\theta}_{n-1})$ is the natural gradient update, and $\eta_n$ the step size. %The parametric Bayesian boostrap has been discussed in the context of predictive resampling by \cite{holmes2023statistical} and then studied mathematically by \cite{fong2024asymptotics}. 
The use of the natural gradient follows the description of the parametric Bayesian bootstrap given by \cite{fong2024asymptotics}. %, who use SGD with the natural gradient. We will return to this difference with the one by \cite{holmes2023statistical} below. The key challenge, compared to a standard parametric setting, is that the limit of $\hat{\theta}_n$ (if it exists) is random. 

To state the main result, we make the following assumptions,

\begin{assumption}\label{ass:param_acid} Consider the data-generating process \eqref{eq:param_sequence}-\eqref{eq:sgd_natural} and assume
\begin{enumerate}[label=\alph*)]
\item  \hspace{0.2cm} $p$ is twice differentiable with respect to $\theta$.
\item \hspace{0.2cm} $\mathcal{I}(\theta)$ is continuous and there exists $\epsilon$ such that $\epsilon<\mathcal{I}(\theta)<\infty$ for all $\theta \in \Theta$. % for some $\epsilon>0$.
\item \hspace{0.2cm} The sequence $(Z(X_n,\hat{\theta}_{n-1}))_{n \geq 0}$ is uniformly integrable.
\item \hspace{0.2cm} There exists $C<\infty$ such that $\int p''(x;\theta) \di x <C$ for all $\theta \in \Theta$.
\end{enumerate}
\end{assumption}
Assumptions~\ref{ass:param_acid} a)–c) match those in \cite{fong2024asymptotics}.  a) and b) are standard. Assumption c) allows us to use Doob's martingale convergence theorem on $(\hat{\theta}_n)_{n \geq 1}$. This is a fairly strong assumption, which may be difficult to verify in practice. However, \cite{fong2024asymptotics} provide a sufficient condition that is easier to check in examples (see Appendix Proposition~\ref{prop:fongyiu}). We discuss alternatives below. Assumption~\ref{ass:param_acid} d) is used in a Taylor expansion of the density $p(x;\hat{\theta}_{n})$, which we employ to bound $\text{TV}(\mathbb{E}(P_{\hat{\theta}_{n}}|\mathcal{G}_{n-1}),P_{\hat{\theta}_{n-1}})$. It is reminiscent of classical conditions for asymptotic normality of maximum likelihood estimators, except that it is a condition on derivatives of the density and not of the log density as more commonly done (as in \cite{van2000asymptotic} Section 5.6). %It can be substituted by alternatives more clearly reminiscent of classical CLT conditions (as in \cite{van1996weak} Section 5.6). For example,
%d') There exists a function $\Psi(x)$ with finite expectation such that for all $x$, $\theta$, and $\nu$ in a ball of $\theta$
%\[
%\frac{|p''(x;\nu)|}{p(x;\theta)} < \Psi(x).
%\]
 Nevertheless, we will show that it is verifiable for some parametric models.

\begin{proposition}\label{prop:param_boot}Let $(X_n)_{n \geq 1}$ be sampled inductively as described in \eqref{eq:param_sequence}-\eqref{eq:sgd_natural}. Under Assumption~\ref{ass:param_acid} and $\sum_{n=1}^{\infty} \eta^2_n<\infty$, it holds that
\begin{itemize}
    \item \hspace{0.2cm} $(X_n)_{n \geq 1}$ is a.c.i.d., with %$\xi_{n}:=C \frac{\eta_{n+1}^2}{2\mathcal{I}(\hat{\theta}_n)}$,
   $\xi_{n}:=C \frac{\eta_{n+1}^2}{2\epsilon}$,
    where $C$ and $\epsilon$ are from Assumption~\ref{ass:param_acid} d) and b), respectively;
    %\item \hspace{0.2cm} $\sum_{n=0}^{\infty} \xi_n <\infty$ a.s.;
    \item \hspace{0.2cm} $(X_n)_{n \geq 1}$ is  asymptotically exchangeable.
\end{itemize} 
\end{proposition}

The sequence $(P_{\hat{\theta}_{n}})_{n \geq 0}$ is not a martingale, which means that the resulting sequence $(X_n)_{n \geq 1}$ is not c.i.d. However, Proposition~\ref{prop:param_boot} shows that $(X_n)_{n \geq 1}$ is a.c.i.d. The result is stated for a univariate parameter but it can be generalized to the multivariate setting. This result complements the work of \cite{fong2024asymptotics} because they study asymptotic properties for the sequence $(\hat{\theta}_n)_{n \geq 0}$ while we characterize the law of the observables. %Proposition 1 of \cite{fong2024asymptotics}, which we recall %in Proposition~\ref{prop:fongyiu}
%in the supplementary material, provides a condition that implies Assumption~\ref{ass:param_acid} c) and is easier to check in the examples.

\begin{example} (Gaussian, unknown mean, known variance) Let $P_{\theta} = \text{N}(\theta,\sigma^2)$. We have that $\mathcal{I}(\theta)=\frac{1}{\sigma^2}$ and $s(x;\theta) = \frac{(x-\theta)}{\sigma^2}$, and so $Z(x,\theta)=(x-\theta)$. In this case, Assumptions~\ref{ass:param_acid} a)-c) trivially hold. The second derivative is
\[
p''(x;\theta) =\phi(x;\theta,\sigma^2)\left[\frac{(x-\theta)^2}{\sigma^2} - \frac{1}{\sigma^2}\right],
\]
which implies that $\int p''(x;\theta) \di x =\frac{\sigma^{2}-1}{\sigma^{2}}<\infty$. This shows Assumption~\ref{ass:param_acid} d) holds. 
\end{example}

%\[
%\frac{\partial^2 f}{\partial (\sigma^2)^2} = f(x \mid \mu, \sigma^2) \left( -\frac{1}{4\sigma^4} + \frac{(x - \mu)^2}{2\sigma^6} + \frac{(x - \mu)^4}{4\sigma^8} \right).
%\]

\begin{example}\label{ex:studentt}(Student-t distribution)  Let $P_{\theta}$ be a location-scale Student-$t$ distribution with location $\theta$ and known scale $\tau$ and degrees of freedom $\nu$.%, with density
%\begin{equation} \label{eq:student0}
%p_{\theta}(x ) = \frac{\Gamma\left(\frac{\nu+1}{2}\right)}{\sqrt{\nu \pi} \tau \Gamma\left(\frac{\nu}{2}\right)} \left(1 + \frac{(x - \theta)^2}{\nu \tau^2} \right)^{-\frac{\nu+1}{2}}.
%\end{equation}
Setting \( \tau = 1 \), the update is given by \cite{fong2024asymptotics}: 
\[
Z(\theta,Y) = \frac{(\nu +3)(Y-\theta)}{\nu +(Y-\theta)^2},
\]
who also prove that Assumptions~\ref{ass:param_acid} a)-c) hold. The case is interesting because both $Z(\theta,Y)$ and $\mathcal{I}(\theta)$ are unbounded. 
%Differentiating the density  $p_{\theta}(x )$ yields 
%\[
%p''(x;\theta) = p_{\theta}(x )  \left[ \frac{\nu+1}{\nu} \left( 1 + \frac{(x - \theta)^2}{\nu} \right)^{-1} - \frac{2 (x - \theta)^2 (\nu+1)}{\nu^2} \left( 1 + \frac{(x - \theta)^2}{\nu} \right)^{-2} \right].
%\]
In Appendix Section~\ref{sec:proof_param_boot}, we show that $\int p''(x;\theta) \di x <C$, for some finite constant $C$, which implies that Assumption~\ref{ass:param_acid} d) also holds.
\end{example}

\begin{remark}  An alternative update for the parametric bootstrap is proposed in 
\cite{holmes2023statistical}. This version uses the standard gradient $Z'(x,\theta)=s(x,\theta)$, instead of the natural gradient.
We remark that Proposition~\ref{prop:param_boot} continues to hold with this update, provided Assumption~\ref{ass:param_acid}(c) is replaced by the assumption that the sequence $(\mathcal{I}(\hat{\theta}_n))_{n \geq 0}$ is uniformly integrable. An alternative assumption is also discussed below
\end{remark}

%\textcolor{red}{(POSSIAMO LEVARE QUESTA PARTE QUA SOTTO?)}
Assuming that $\Theta=[\underline{\theta},\bar{\theta}]$, for arbitrary finite $\underline{\theta},\bar{\theta}$, allows one to drop the uniform integrability assumption completely, both in the case of $Z(x,\theta)$ and $Z'(x,\theta)$. Define a constrained SGD sequence as follows

\begin{equation}\label{eq:sgd_constrained}
			\tilde{\theta}_n=\begin{cases}
				\underline{\theta} & \text{if } \,\, \hat{\theta}_n <\underline{\theta}\\
				\hat{\theta}_n & \text{if } \,\, \underline{\theta}\leq \hat{\theta}_n \leq \bar{\theta} \\
				 \bar{\theta} & \text{if } \,\, \hat{\theta}_n >\bar{\theta},
			\end{cases}
\end{equation}

$(\tilde{\theta}_n)_{n \geq 0}$  is not a martingale but it is still possible to prove that $(X_{n})_{n\geq 1}$ is a.c.i.d. and drop Assumption~\ref{ass:param_acid} c). 

\begin{proposition}\label{prop:param_boot_constrained}Let $(X_n)_{n \geq 1}$ be sampled inductively as described in \eqref{eq:param_sequence}-\eqref{eq:sgd_constrained}. Under Assumption~\ref{ass:param_acid} a), b), and d), %$|\underline{\theta}|<\infty$, $|\bar{\theta}|<\infty$, 
and $\sum_{n=1}^{\infty} \eta^2_n<\infty$, the statement of Proposition~\ref{prop:param_boot} holds with $\xi_{n}:=C \frac{\eta_{n+1}^2}{2\epsilon}$, where $\epsilon$ and $C$ are from Assumption~\ref{ass:param_acid} b) and d). 
\end{proposition}

%\subsection{Regression}
%- marginale (condizioni per cui sia acid)

\subsection{Kernel Estimators}\label{sec:kernel}

Kernel density estimators are popular nonparametric methods to estimate smooth density functions.  Let $K:\mathbb{R}\to \mathbb{R}$ be 
an integrable function satisfying $\int_{\mathbb{R}}K(u)\di u=1$ called \textit{kernel}, and $h=(h_{n})_{n\geq 0}$ be a sequence of parameters called \textit{bandwidth}. The pair $(K,h)$ is generally chosen by the user, with $K$ bounded, and $h$  such that $h_n \to 0$ at a prespecified rate; see \cite{silverman2018density} for an introduction, and \cite{wolverton1969asymptotically} for the sequential version of kernel estimators.

We use kernels to construct sequences of predictive distributions: choose an initial distribution $\alpha_0$, and a pair $(K,h)$, define recursively $\alpha_{n}(A)$%of predictive measures $\alpha_{n}(A)$%$:=\mathbb{P}(X_{n+1}\in A |\mathcal{G}_{n})$ 
 %as follows 
\begin{equation} \label{eq:kernel_Fn}
	\alpha_{n}(A)=\frac{n-1}{n}\alpha_{n-1}(A)+\frac{1}{n h_{n}} \int_{A} K\left(\frac{x-X_{n}}{h_{n}}\right) \di x,
\end{equation}
%where $\mathcal{G}_{n}=\sigma(X_{1:n})$. 
and $X_{n+1} \sim \alpha_n$. Since $\int \frac{1}{h_n} K(\frac{x}{h_n}) \di x =1$, we will denote the normalized kernel $K_n (x) := \frac{1}{h_n} K(\frac{x}{h_n})$ and let $\mu_n (A):= \int_A K_n (x) \di x$ be the corresponding measure. We remark that the predictive scheme \eqref{eq:kernel_Fn} can also be seen as a specific case of the general class of recursive scheme proposed in Example 3.1 of \cite{lpr26}.

The next two propositions show that the sequence of predictive distributions $(\alpha_n)_{n\geq0}$ defined in \eqref{eq:kernel_Fn} is never c.i.d., except for the degenerate case of no smoothing, but it can be a.c.i.d. under suitable conditions on the kernel and bandwidth. 

More specifically, Proposition~\ref{prop:kernel.cid} establishes that the only instance in which the sequence defined by \eqref{eq:kernel_Fn} is c.i.d. occurs when the kernel is degenerate and equal to the Dirac delta function. This scenario entails no smoothing and corresponds to adopting the empirical distribution function as the predictive distribution. %; in other words, it recovers the Bayesian Bootstrap. 
Informally, it means that the martingale condition would rule out any sequential kernel estimator. 

\begin{proposition}\label{prop:kernel.cid}
	Let $(\alpha_{n}(\cdot))_{n\geq0}$ be a sequence of the predictive distributions defined by \eqref{eq:kernel_Fn}. Then, the corresponding probability law is c.i.d. if and only if $K(x)=\delta(x)$ (Dirac delta function) almost everywhere. This corresponds to the empirical measure $\alpha_{n}(\cdot)= \frac{1}{n}\sum_{i=1}^{n} \mathbb{I}(X_{i}\in \cdot)$.
\end{proposition}

Proposition~\ref{prop:kernel.cid} corresponds to Theorem 1 of \cite{west1991kernel}, who proved a similar result for the classical (non-recursive) kernel density estimator where, for a fixed sample size $n$, the bandwidths $h_{1:n}$ are all equal. %Although his result is formulated in terms of statistical coherence, the notion of coherence adopted there implicitly gives the martingale property.

%Proposition~\ref{prop:kernel.acid} establishes that the sequence defined by \eqref{eq:kernel_Fn} is a.c.i.d. and becomes asymptotically exchangeable under suitable conditions. In this sense, 

Proposition~\ref{prop:kernel.acid} shows that $(X_n)_{n\geq1}$ can be a.c.i.d. and asymptotically exchangeable. Consequently, kernel estimators can be used for predictive inference if one works with our weaker technical condition. For $i\leq n$, define the kernel convolution
\begin{equation} \label{eq:kernel.convolution}
	(K_{n+1} * K_{i})(x;X_i):=\int  \frac{1}{h_{n+1}} K\left(\frac{x-x_{n+1}}{h_{n+1}}\right) \frac{1}{h_i}  K\left(\frac{x_{n+1}-X_i}{h_{i}}\right) \di x_{n+1},
\end{equation}
and let $\mu_{n+1*i}(\cdot;X_i)$ be the corresponding probability measure. 
\begin{proposition} \label{prop:kernel.acid}
	Let $(\alpha_{n}(\cdot))_{n \geq 0}$ be a sequence of the predictive distributions defined by \eqref{eq:kernel_Fn}. Suppose that there exist $d>0$ and $C>0$ such that $K(\cdot)$ satisfies for all $n\geq i\geq 1$
	\begin{equation}\label{eq:kernel_tv}
		TV(\mu_{n+1*i} (\cdot, X_i), \mu_{i} (\cdot,X_i)) \leq C \frac{h^d_{n+1}}{h^d_{i}}\ \ \ \ \ \text{a.s.}.
	\end{equation}
	Then, the corresponding probability law is a.c.i.d. with parameter
	\begin{equation} \label{eq:ker.norm.eta}
		\xi_{n}:=  \frac{1}{n+1} \sum_{i=1}^{n}\frac{1}{n} C \frac{h^d_{n+1}}{h^d_{i}}. 
	\end{equation}
	Moreover, a sufficient condition for $(X_n)_{n \geq 1}$ to be asymptotically exchangeable is that
 $h$ satisfies
	\begin{equation}\label{eq:kernel_bandwidth_cond}
		\frac{ h^d_{n+1}}{n+1} \sum_{i=1}^{n} 
		\frac{1}{h^d_{i}}=o \left(\log^{-1} n\right).
	\end{equation}
\end{proposition}

Proposition~\ref{prop:kernel.acid} provides an answer to a question raised by \cite{west1991kernel} regarding the existence of a statistical model underlying kernel-based predictives. The result suggests that, although kernel estimators may not correspond to a proper statistical model at finite sample sizes, they do in the limit, as they define a sequence that is asymptotically exchangeable under suitable conditions. % So, in the limit, $(X_{n})_{n\geq 1}$ corresponds to an exchangeable sequence generated from a given statistical model.  
This interpretation is analogous to results of \cite{fortini2020quasi}  for the predictive recursion of \cite{new98}.

At present, we lack theoretical results for a.c.i.d. sequences that would allow us to determine whether the limiting distribution is absolutely continuous. Such a result is available for c.i.d. sequences in \cite{ber12} and has been used by \cite{fortini2020quasi} and \cite{fong2021martingale}. However, that argument of \cite{ber12} relies on the martingale structure of the associated measure-valued process and does not apply to non-martingale settings. Moreover, it does not appear to be readily extendable to our framework. All we know is that $(P_n)_{n\geq 0}$ is absolutely continuous for all $n$. In the rest of the section, we discuss kernels and bandwidths that satisfy the conditions in Proposition~\ref{prop:kernel.acid}. 
\vspace{0.3cm}

\textit{Choice of the kernel.} The condition \eqref{eq:kernel_tv} on $K$ is not particularly restrictive. As examples, it is shown to hold for Gaussian, Laplace and Uniform kernels. Details on the calculations for the uniform kernel can be found in the Appendix. 

\begin{example} (Gaussian kernel)  Let $K(x) = \phi(x;0,1)$ for $x\in \mathbb{R}$, which implies $K_n(x)=\phi(x;X_n,h_n^2)$. The convolution $(K_{n+1} * K_{i})(x;X_i)=\phi(x;X_{i}, h_{i}^{2}+h_{n+1}^{2})$, %is available in closed form
	%\begin{align*}
	%	(K_{n+1} * K_{i})(x;X_i,h) & = \int_{-\infty}^{\infty}   \phi(x;x_{n+1},h_{n+1}^{2})  \phi(x_{n+1};X_{i},h_{i}^{2})\di x_{n+1} \\
     %   &= \phi(x;X_{i}, h_{i}^{2}+h_{n+1}^{2}),
	%\end{align*}
	following standard calculation of convolution of Gaussians. Condition \eqref{eq:kernel_tv} holds: $TV(\mu_{n+1*i}(\cdot,X_i), \mu_{i} (\cdot,X_i))\leq C \frac{h_{n+1}^2}{h_i^2},$
	%\begin{align*}
	%	TV(\mu_{n+1*i}(\cdot,X_i), \mu_{i} (\cdot,X_i)) & = \sup_{A\in \mathcal{B}(\mathbb{R})} \left| \int_A \phi(x;X_{i}, h_{i}^{2}+h_{n+1}^{2}) \di x - \int_A \phi(x;X_{i}, h_{i}^{2})  \right| \\
     %   &\leq C \frac{h_{n+1}^2}{h_i^2},
	%\end{align*}
	where the inequality follows from Theorem 1.3 of \cite{devroye2018total} (see Appendix Section~\ref{sec:background_tv_bounds}). 
\end{example}

\begin{example}\label{ex:laplace} (Laplace kernel) Let $K(x)=\text{Laplace}(x;\mu,b)$, a Laplace kernel with location $\mu$ and scale $b$, so that
 $K_n(x)=\text{Laplace}(x;X_n,h_n)$. Then, $(K_{n+1} * K_{i})(x;X_i)=\text{Laplace}(x;X_i,h_{n+1}+h_i)$ % The convolution of Laplace kernels is available in closed form
%\begin{equation*}
%		(K_{n+1} * K_{i})(x;X_i,h)= \frac{1}{2(h_{n+1}+h_i)} \exp \left(-\frac{|x-X_i|}{h_{n+1}+h_i}\right),
%\end{equation*}
%i.e., it is a Laplace density with location $X_i$ and bandwidth $h_{n+1}+h_i$. 
 and  Proposition~\ref{prop:tv_laplace}\textit{(a)} in the Appendix shows that  condition \eqref{eq:kernel_tv} holds with an upper bound on the total variation given by  $C\frac{h_{n+1}}{h_i}$ for some $C>0$.
\end{example}
\vspace{0.3cm}

\textit{Choice of the bandwidth.} The choice of the bandwidth is more critical. Condition \eqref{eq:kernel_bandwidth_cond} is satisfied by sequences $h$ of the type $h_n=e^{-n}$. This is because 
%\[
$\sum_{i=1}^n \frac{1}{h_i^d}= e \frac{e^{\epsilon(n+1)}-1}{e^d -1},$
%\]
which implies that 
\[
\frac{h_{n+1}^d}{n+1} \sum_{i=1}^{n} \frac{1}{h_i^d} = \frac{1}{(n+1)e^{d(n+1)}}e \frac{e^{d(n+1)}-1}{e^d -1}=\mathcal{O}\left(\frac{1}{n}\right).
\]
%Any sequence with a geometric decay will satisfy the above. 
%In the numerical experiments we will consider $h_n=a e^{-b(n-c)}$ for some suitable choices of $a,b,c>0$. 
This suggests that the resulting sequence $(\alpha_n)_{n \geq 0}$ needs to approach the empirical distribution rapidly, as $n$ increases, %if we want the resulting sequence of $(X_n)_{n \geq 1}$ to be asymptotically exchangeable
in order for Proposition~\ref{prop:kernel.acid} to hold. This happens by virtue of the kernel ``shrinking" as $n$ increases due to the fast decay of the bandwidth. %This is consistent with Proposition~\ref{prop:kernel.cid}: convergent almost supermartingales become martingale in the limit, the proposition states that in order for $(\alpha_n)_{n \geq 0}$ to be a martingale, one must use the empirical distribution. 

A bandwidth that decays at a polynomial rate ($h_n = n^{-b}$ with $b>0$) does not satisfy \eqref{eq:kernel_bandwidth_cond}. This may at first seem very restrictive, since a standard condition for kernel density estimation consistency is $h_n n \to \infty$. However, note that \eqref{eq:kernel_bandwidth_cond} is a condition on the sequence of predictive distributions, not on the sequence of estimators. This example highlights a crucial aspect: we are prescribing a rule for the bandwidth in generating synthetic data for predictive resampling, not for the bandwidth of the estimator used in predictive resampling. The latter can be chosen according to standard results on the frequentist consistency of kernel estimators. This is in line with what done by \cite{fong2021martingale}. In addition, we do not claim that such a decay is necessary for asymptotic exchangeability. It is sufficient for the observables to belong to the a.c.i.d. class.

A practical method to select the bandwidth is described in Appendix Section~\ref{sec:implementation_details}.

%The bandwidth used for predictive resampling and the one used for estimation do not need to match. We will return to this remark in \red{Section X}.

\begin{remark}\label{rmk:multi} (Multivariate kernels) As is often the case in the kernel literature, results for the univariate setting generalize to the multivariate one. Specifically, let $K:\mathbb{R}^p\to \mathbb{R}$ be an integrable function satisfying $\int_{\mathbb{R}^p}K(u)\di u=1$, and consider the case of the isotropic bandwidth $h=(h_{n})_{n\in \mathbb{N}}$, meaning that each coordinate is scaled by the same $h_n$. In this setting, we write $K_n (x)= \frac{1}{h_n^p} K(\frac{x}{h_n})$ with $x\in \mathbb{R}^p$. Similarly, we will have that $\int_{\mathbb{R}^p} K_n (x) \di x=1$. Proposition~\ref{prop:kernel.acid} will hold with identical conditions. 
\end{remark}

\subsection{Kernel Regression}\label{sec:kernel_regr}

Let $(Y_{n},X_{n})_{n \geq 1}$ be %a sequence of random variables defined on $(\Omega,\mathcal{A},\mathbb{P})$, 
composed of a response variable $Y_{n}\in \mathbb{R}$, and a vector of covariates, $X_n\in \mathbb{R}^p$. Similar to the previous section, given an initial distribution $\alpha_0$,  kernels $K_y: \mathbb{R} \to \mathbb{R}$ and  $K_x: \mathbb{R}^{p} \to \mathbb{R}$, and a sequence $h$, for $A\in \mathcal{B}(\mathbb{R}^{p+1})$, define $\alpha_{n}(A):=\mathbb{P}((Y_{n+1},X_{n+1})\in A |\mathcal{G}_{n})$ recursively as follows \begin{equation} \label{eq:kernel_reg}
	\alpha_{n}(A)=\frac{n-1}{n}\alpha_{n-1}(A)+\frac{1}{n} \int_{A} \frac{1}{h_n} K_y\left(\frac{y-Y_{n}}{h_{n}}\right) \frac{1}{h^p_n} K_x\left(\frac{x-X_{n}}{h_{n}}\right)\di x \di y,
\end{equation}
%where $\mathcal{G}_{n}=\sigma(Y_{1:n},X_{1:n})$, $A\in \mathcal{B}(\mathbb{R}^{p+1})$, and the bandwidth is isotropic. 
The predictive distribution above implies that the one-step ahead conditional mean
\begin{equation}\label{eq:mean_kernel_regr}
\mathbb{E}(Y_{n+1}|x ,\mathcal{G}_n) =\frac{\sum_{i=1}^n Y_i \, K_x(\frac{x-X_i}{h_i})}{\sum_{i=1}^n   K_x(\frac{x-X_i}{h_i})}
\end{equation}
corresponds to a recursive version of the Nadaraya-Watson kernel regression estimator; see \cite{gyorfi2006distribution} for more details and examples. %Note that $K_x$ does not need to be necessarily a multivariate kernel, it can also be a univariate isotropic kernel as it is often the case in kernel regression. 

Let $K(y,x)=K_y(y)K_x(x)$. If $K(y,x)$ is such that the corresponding measure and kernel convolution measure satisfies \eqref{eq:kernel_tv}, and $h$ satisfies \eqref{eq:kernel_bandwidth_cond}, $(Y_n, X_n)_{n \geq 1}$ is %$\mathcal{G}$-
a.c.i.d. and asymptotically exchangeable by Proposition~\ref{prop:kernel.acid} (see Remark~\ref{rmk:multi}). Below are two examples where this holds.

\begin{example} (Gaussian isotropic kernel) Choose both $K_y$ and $K_x$ to be standard univariate Gaussian kernels. The Gaussian isotropic kernel is such that $\frac{1}{h^p_n} K_x\left(\frac{||x-X_n||_2}{h_n}\right)=\phi_p\left(x;X_n,h^2_n I_p\right)$, where $\phi_p$ denotes the multivariate normal density and $I_p$ a diagonal matrix of size $p\times p$. Similarly,
	\[
	K_n(y,x) = \frac{1}{h_n}K_y \left(\frac{y-Y_n}{h_n}\right)\frac{1}{h^p_n}K_x\left(\frac{||x-X_n||_2}{h_n}\right) =\phi_{p+1}(y,x;(y_{n},x_{n}),h_{n}^{2} I_{p+1}).
	\]
	The kernel convolution is
%\begin{align*}
	$(K_{n+1} * K_{i})(y,x;(Y_i,X_i))
	= \phi_{p+1}(y,x;(Y_i,X_{i}), (h_{i}^{2}+h_{n+1}^{2})I_{p+1})$.
%\end{align*}
The TV bound \eqref{eq:kernel_tv} is
%\begin{equation*} %\label{eq:kernel_tv_regr}
	 $TV(\mu_{n+1*i}(\cdot,(Y_i,X_i)), \mu_{i}(\cdot,(Y_i,X_i))) \leq \frac{3}{2} (p+1)\frac{h^2_{n+1}}{h^2_{i}},$
%\end{equation*}
by Theorem 1.1 of \cite{devroye2018total} (see Appendix Section~\ref{sec:background_tv_bounds}). Hence, if $h$ satisfies \eqref{eq:kernel_bandwidth_cond},  $(Y_n,X_n)_{n \geq 1}$ is a.c.i.d.

\end{example}

\begin{example}(Random Forests) Random Forests (RFs) are a popular family of learning algorithms for classification and regression. % They are ensemble methods, based on many regression or classification trees which are grown independently. The different algorithms differ by how they randomize and grow their trees, with arguably the most popular version due to \cite{breiman2001random}. %The study of its statistical properties has proven to be particularly challenging; see \cite{biau2016random} for a review of recent results. 
Most theoretical results for the RF exploit the connection between kernel methods and random forest; see \cite{biau2016random} for a review.% We will adopt the same strategy. %To draw a connection between RFs and kernel methods require several simplifying assumptions. %Among which, the partitions of the underlying regression trees can be randomized via auxiliary variables not data-dependent. 

Under several assumptions - namely, the regressors $X\in [0,1]^p$ %, denoted by $\mathbb{P}(X)$, 
and are uniformly distributed on the $p$-dimensional unit cube, regression trees have a fixed number of leaves $l$, the partition of each tree is such that at each stage the feature to be used in the split is chosen at random, and the split point is a draw from a uniform random variable - \cite{breiman2000some} showed that, as the number of trees and the sample size go to infinity, the mean regression function of a RF $f_{\text{RF}} (x)\approx f^{ker}_{\text{RF}} (x)$, where $f^{ker}_{\text{RF}} (x)$
\begin{equation}\label{eq:rf_kernel_mean}
f^{ker}_{\text{RF}} (x) =\frac{\sum_{i=1}^n Y_i \, \exp( -\lambda \mid\mid x-X_i\mid\mid_1)}{\sum_{i=1}^n   \exp( -\lambda \mid\mid x-X_i\mid\mid_1)},
\end{equation}
where $\lambda= \log l/p$. $f_{\text{RF}} (x)$ matches the regression function of a Nadaraya-Watson kernel regression estimator with a Laplace isotropic kernel. This approximation is referred to as kernel-based RF \citep{scornet2016random}, where the approximate nature is clear since $f_{\text{RF}}$ with a finite number of trees is piecewise constant while \eqref{eq:rf_kernel_mean} is not. %Nevertheless, it has motivated both methodological works \citep{davies2014random,scornet2016random,balog2016mondrian} and helped establishing its mathematical properties \citep{scornet2016random}. 

%Note that, \eqref{eq:rf_kernel_mean} differs from \eqref{eq:mean_kernel_regr} simply because the bandwidth is constant for all samples while in \eqref{eq:mean_kernel_regr} is indexed by $i$. The difference follows simply

Let us select $K_y=\text{Laplace}(y;\mu_y,b)$ and $K_{x}(x)=\prod_{j=1}^p \text{Laplace}(x_j;\mu_{x_j},b) $, and \\ $ K_n(y,x)=\text{Laplace}(y;\mu_y,h_n) \prod_{j=1}^p \text{Laplace}(x_j;\mu_{x_j},h_n)$. Define a sequence of predictive distributions $(\alpha_n)_{n\geq 0}$ via the online scheme \eqref{eq:kernel_reg}. The sequence $(Y_{n},X_{n})_{n \geq 1}$ sampled via $(\alpha_n)_{n\geq 0}$ has a mean regression function equal to 
\begin{equation*}
\mathbb{E}(Y_{n+1}|x ,\mathcal{G}_n)=\frac{\sum_{i=1}^n Y_i \, \exp( -h_i \mid\mid x-X_i\mid\mid_1)}{\sum_{i=1}^n   \exp( -h_i \mid\mid x-X_i\mid\mid_1)},
\end{equation*}
a recursive version of \eqref{eq:rf_kernel_mean} with bandwidths indexed by $i$. This corresponds to an online version of the kernel-based RF. 
%Fix $K_y$ to be standard Laplace kernel and $K_{x}(x)=\prod_{j=1}^p K_{x_j}(x_j)$, with $K_{x_j}(x_j)$ standard Laplace kernels for all $j$. We have that 
%\begin{align*}
%     K_n(y,x) & =\frac{1}{h_n}K_y \left(\frac{y-Y_n}{h_n}\right)\frac{1}{h^p_n} \prod_{i=1}^p K_{x_j}\left(\frac{|x_j-X_{ij}|_1}{h_n}\right) \\
%     & = \frac{1}{h^{p+1}_n} \exp \left(-\frac{\mid y-Y_n\mid }{h_n}\right) \exp \left(-\frac{\mid\mid x -X_{i}\mid \mid_1}{2 h_n}\right).
%\end{align*}
%A sequence $(Y_{n},X_{n})_{n \geq 1}$ generated with the kernel above via the online scheme \eqref{eq:kernel_reg} has a mean regression function equal to the right side of \eqref{eq:rf_kernel_mean}; approximating an online version of the infinite-limit of the RF. In addition, 
 Following the same calculations as in Example~\ref{ex:laplace}, the kernel convolution is %$(K_{n+1} * K_{i})(y,x;(Y_i,X_i))=\text{Laplace}(\mu_y,h_n+h_i) \prod_{j=1}^p \text{Laplace}(\mu_{x_j},h_n+h_i)$
\begin{align*}
	(K_{n+1} * K_{i})(y,x;(Y_i,X_i))
	=& %\frac{1}{(h_n+h_i)^{p+1}} \exp \left(-\frac{\mid y-Y_i\mid }{h_n+h_i}\right) \exp \left(-\frac{\mid\mid x-X_{i}\mid \mid_1}{2 h_n+h_i}\right),
    \text{Laplace}(y;\mu_y,h_n+h_i) \prod_{j=1}^p \text{Laplace}(x_j;\mu_{x_j},h_n+h_i),
\end{align*}
and an upper bound for \eqref{eq:kernel_tv} is $C (p+1) \frac{h_{n+1}}{h_{i}}$ 
%\begin{equation*} 
	%$TV(\mu_{n+1*i}(\cdot,(Y_i,X_i)), \mu_{i}(\cdot,(Y_i,X_i))) \leq C (p+1) \frac{h_{n+1}}{h_{i}},$
%\end{equation*}
(see Proposition~\ref{prop:tv_laplace} b)). Thus, the sequence $(Y_{n},X_{n})_{n \geq 1}$ is a.c.i.d. 
\end{example}

%\section{Predictive Resampling with a.c.i.d. sequences}

%[not sure if necessary, but the idea here could be that we quickly retake what's in the background and we explain. ]

\section{Numerical Illustrations}\label{sec:simulations}

In this section, we illustrate  predictive resampling using the a.c.i.d. predictive schemes presented in the previous section. %All the schemes discussed satisfy Theorem~\ref{thm:asym.exch} conditions for the resulting $(X_n)_{n\geq 1}$ to be asymptotically exchangeable, making the underlying predictive rule valid to construct new bootstrap schemes. 
We focus on the nonparametric schemes based on kernels since illustrations for the parametric examples can be found in \cite{holmes2023statistical,glpr26a,fong2024asymptotics}. The code for reproducibility is available online. %at \url{https://github.com/lorenzocapp/acid_paper}.
The goal of the section is not to advocate in favor of one bootstrap scheme versus another, or predictive resampling versus standard ways of computing the posterior. %This latter question is discussed at length in the discussion paper of \cite{fong2021martingale}. 
Rather, we want to provide numerical evidence that the examples in Section~\ref{sec:examples} provide viable alternatives to existing proposals. As argued by \cite{glpr26a}, optimal predictive schemes will vary based on the problem and data at hand. To our knowledge, there does not yet exist a theoretical framework in this recent literature to make such a comparison, and numerical studies are currently the only available means of comparison.

The martingale posterior can be defined for any parameter of interest. \cite{glpr26b} consider selected population functionals, such as mean, variance, kurtosis, etc. Here, we consider infinite-dimensional ones. As discussed in previous sections, we do not have theoretical guarantees on whether KerP's limiting model is absolutely continuous. For this reason, we consider two functionals $\theta$ that do not require such guarantees: in Subsection~\ref{sec:ill_density}, a probability distribution; in Subsection~\ref{sec:ill_regression}, a regression function. Subsection~\ref{sec:ill_data} includes applications to real data. Kernel-based Predictive (KerP) schemes are implemented in \texttt{R}. We consider three types of kernels: Gaussian, Uniform, and Laplace. Methods to select the bandwidth $h$ are described in Appendix Section~\ref{sec:implementation_details}: they differ depending on whether we consider univariate or multivariate data, but all of them rely on standard selection rules implemented in \texttt{R}. More details on the implementation and pseudo-code can be found in Appendix Section~\ref{sec:implementation_details}. In Appendix~\ref{sec:sensitivity_bw}, we replicate the simulation study to assess the sensitivity of the results for KerP to different bandwidth selection rules. Although bandwidth choice is important for any kernel-based procedure, we find that our results are largely insensitive to the alternatives considered. This robustness may reflect the fact that all of the bandwidth selection rules are based on widely used criteria. Nevertheless, the proposed rule is heuristic and warrants further investigation. In the remainder of the paper, we therefore report results based on one of the most standard configurations.
 % We will define in each subsection the number of synthetic samples $M$ and number of boostrap replicates $B$ used to compute the martingale posterior. %To define the martingale posterior, we use $M=1000$ synthetic samples and $N=500$ replicates. 

To estimate a probability distribution, we compare KerP to predictive resampling based on the Gaussian Copula (GauC) algorithm of \cite{hahn18} (implemented in \cite{fong2021martingale}, code available at \url{https://github.com/edfong/MP}), Bayesian inference based on Gaussian Dirichlet Process Mixtures (DPM) (implemented in the \texttt{dirichletprocess} \texttt{R} package of \cite{ross2018dirichletprocess}), and the mixture of Gaussians driven by sample means and variances (GLPR) of \cite{glpr26b} (code made available by the authors). In regression, we compare KerP to regression function estimates of Gaussian Processes (GP) (implemented in the \texttt{laGP} \texttt{R} package of \cite{gramacy2016lagp}), Gaussian DPMs, and GLPR. GPs directly target the mean regression function, while for DPMs and kernel-based methods, we estimate the joint density of $(Y, X)$ and compute numerically $f(x):=\mathbb{E}(Y|x)$. For GLPR, we use the regression function defined in \cite{glpr26b}. All predictive methods use the same number of synthetic samples and number of bootstrap replicates. Both GauC and GLPR require hyperparameter tuning: we rely on the built-in procedure for GauC and greedy search for GLPR. Further details on alternative methods are in Appendix Section~\ref{sec:alternatives}.

Given an estimand $\theta$, let $\hat{\theta}$ denote the posterior mean, and $\hat{\theta}_{2.5}$ and $\hat{\theta}_{97.5}$ respectively the posterior $.025$ and $.975$ quantiles. Given a grid $\{x_i\}_{1:N}$, we use the following statistics to compare the methods:  $env=\frac{1}{N} \sum_{i=1}^N\mathbb{I} (\hat{\theta}_{2.5}(x_i) \leq \theta(x_i) \leq \hat{\theta}_{97.5}(x_i))$, which measures if $\theta$ is covered by the $95\%$ credible region along the grid (a proxy for coverage), $dev=\frac{1}{N} \sum_{i=1}^N| \hat{\theta}(x_i)-\theta(x_i)|$, which measures the bias in L1 norm, and $awd=\frac{1}{N} \sum_{i=1}^N(\hat{\theta}_{97.5}(x_i)-\hat{\theta}_{2.5}(x_i))$, which measures the average width of the interval. $rdev$ and $rwd$ will denote the corresponding metrics normalized by the true values. In the regression data sets, we will also compute the mean squared error $MSE=\frac{1}{n_{test}} \sum_{i=1}^{n_{test}} (y_i -\hat{f}(x_i))^2$. We will also report computing times as obtained on an Intel Xeon E5-2630\footnote{Available in the computer cluster Correfoc at the Universitat Pompeu Fabra.}. The different methods are implemented in different programming languages - GauC in \texttt{Python}, DPM and KerP in \texttt{R}, and GLPR in \texttt{C++/Rcpp} - which complicates a direct comparison of their computational efficiency.

\subsection{Probability Distributions}\label{sec:ill_density}

Let $Y_n | P \iidsim P$, where $P$ is a mixture model with a finite number of components and density $p$. $P$ is random: for each new data set we sample the number of components, the weights of the mixture, the type of kernel (Gaussian or Student-t ($\nu=10$)), and the location and scale parameters of each kernel; see Appendix Section~\ref{sec:data_generation} for details. DPMs and GauC are specifically designed to estimate this class of mixture models, whereas KerP is not. DPMs will be considered as a benchmark, since inference is exact and does not rely on predictive resampling.  We consider $100$ simulated data sets sampled from $100$ distinct distributions for two sample sizes ($n=50,200$). GauC requires a grid to numerically approximate a density (DPM and KerP do not). We fix it to $\{x_i\}_{1:N}:=\{-4,\ldots, 3.95\}\times\hat{\sigma} + \bar{X}$, where $\bar{X}, \hat{\sigma}$ are respectively the sample mean and standard deviation, and $N=100$. The same grid is also used to compute $env,dev,$ and $awd$. GLPR requires the number of mixture components as a hyperparameter: we fix it to the maximum number of components of $P$ in the simulation study. Figure~\ref{fig:density} depicts estimates (solid lines), $95\%$ credible regions (gray shaded area), true density (dashed line), and data (points) for one possible realization $P$ and data for $n=200$. In this example, all methods perform similarly. 

\begin{figure}[!t]
        \begin{center}
            \includegraphics[width=14cm, height=6cm, keepaspectratio ]{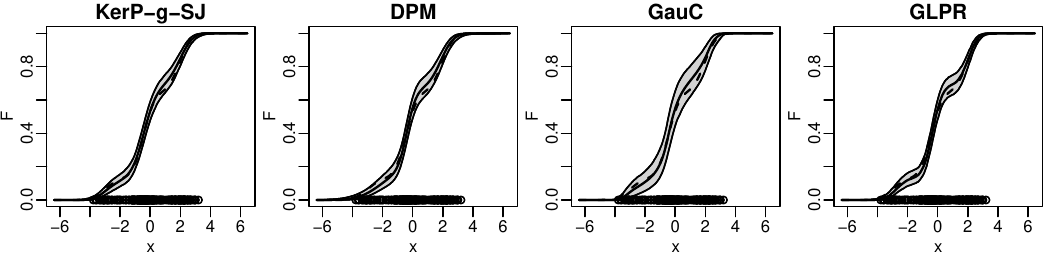}
       \end{center}

        \caption{Simulations: probability distribution  (one example). \small{Estimates for one of the random $P$ ($n=200$). In each panel: true $p$ (dashed line), posterior mean $\hat{p}$ (solid line), and $95\%$ credible regions (gray shaded area). Dots depict the observations. The four panels differ by the method: KerP with Gaussian kernel and bandwidth via the method of  \cite{sheather1991reliable} (left), DPM (middle-left), GauC (middle-right), and GLPR (right). 
        }}
        \label{fig:density}
    \end{figure}
    
\begin{table}[!t] \centering 
  \label{tab:density}
\begin{tabular}{@{\extracolsep{5pt}} l|cccc|cccc} 
	\\[-1.8ex]\hline 
 & \multicolumn{4}{c }{n = 50 }& \multicolumn{4}{ c}{n = 200 } \\
Method & $env$ & $dev$ & $awd$ &$time$ & $env$ & $dev$ & $awd$ &$time$ \\ 
\hline 
GauC & $0.500$ & $0.022$ & $0.053$ & $-$ & $0.541$ & $0.013$ & $0.040$ & $-$ \\ 
DPM & $0.994$ & $0.018$ & $0.090$ & $10.685$ & $0.931$ & $0.009$ & $0.046$ & $73.647$ \\ 
GLPR & $0.572$ & $0.022$ & $0.087$ & $0.338$ & $0.509$ & $0.011$ & $0.043$ & $0.366$ \\ 
KerP-g-cv & $0.831$ & $0.020$ & $0.092$ & $10.916$ & $0.769$ & $0.011$ & $0.044$ & $11.612$ \\ 
KerP-l-cv & $0.744$ & $0.026$ & $0.095$ & $8.842$ & $0.719$ & $0.014$ & $0.045$ & $9.396$ \\ 
KerP-u-cv & $0.700$ & $0.019$ & $0.091$ & $7.608$ & $0.694$ & $0.009$ & $0.044$ & $8.084$ \\

\hline
\end{tabular} 
\caption{Simulations: probability distribution. \small{Median across $100$ data sets for different sample sizes (n). $env$ is a measure of coverage (the higher the better), $dev$ is an L1 bias (the lower the better), and $awd$ is the average width of the credible bands (the lower the better). In KerP, g stands for Gaussian kernel, l for Laplace, u for uniform, and SJ for the bandwidth selection of \cite{sheather1991reliable}.} }
\end{table}

Table~\ref{tab:density} summarizes the median metrics across $100$ repetitions. Performances are comparable across methodologies. As expected, DPM achieves the lowest bias ($dev$) and highest coverage ($env$) but can have the widest credible bands ($awd$) when $n=200$. KerP based on the Gaussian and Uniform kernel achieves similar performances in terms of $dev$ and $awd$, even if the coverage is lower. It remains superior to the other predictive resampling methods across all accuracy metrics. Comparable numerical performance is achieved about seven times faster than MCMC-based methods. GauC's time is not reported because some runs' times were unreliable (due to our implementation on the cluster). Nevertheless, GLPR is the fastest method. It is unclear why it should be faster than kernel-based methods. A possible explanation is that it is implemented in a more efficient programming language (\texttt{C++} as opposed to the others that are on \texttt{R}). The resampling methods can be significantly accelerated when run on a computing cluster in parallel, see \cite{fong2021martingale}.

%More interesting is the relative performance between kernel methods. It is no surprise that the Gaussian kernel outperforms the others, which are less commonly used for density estimation. The uniform kernel may achieve the lowest bias, but it has wider credible regions and lower coverage. The Laplace kernel is the worst-performing one. GauC has the narrowest credible bands ($awd$), which leads to a lower coverage ($env$). KerP's performance improves with a different bandwidth; see SM Section~\ref{sec:sensitivity_bw}.

\subsection{Regression Function}\label{sec:ill_regression}

Let $Y_n = f(X_n) +\epsilon_n$ with $\epsilon_n\iidsim P$ and $X_n\in \mathbb{R}$. For each regression function $f$, we sample $n=200$ training data with $X_n \iidsim \mathcal{U}(0,5)$, while test points are defined on a regular grid $\{x_i\}_{i:N}:=\{-4,\ldots, 3.95\}$ with $N=100$ evenly spaced values. To simulate data, we consider three Data-Generating Mechanisms (DGM) where we vary how we sample $f$ and the noise distribution $P$.  DGMs 1 and 2 have $f\sim GP$ with mean 0 and a Gaussian kernel. In DGM1, $P$ is Gaussian zero mean and variance one, while in DGM2, $P$ is Student-t zero mean and $\nu=5$. DGM3 has $f\sim GP$ with mean 0 and a kernel obtained by summing a Gaussian and an Exponential one, $P$ is Gaussian zero mean and variance one. Since all DGMs are GPs, GP performance can be interpreted as an oracle, because the model is correctly specified. More details and the exact parametrization are in Appendix Section~\ref{sec:data_generation}. For each combination, we simulate $100$ data sets. We compute $env,dev,awd$, and $MSE$ on the test data for KerPs, GPs, DPMs, and GLPR (the hyperparameter ``number of clusters" is optimized on a grid ($3,5,6,7,9$), we report the best-performing version of the algorithm). Figure~\ref{fig:regression} depicts estimates (solid lines), $95\%$ credible regions (gray shaded area), true regression function (dashed line), and data (points) for one possible realization of DGM1. In this example, KerP (Laplace kernel and CV bandwidth) GLPR and GP recover $f$, with GP and GLPR having smoother bands. This largely follows from the choice of the kernel. DPM oversmooths and fails to capture some local features of $f$, and has credible bands comparable to KerP. This provides further evidence that predictive resampling offers an alternative to computationally intensive MCMC-based methods.

Table~\ref{tab:regression} summarizes the median of the 100 repetitions of the statistics considered for the three DGMs. GPs achieve the best (oracle) performance across the considered metrics. Surprisingly, KerP with Laplace and Uniform kernels is comparable to GPs' performance - which is a baseline oracle performance. %GP hyperparameters are selected via MLE with \texttt{laGP} implementation. This raises the question of whether alternative hyperparameter inference methods, such as fully Bayesian approaches, could yield different results. 
Interestingly, KerP with Laplace kernel achieves the lowest test MSE, even in DGM1 when GP is the true DGM. Recall that this algorithm is linked to an infinite limit of the RF, a connection that may explain the good performance. The good performance of KerP with a uniform kernel may be explained by the uniform sampling of the covariates. KerP with the Gaussian kernel and Gaussian DPMs oversmooth the regression function. GLPR has bias comparable to that of KerP with Gaussian kernel; however, its credible bands have lower coverage, and are substantially narrower. In terms of computational times (for DGM3), GP is by far the fastest method. While GLPR remains faster than KerP, the difference becomes less substantial compared to Table~\ref{tab:density}. Moreover, if hyperparameter optimization was included in the computational time, there would be practically no differences between the two methods, despite GLPR's more efficient implementation in \texttt{C++}. However, as mentioned above, a rigorous comparison of computational efficiency requires additional work. %The use of DPMs to estimate a mean regression function is not standard. We included it solely to draw an analogy with the way we implemented KerP, where we estimate the regression function in the same way (we estimate the multi.  

\begin{figure}[!t]
        \begin{center}
            \includegraphics[width=15cm, height=6cm, keepaspectratio]{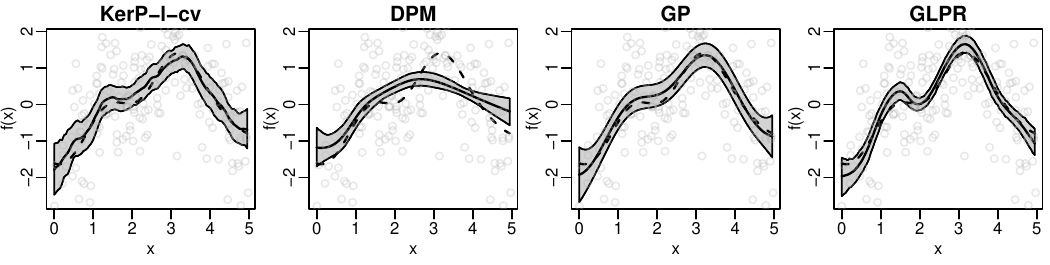}
        \end{center}
        \caption{Simulations: regression function (one example). \small{Estimates for one of the random $f$ ($n=200$, DMG1). In each panel: true $p$ (dashed line), posterior mean $\hat{p}$ (solid line), and $95\%$ credible regions (gray shaded area). Dots depict the observations. Four panels differ by the method: KerP with Laplace kernel and bandwidth via cross-validation (left), DPM (middle-left), GP (middle-right), and GLPR (right).
        }}
        \label{fig:regression}
    \end{figure}

\begin{table}[!t] \centering 
  \label{tab:regression}
\scalebox{0.80}{\begin{tabular}{@{\extracolsep{5pt}} l|cccc|cccc|ccccc} 
\\[-1.8ex]\hline 
& \multicolumn{4}{c}{DGM1: Gaussian $\epsilon$} & \multicolumn{4}{c}{DGM2: Student-t $\epsilon$} &\multicolumn{4}{c}{DGM3: Gauss-Exp Kernel}\\
Method & $env$ & $dev$ & $awd$ & $MSE$ & $env$ & $dev$ & $awd$ & $MSE$ & $env$ & $dev$ & $awd$ & $MSE$ & $time$ \\ 
\hline \\[-1.8ex]
DPM & 0.34 & 0.34 & 0.38 & 1.15 & 0.36 & 0.31 & 0.37 & 1.13 & 0.43 & 0.33 & 0.47 & 1.71 & 1404.52 \\ 
GLPR & 0.67 & 0.21 & 0.51 & 1.06 & 0.70 & 0.19 & 0.51 & 1.07 & 0.72 & 0.22 & 0.60 & 1.62 & 91.74 \\ 
GP & 0.94 & 0.15 & 0.73 & 1.05 & 0.94 & 0.16 & 0.72 & 1.04 & 0.94 & 0.19 & 0.89 & 1.58 & 0.06 \\ 
KerP-g-cv & 0.78 & 0.21 & 0.68 & 1.06 & 0.79 & 0.20 & 0.66 & 1.06 & 0.84 & 0.24 & 0.78 & 1.62 & 365.86 \\ 
KerP-l-cv & 0.90 & 0.17 & 0.75 & 1.04 & 0.90 & 0.17 & 0.73 & 1.04 & 0.92 & 0.20 & 0.89 & 1.58 & 235.68 \\ 
KerP-u-cv & 0.94 & 0.17 & 0.79 & 1.04 & 0.94 & 0.16 & 0.76 & 1.03 & 0.92 & 0.21 & 0.93 & 1.60 & 228.71 \\ 
\hline 
\end{tabular} }
 \caption{Simulations: regression function. \small{Median across $100$ data sets for three different data generating mechanisms ($n_{train}=200$,$n_{test}=100)$. $env$ is a measure of coverage (the higher the better), $dev$ is an L1 bias (the lower the better), and $awd$ is the average width of the credible bands (the lower the better), $MSE$ is the sample mean squared error on test data (the lower the better). In KerP, g stands for Gaussian kernel, l for Laplace, u for uniform, cv for the bandwidth selection with smoothed cross-validation. }}
\end{table} 

\subsection{Data Analysis}\label{sec:ill_data}

%We reanalyze a few of the data sets considered by \cite{fong2021martingale}. 

In this subsection, although there are no guarantees that the limiting model associated with the sequence induced by KerP is absolutely continuous, we plot density functions instead of cumulative distribution functions for ease of visualization. Recall that, due to the finite truncation involved in predictive resampling, the predictive distribution 
$P_{n+M}$ defined by KerP is absolutely continuous for finite $M$. Determining whether the full sequence converges to an absolutely continuous model is left for future work. 

\noindent\textbf{Galaxy data.} This data set is a standard benchmark for Bayesian nonparametric estimation and clustering (\cite{roeder1990density}, publicly available in \texttt{R}). It contains $n=82$ galaxies from six conic sections of an unfilled survey of the Corona Borealis region. Figure~\ref{fig:galaxy} includes replicates of the analysis of \cite{fong2021martingale} for GauC and DPM (middle and right panels) and the predictive resampling posterior obtained with KerP with Gaussian kernel and bandwidth sequence selected via the method of \cite{sheather1991reliable} ($B=2000$, $M=3000$). KerP estimates are more closely aligned with DPM than GauC, with a mode around $17$ km/sec not captured by GauC.

\begin{figure}[!t]
        \begin{center}
            \includegraphics[width=14cm, height=6cm, keepaspectratio ]{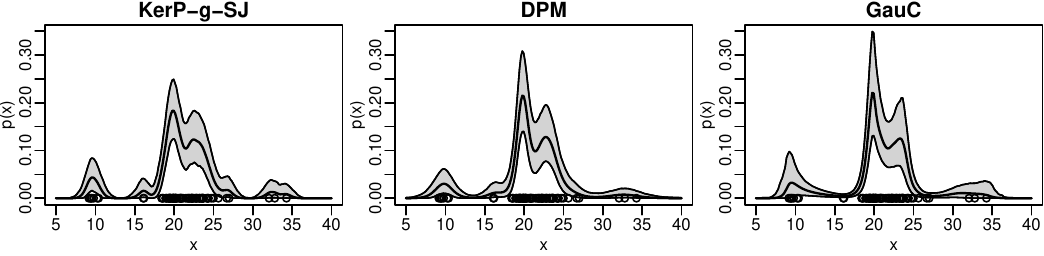}
        \end{center}
        \caption{Galaxy data. \small{In each panel: posterior mean $\hat{p}$ (solid line), and $95\%$ credible regions (gray shaded area). Dots depict the observations. Three panels differ by the method: KerP with Gaussian kernel and bandwidth via \cite{sheather1991reliable} (left), DPM (middle), and GauC (right).}}
        \label{fig:galaxy}
    \end{figure}

\noindent\textbf{Air quality data.} In this example, we estimate a bivariate density of the radiation and ozone variables from the \texttt{airquality} data set in \texttt{R}. They are daily measurements (from 01-05-1973  to 30-09-1973) collected in New York of solar radiation (in Langleys) and mean ozone (in parts per billion). There are 42 missing observations which we remove, resulting in $n=111$. Figure~\ref{fig:airquality_den} depicts contour plots of the mean density estimates obtained via predictive resampling of GauC (same parametrization as in \cite{fong2021martingale}) and KerP with Gaussian kernel with bandwidth selected via the multivariate plug-in method of \cite{wand1994multivariate} ($M=2000$, $B=500$). This results in two distinct bandwidth sequences for the two dimensions, similar to what is done for GauC where each dimension has its correlation parameter. Contour plots of the posterior standard deviation for the two methods are in the Appendix (Figure~\ref{fig:airquality_sd}). Compared to GauC, KerP's estimate is more diffuse and the mode in the upper corner of the panel appears to be more skewed. KerP's predictive resampling posterior has wider credible regions around the two modes (Figure~\ref{fig:airquality_sd}).

\begin{figure}[!t]
        \begin{center}
            \includegraphics[scale=0.8]{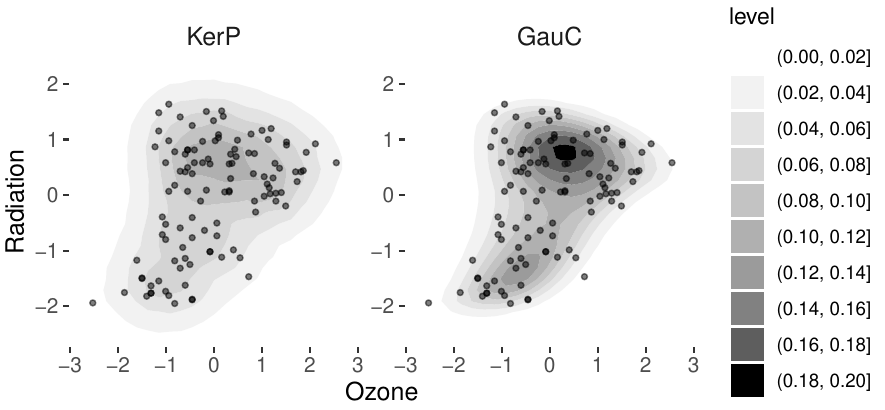}
        \end{center}
        \caption{Air quality data: posterior mean density.  \small{In each panel: contour plot of the posterior mean. Dots depict the observations. The two panels differ by the method: KerP with Gaussian kernel and bandwidth via the multivariate plug-in method of  \cite{wand1994multivariate} (left), and GauC with parameters as in \cite{fong2021martingale} (right).}}
        \label{fig:airquality_den}
    \end{figure}

\noindent \textbf{LIDAR data.} It consists of $221$ observations from a light detection and ranging (LIDAR) experiment. It includes a regressor, the distance traveled before the light is reflected back to its source ($x$, range), and a dependent variable,  the logarithm of the ratio of received light from two laser sources ($y$, log intensity). We estimate the regression function with the full data sets via KerP and DPM (Figure~\ref{fig:lidar_regression}) and the conditional densities at $x=0$ and $x=-3$ (Appendix Figures~\ref{fig:lidar_density0} and \ref{fig:lidar_density-3}). In addition, using holdout data sets ($10$ random splits, $80\%$ training, $20\%$ test), we do a prediction exercise using the regression function estimated on the training data. For KerP, we used a Gaussian kernel with bandwidth selected via the multivariate plug-in method of \cite{wand1994multivariate} ($M=1000$, $B=600$). 

In Figure~\ref{fig:lidar_regression}, KerP regression function appears less smooth than DPM and with wider credible regions. The two estimates differ slightly for positive values of range, with DPM having a linear behavior, while the KerP estimate has a slightly positive curvature. KerP $MSE$ is $0.1003$ and DPM $MSE$ is $0.1085$, approximately $8\%$ higher. Appendix Figure~\ref{fig:lidar_density0} depicts a conditional density at a value $x=0$, which is in the middle of the data set. The DPM estimate appears remarkably smoother and with much narrower credible bands than KerP. As in \cite{fong2021martingale}, we estimated the conditional density at $x=-3$, which is well outside the interval of observed values (Appendix Figure~\ref{fig:lidar_density-3}). KerP does not perform well, with wide credible bands and bumpy estimates. This illustrates a known limitation of kernel methods: with low bandwidth and no nearby observations, there is no borrowing of information. Additionally, the Gaussian kernel may be ill-suited in this scenario, because the regressor (range) forms an almost regular grid.

\begin{figure}[!t]
        \begin{center}
            \includegraphics[scale=0.8]{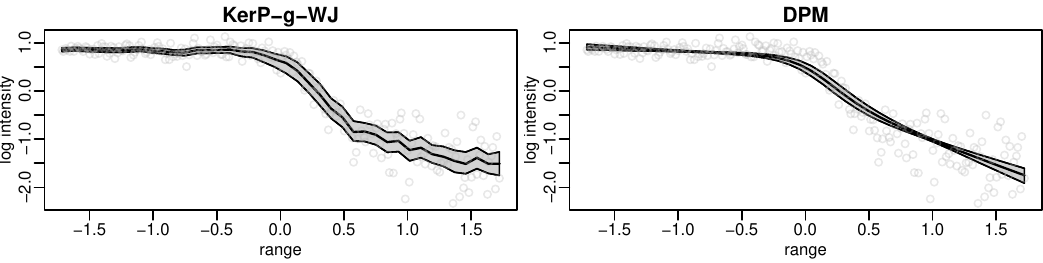}
        \end{center} 
        \caption{LIDAR: regression function. \small{In each panel: posterior mean $\hat{f}$ (solid line), and $95\%$ credible regions (gray shaded area). Dots depict the observations. The two panels differ by the method: KerP with Gaussian kernel and bandwidth via the multivariate plug-in method of \cite{wand1994multivariate} (left), and DPM (right).}}
        \label{fig:lidar_regression}
    \end{figure}

\section{Discussion}\label{sec:conclusion}

We defined a new  class of predictive schemes suitable for bootstrap methods like predictive resampling. %, belong to a broader class than those characterized by exchangeability or conditional identical distribution.
It was known that, if one only requires the sequence of random variables to be asymptotically exchangeable, then the predictive rule does not need to be a martingale. The main contribution of this paper is the definition of a new property - almost conditionally identically distributed variables - which provides a practical criterion for determining whether an algorithm can be used for predictive resampling. %Our motivation stems from the desire to use algorithms developed in machine learning and learning theory for prediction within a statistical framework for parameter inference. 
We demonstrated that existing and new predictive rules satisfy the a.c.i.d. property: notably, kernel estimators, including an approximation to the random forests, and the parametric Bayesian bootstrap. 

The departure from the martingale property may at first appear non-innocuous: the predictive distribution is no longer an unbiased estimator of the limiting measure, a property often associated with coherence in the Bayesian literature (see also predictive coherence, \cite{fong2021martingale}). However, modern statistics provides numerous examples in which introducing bias into an estimator need not be detrimental. Our numerical results provide a further example in the context of Bayesian predictive inference: finite-sample accuracy is not compromised and, in some cases, is even improved. 

To provide a possible explanation for this observation, recall that predictive distributions associated with exchangeable sequences are characterized by two properties \citep{for00}: (i) they are measure-valued martingales, and (ii) they are symmetric functions of past observations. C.i.d. r.v.s generalize exchangeability to asymptotic exchangeability by relaxing the second property while preserving the first, thereby retaining the no-bias property of the predictive distribution. A.c.i.d. r.v.s, in contrast, relax both properties. Interestingly, some of the a.c.i.d. schemes considered here, despite allowing for bias in the predictive distribution, appear to be less sensitive to the ordering of the observations than their c.i.d. (martingale) counterparts. In some respects, it might be possible that they could be ``closer'' to exchangeability than the corresponding c.i.d. schemes. This raises the question of whether reduced sensitivity to the ordering of past observations may itself be a desirable feature and, if so, whether a suitable notion of ``closeness to exchangeability'' can be formalized, one that accounts simultaneously and balances for deviations from both predictive properties characterizing exchangeability.

To establish the a.c.i.d. property, we relied on total variation bounds. This approach can be substantially more tractable than verifying almost sure weak convergence of a predictive sequence, although it is not always feasible or straightforward. Developing conditions that imply the a.c.i.d. property while working with a smaller class of sets would be valuable. One way would be to use algebras that generate the Borel sigma-algebra and monotone class-type arguments. An alternative route may involve bounding other metrics (for example, the Wasserstein distance) that enable control over the total variation distance. For instance, \cite{chae2020wasserstein} demonstrate that, under sufficient smoothness of the underlying densities, the total variation distance can be bounded by a power of the Wasserstein distance.

More technical tools for almost-supermartingale random measures are necessary. In the martingale case, there are central limit theorems \citep{berti2004limit} and sufficient conditions for the limiting model to be absolutely continuous \citep{berti2013exchangeable}. These results are of interest to statisticians: for example, \cite{fortini2020quasi} use CLTs to build asymptotic credible intervals; \cite{fortini2020quasi} also characterize the limit model of Newton’s algorithm, and \cite{fong2021martingale} that of the Gaussian copula. Similar results would enable analogous developments in this context for specific procedures, such as kernel estimators.

From a methodological standpoint, it would be of interest to examine the a.c.i.d. property for other algorithms in machine learning and statistical learning. We did not investigate wavelet-based or orthonormal basis estimators, and our analysis of random forests was based on a simplified approximation: there are alternatives to this such as the proposals in \cite{scornet2016random}. Among the methods we considered, kernel-based predictive resampling stands out as especially promising. Compared to the Bayesian bootstrap, kernel methods are better suited for estimating continuous parameters; compared to copula-based approaches, they do not require a discretization grid, and benefit from an extensive literature and widespread use. However, it remains crucial to develop principled procedures for selecting key parameters, such as the bandwidth, when using kernel methods in this setting.

\subsection*{Acknowledgments}

{L.C. acknowledges the support of Ayudas Fundacion BBVA a Proyectos de Investigacion Cientifica 2021, the Spanish Ministry of Economy and Competitiveness grant PID2022-138268NB-I00, financed by \\MCIN/AEI/0.13039/501100011033, FSE+MTM2015-67304-P, and FEDER, EU;  the grant Ramon y Cajal 2022 RYC2022-038467-I, financed by MCIN/AEI/10.13039/501100011033 and FSE+, and the Severo Ochoa Programme for Centres of Excellence in R\&D (Barcelona School of Economics CEX2019-000915-S), funded by MCIN/AEI/10.13039/50110001103. 
M.B. gratefully acknowledges support for this work from the London Mathematical Society, Research in Pairs grant [grant number 42418]. 
We also thank the staff at the Scientific computing core facility at Universitat Pompeu Fabra for their help. L.C. and M.B. thank the Associate Editor and three Reviewers for helpful suggestions that significantly improved the paper.

%\newpage\appendix
\bibliographystyle{agsm}
\bibliography{bibliography_jmlr.bib}

\appendix

\section{Background and Auxiliary results}

\subsection{Almost Supermartingales}\label{sec:app.almostsupermartingales}
For ease of reading, we recall some results about almost supermartingale from \cite{robbins1971convergence}. We present these results in a slightly simplified formulation, which is still sufficient for the proof of Theorem~\ref{thm:asym.exch}. Let $(Z_{n})_{n\geq 0}$ and $(\xi_{n})_{n\geq 0}$ be two sequences of random variables adapted to a filtration $(\mathcal{G}_{n})_{n\geq 0}$. Also, let $\xi_{n}$ be non-negative for all $n\geq 0$. Then, $(Z_{n})_{n\geq 0}$ is a $(\xi_{n})_{n\geq 0}$-\textit{almost supermartingale} if for all $n\geq 0$,
\begin{equation*}
    \mathbb{E}(Z_{n+1}|\mathcal{G}_{n}) \leq Z_{n} + \xi_{n} \ \ \ \ \ \text{a.s.}.
\end{equation*}

\cite{robbins1971convergence} have shown that almost supermartingales satisfy counterparts of the Doob convergence theorem and maximal inequality, valid for supermartingales. Specifically, 

\begin{theorem}[Theorem 1, \cite{robbins1971convergence}] \label{th:robsig}
    If $(Z_{n})_{n\geq 0}$ is a $(\xi_{n})_{n\geq 0}$-almost supermartingale, then $\lim_{n\to \infty} Z_{n}$ exists and is finite on the set $\{ \omega\in \Omega: \sum_{n=0}^{\infty} \xi(\omega) < \infty \}$.
\end{theorem}

\begin{proposition}[Proposition 2, \cite{robbins1971convergence}] \label{eq:max.ineq.robb}
     If $(Z_{n})_{n\geq 0}$ is a $(\xi_{n})_{n\geq 0}$-almost supermartingale, then for any $a>0$, $m\in \mathbb{N}$, and $n\geq m$
     \begin{equation*}
         \mathbb{P}\left( \max_{m\leq k\leq n} Z_{k} \geq a |\mathcal{G}_{m}\right) \leq \frac{1}{a}\left[Z_{m} + \mathbb{E}\left(\sum_{k=m}^{n-1}\xi_{k} |\mathcal{G}_{m}\right)\right].
     \end{equation*}
\end{proposition}

\subsection{Bounds on the Total Variation Distance}\label{sec:background_tv_bounds}

\noindent \textit{TV bounds between Gaussians.} The following results are given in \cite{devroye2018total}. Theorem~\ref{thm:devroye_univariate} provides an upper and lower bounds between univariate Gaussians with arbitrary mean and variance.  Theorem~\ref{thm:devroye_multivariate} is the corresponding multivariate result under the assumption of equal mean vectors. Alternative bounds in total variation between high-dimensional Gaussian distributions were derived in \cite{bp24}.

\begin{theorem}\label{thm:devroye_univariate}(Theorem 1.3, \cite{devroye2018total}) Let $\mu_1,\mu_2 \in \mathbb{R}$ and $\sigma^2_1,\sigma^2_2 \in \mathbb{R}^+$, the following upper and lower bounds hold
\begin{align*}
\frac{1}{200} \min & \left\{1,  \max \left\{\frac{|\sigma^2_1-\sigma^2_2|}{\sigma^{2}_{1}},  \frac{40|\mu_1-\mu_2|}{\sigma_{1}} \right\} \right\}   \leq \\
& \leq TV (\text{N}(\mu_1, \sigma^2_1), \text{N}( \mu_2, \sigma^2_2) 
 \leq \frac{3|\sigma^2_1-\sigma^2_2|}{2\sigma^2_1}+\frac{|\mu_1-\mu_2|}{2\sigma_1}.
\end{align*}
\end{theorem}

\begin{theorem}\label{thm:devroye_multivariate}(Theorem 1.1, \cite{devroye2018total}) Let $\mu \in \mathbb{R}^p$, $\Sigma_1$ and $\Sigma_2$ be positive definite matrices, and $\lambda_1,\ldots,\lambda_p$ denote the eigenvalues of $\Sigma_1^{-1}\Sigma_2-I_d$
\[
\frac{1}{100}  \leq \frac{\text{TV} (\Phi(\cdot; \mu, \Sigma_1),\Phi(\cdot; \mu, \Sigma_2)}{\min\left\{1,\sqrt{\sum_{i=1}^p \lambda_i^2} \right\}} \leq \frac{3}{2}.
\]    
\end{theorem}

\noindent \textit{TV bounds between Laplace.} We include a proof for the following result as we did not find a reference for it. In the main manuscript, we denoted with $\text{Laplace}(x;\mu,c)$, a Laplace kernel with location $\mu$ and scale $c$. In this section, we use $\text{Laplace}(\mu,c)$ (or $\text{Laplace}(\cdot;\mu,c)$ )to refer to the corresponding probability measure. 

\begin{proposition}\label{prop:tv_laplace} \textit{(a)} Let $X\sim\text{Laplace}(\mu,c)$ and $Y\sim \text{Laplace}(\mu,c+a_n)$ for $c\geq a_n>0$ and $a_n \to 0$. For some $C>0$, the following holds
\[
\text{TV} (\text{Laplace}(\mu,c),\text{Laplace}(\mu,c+a_n)) = \left(\frac{c}{c+a_n}\right)^{\frac{c}{a_n}}-\left(\frac{c}{c+a_n}\right)^{\frac{c+a_n}{a_n}} \leq C\frac{a_n}{c}.
\]  
\textit{(b)} Let $X=(X_1,\ldots, X_p) \sim P_X$ be such that $X_i\iidsim \text{Laplace}(\mu,c)$ for $i=1,\ldots,p$ and $Y=(Y_1,\ldots, Y_p)\sim P_Y$ be such that $Y_i\iidsim \text{Laplace}(\mu,c+a_n)$ for $i=1,\ldots,p$, then 
\[
\text{TV} (P_X,P_Y) \leq C p \frac{a_n}{c}.
\]
\end{proposition}
\begin{proof} For part \textit{(a)} we have that
\begin{align*}
   \text{TV} &(\text{Laplace}(\mu,c),\text{Laplace}(\mu,c+a_n))  =\\
  & =\int\frac{1}{2}\left\vert \frac{1}{2 (c+a_n)} \exp\left( -\frac{|x|}{c+a_n}  \right) - \frac{1}{2 c} \exp\left( -\frac{|x|}{c}  \right) \right\vert \di x \\
   &=\int_0^{\infty}\frac{1}{2}\left\vert \frac{1}{ c+a_n} \exp\left( -\frac{x}{c+a_n}  \right) - \frac{1}{ c} \exp\left( -\frac{x}{c}  \right) \right\vert \di x\\
   &=\int_{x^*}^{\infty} \frac{1}{c+a_n} \exp\left( -\frac{x}{c+a_n}  \right) - \frac{1}{ c} \exp\left( -\frac{x}{c}  \right) \di x \\
   &=\exp\left( -\frac{x^*}{c+a_n} \right) - \exp\left( -\frac{x^*}{c}  \right),
\end{align*}
where $x^*=\log\frac{c+a_n}{c}(c+a_n)c/a_n$ is the point where the two kernels intersect. This proves the equality
\[
\text{TV} (\text{Laplace}(\mu,c),\text{Laplace}(\mu,c+a_n)) = \left(\frac{c}{c+a_n}\right)^{\frac{c}{a_n}}-\left(\frac{c}{c+a_n}\right)^{\frac{c+a_n}{a_n}}.
\]
To see the inequality, rewrite the TV as 
\[
\left(\frac{c}{c+a_n}\right)^{\frac{c}{a_n}}-\left(\frac{c}{c+a_n}\right)^{\frac{c+a_n}{a_n}}=\left(\frac{c}{c+a_n}\right)^{\frac{c}{a_n}}\left[1 - \frac{c}{c+a_n}\right] \leq \left(\frac{c}{c+a_n}\right)^{\frac{c}{a_n}} \frac{a_n}{c},
\]
and note that the first term is positive, bounded, and strictly decreasing. 

For part \textit{(b)}, consider for simplicity the case $p=2$. The proof for general $p$ will follow from an identical argument.
\begin{align*}
   \text{TV}  (P_X, & P_Y)  =\int\int\frac{1}{2}\left\vert p_{x_1}(x_1) p_{x_2}(x_2) - p_{y_1}(x_1) p_{y_2}(x_2) \right\vert \di x_1 \di x_2 \\
   & =\int\int\frac{1}{2}\left\vert (p_{x_1}(x_1) - p_{y_1}(x_1)) p_{x_2}(x_2) + (p_{x_2}(x_2) -p_{y_2}(x_2)) p_{y_1}(x_1)\right\vert \di x_1 \di x_2 \\
   &\leq \int\int\frac{1}{2}\left\vert (p_{x_1}(x_1) - p_{y_1}(x_1))  \right\vert \di x_1 p_{x_2}(x_2) \di x_2 + \\
   & \hspace{1.5cm} +\int\frac{1}{2}\left\vert (p_{x_2}(x_2) -p_{y_2}(x_2)) \right\vert \int p_{y_1}(x_1) \di x_1 \di x_2 \\
    &\leq 2 C\frac{a_n}{c},
\end{align*}
where the first inequality is the triangle inequality, the second one is part \textit{(a)}. 
\end{proof}

\subsection{Auxiliary Results on Parametric Bootstrap}\label{sec:natural gradient}

We recall Proposition 1 of \cite{fong2024asymptotics}. This proposition provides a condition that implies Assumption~\ref{ass:param_acid} c), from Section~\ref{sec:param_boot} in the main manuscript.

\begin{proposition}\label{prop:fongyiu}(\cite{fong2024asymptotics}) For $X\sim P_{\theta}$, suppose there exists non-negative constants $B,C<\infty$ such that the following holds for all $\theta\in\Theta \subseteq\mathbb{R}$,
\[
\mathbb{E}(Z(X, \theta)^4) \leq B+C\theta^4,
\]
then Assumption~\ref{ass:param_acid} c) is satisfied. 
\end{proposition}

\section{Proofs} \label{sec:Proofs}

\subsection{Proof of Section~\ref{sec:acid}}
\begin{proof}[Proof of Theorem~\ref{thm:asym.exch}]
The state space $E$ is assumed to be a Polish space, that is a separable topological space, metrizable by a complete metric, and $\mathcal{E}$ the corresponding Borel $\sigma$-algebra. Therefore, there exists a countable subclass $\mathcal{A} \subseteq \mathcal{E}$ that generates $\mathcal{E}$, that is, $\sigma(\mathcal{A})=\mathcal{E}$, and is probability determining, which means that if two probability measures agree on $\mathcal{A}$, then they must also agree on $\mathcal{E}$. Indeed, let us consider the class of finite intersections of open balls with rational radius ($E$ is completely metrizable) centered at the points of a countable dense subset ($E$ is separable). We denote this class by $\mathcal{A}$, which is countable. Because $E$ is separable, every open set is a countable union of open balls, hence also of elements of $\mathcal{A}$. Therefore, being $\mathcal{E}$ the smallest $\sigma$-algebra generated by the open sets, it is also generated by $\mathcal{A}$. Finally, the class $\mathcal{A}$ is closed under intersections by construction, hence it is a probability determining class from Dynkyn's Lemma.  \vspace{0.2cm} \\
For any $A \in \mathcal{A}$, $(\alpha_n(\cdot,A))_{n\geq 0} = ( \mathbb{P}(X_n \in A| \mathcal{G}_{n-1})) )_{n\geq 0}$ forms an almost $\mathcal{G}$-supermartingale. Therefore, it converges almost surely by Theorem~\ref{th:robsig} and the assumption that $\sum_{n=0}^\infty  \xi_n  < \infty$ a.s. Being the class $\mathcal{A}$ countable, the exceptional null sets can be grouped into a unique null set, denoted $N_{1}$. 
Therefore, the sequence of probability measures $(\alpha_n(\omega,\cdot ))_{n\geq 0}$ converges over all elements $A\in \mathcal{A}$ of a determining class, for all $\omega \in N_{1}^{c}$. 
If it can also be shown that that the sequence $(\alpha_n(\cdot,\cdot))_{n\geq 0}$ is a.s. tight,
which means that there exist a null set $N_{2}\in \mathcal{E}$ such that for all $\omega \in N_{2}^{c}$ the sequence $(\alpha_n(\omega,\cdot))_{n\geq 0}$ is tight. Then, for all $\omega \in (N_{1}\cup N_{2})^{c}$, the sequence of probability measures $(\alpha_n(\omega,\cdot))_{n\geq 0}$ converges for all $A$ in a probability determining class and is tight, hence converges weakly to a probability measure $\alpha(\omega,\cdot)$, 
from Theorem 5.1 and subsequent Corollary of \cite{bil99}, hence the result of Theorem~\ref{thm:asym.exch} follows. \vspace{0.2cm} \\
Therefore, it remains to show that $(\alpha_n(\cdot,\cdot))_{n\geq 0}$ is a.s. tight.    
Fix $\epsilon>0$ arbitrary, and, for every $j\in \mathbb{N}$, set $\delta_{j} := \epsilon 2^{-2j}$. Moreover, for every $m\in \mathbb{N}$, let $\bar{\alpha}_{m}$ denote the mean measure of $\alpha_{m}$, defined as 
$$\bar{\alpha}_{m}(A):= \int_{\Omega} \alpha_{m}(\omega,A) \mathbb{P}(d\omega).$$ Because $E$ is a Polish space, $\bar{\alpha}_{m}$ is a tight probability measure for every $m\in \mathbb{N}$, from Theorem 1.3 of \cite{bil99}. Therefore, for every $m\in \mathbb{N}$, there exists a sequence of compact subsets $(K_{m,j})_{j\in \mathbb{N}}$ of $E$, s.t. $\bar{\alpha}_{m}(K_{m,j}^{c})\leq \delta_{j}=\epsilon 2^{-2j}$, for all $m$ and $j\in \mathbb{N}$.  For each $m\in \mathbb{N}$, by denoting with $\mathcal{P}(\mathcal{E})$ the set of probability measures on $(E,\mathcal{E})$, we can define the sequence of sets
\begin{equation*}
    \Theta_{m}:=\{ \theta \in \mathcal{P}(\mathcal{E}) : \theta(K^{c}_{m,j} )\leq 2^{-j}, \forall j \geq 1\}.
\end{equation*}
For each $m$, $\Theta_{m}$ is a subset of probability measures on $E$, whose elements are tight measures (using the same approximating compact sets as $\bar{\alpha}_{m}$ for every level of approximation $j$). Also, let us notice that if, for a specific $\omega \in \Omega$,  there exists $m \in \mathbb{N}$ (which may depend on $\omega$) s.t. $\forall k\geq m, \ \alpha_{k}(\omega, \cdot) \in \Theta_{m}$, then the sequence $(\alpha_{n}(\omega, \cdot))_{n\in \mathbb{N}}$ is tight. Indeed, choose $\epsilon'>0$ arbitrary, and pick $j$ large enough s.t. $2^{-j}<\epsilon'$. Then, for all $k\geq m$, $\alpha_{k}(\omega, \cdot) \in \Theta_{m}$, so   $\alpha_{k}(\omega, K_{m,j}) \geq  1- 2^{-j} > 1-\epsilon'$. For each $k<m$, $\alpha_{k}(\omega,\cdot)$ is a tight measure ($E$ Polish), so there is a compact $K_{k}$ (dependent on $\omega$) s.t. $\alpha_{k}(\omega,K_{k})>1-\epsilon'$. Finally, $K= K_{m,j} \cup (\cup_{k=1}^{n} K_{k})$ is a compact set (being finite union of compact sets), s.t. $\alpha_{n}(\omega,K)>1-\epsilon'$ for all $n\in \mathbb{N}$. Therefore, in order to show that $(\alpha(\cdot,\cdot))_{n\geq 0}$ is a.s. tight, it is enough to show that 
\begin{equation}\label{eq:proof.tight0}
    \mathbb{P}\left(\{\omega \in \Omega : \exists m \in \mathbb{N} \ \text{s.t.} \ \forall k\geq m, \ \alpha_{k}(\omega, \cdot) \in \Theta_{m} \} \right) = 1.
\end{equation}

\noindent The left hand side of equation~\eqref{eq:proof.tight0} can be lower bounded as follows,
\begin{align} \label{eq:proof.tight1}
   \mathbb{P} &\left(\{\omega \in \Omega :   \exists m \in \mathbb{N} \ \text{s.t.} \ \forall k\geq m, \ \alpha_{k}(\omega, \cdot) \in \Theta_{m} \} \right)  \nonumber \\
   & = \mathbb{P}\left(\cup_{m\in \mathbb{N}} \cap_{k \geq m} \{\omega \in \Omega :  \ \alpha_{k}(\omega, \cdot) \in \Theta_{m} \} \right)  \nonumber
    \\
   & = \mathbb{P}\left(\cup_{m\in \mathbb{N}} \cap_{k \geq m} \{\omega \in \Omega :  \ \alpha_{k}(\omega, K_{m,j}^{c}) \leq 2^{-j}, \forall j \in \mathbb{N} \} \right) \nonumber \\
   & = \mathbb{P}\left( \cap_{j\in \mathbb{N}} \cup_{m\in \mathbb{N}} \cap_{k \geq m} \{\omega \in \Omega :  \ \alpha_{k}(\omega, K_{m,j}^{c}) \leq 2^{-j} \} \right) \nonumber \\
   & = 1- \mathbb{P}\left( \cup_{j\in \mathbb{N}} \cap_{m\in \mathbb{N}} \cup_{k \geq m} \{\omega \in \Omega :  \ \alpha_{k}(\omega, K_{m,j}^{c}) \geq 2^{-j} \} \right) \nonumber \\
   & \geq 1- \sum_{j\in \mathbb{N}}\mathbb{P}\left(  \cap_{m\in \mathbb{N}} \cup_{k \geq m} \{\omega \in \Omega :  \ \alpha_{k}(\omega, K_{m,j}^{c}) \geq 2^{-j} \} \right), 
\end{align}
where the second equality follows from the definition of $\Theta_{m}$. The $j$-th term in the last sum can be upper bounded as follows,
\begin{align*}
    & \mathbb{P}\left(  \cap_{m\in \mathbb{N}} \cup_{k \geq m} \{\omega \in \Omega :  \ \alpha_{k}(\omega, K_{m,j}^{c}) \geq 2^{-j} \} \right)  \\
    & = \lim_{m \to \infty }\mathbb{P}\left(   \cup_{k \geq m} \{\omega \in \Omega :  \ \alpha_{k}(\omega, K_{m,j}^{c}) \geq 2^{-j} \} \right)  \\
    & = \lim_{m \to \infty }\mathbb{P}\left(    \lim_{n_{m} \to \infty} \cup_{k = m}^{n_{m}} \{\omega \in \Omega :  \ \alpha_{k}(\omega, K_{m,j}^{c}) \geq 2^{-j} \} \right)  \\
       & = \lim_{m \to \infty }  \lim_{n_{m} \to \infty}  \mathbb{P}\left(   \cup_{k = m}^{n_{m}} \{\omega \in \Omega :  \ \alpha_{k}(\omega, K_{m,j}^{c}) \geq 2^{-j} \} \right)  \\
        & = \lim_{m \to \infty }  \lim_{n_{m} \to \infty}  \mathbb{P}\left(    \{\omega \in \Omega :  \  \max_{m\leq k\leq n_{m}}\alpha_{k}(\omega, K_{m,j}^{c}) \geq 2^{-j} \} \right)  \\
         & \leq  \lim_{m \to \infty }  \lim_{n_{m} \to \infty}     2^{j} \left( \mathbb{E}(\alpha_{m}(\cdot, K^{c}_{m,j}) ) + \mathbb{E}( \sum_{k=m}^{n_{m}} \xi_{k} )\right)   \\
          % & =  \lim_{ m\to \infty }  \lim_{n_{m} \to \infty}  \left(   2^{j} \left( \bar{\alpha}_{m}( K^{c}_{m,j})  +  \sum_{k=m}^{n_{m}} \bar{\eta}_{k}(K^{c}_{m,j}) \right)  + \sum_{k=m}^{n_{m}} \bar{\beta}_{k}( K^{c}_{m,j})  \right) \\
         %  & =  \lim_{ m\to \infty }   \left(   2^{j} \left( \bar{\alpha}_{m}( K^{c}_{m,j})  +  \sum_{k=m}^{\infty} \bar{\eta}_{k}(K^{c}_{m,j}) \right)  + \sum_{k=m}^{\infty} \bar{\beta}_{k}( K^{c}_{m,j})  \right) \\
            & =  \lim_{ m\to \infty }      2^{j}  \bar{\alpha}_{m}( K^{c}_{m,j})  +  2^{j}  \lim_{ m\to \infty } \sum_{k=m}^{\infty}  \mathbb{E} (\xi_{k} )  \\
          %   & \leq  \lim_{ m\to \infty }      2^{j}  \bar{\alpha}_{m}( K^{c}_{m,j})  +  2^{j}  \lim_{ m\to \infty } \sum_{k=m}^{\infty} \bar{\eta}_{k}(\mathcal{X})  + \lim_{ m\to \infty } \sum_{k=m}^{\infty} \bar{\beta}_{k}(\mathcal{X})   \\
              & =  \lim_{ m\to \infty }      2^{j}  \bar{\alpha}_{m}( K^{c}_{m,j})     \\
               & \leq   \lim_{ m\to \infty }      2^{j}  \epsilon 2^{-2j}  = \epsilon 2^{-j},
\end{align*}
where the first inequality  follows from the maximal inequality for almost supermartingales, Proposition \ref{eq:max.ineq.robb}, and unconditioning over $\mathcal{G}_{m}$, which means integrating both sides of the inequality over $E$. The second to last equality follows from the assumption that $\mathbb{E}(\sum_{k=1}^{\infty} \xi_{k} )<\infty$%a.s.
, which in turns implies that $\sum_{k=1}^{\infty} \mathbb{E}( \xi_{k} )<\infty$, due to monotone convergence and $\xi_{n}$ being non-negative random variables. Finally, the last inequality follows by the specific choice of $K_{m,j}$. \vspace{0.2cm}  \\
% Also, in the following line, $\lim_{ m\to \infty } \sum_{k=m}^{\infty}  \mathbb{E} (\eta_{k}(\cdot, K^{c}_{m,j}) ) = 0$, and similarly for the term involving $\beta_{k}$, from the summability assumptions that we impose on $\eta$ and $\beta$. Specifically, we should assume something like $\sum_{k=1}^{\infty}\eta_{k}(\omega, A) <+\infty$, $\forall A\in \mathcal{E}$  and for all $\omega$ in a set of probability one. 
%If it turns out that they are measures, it might be enough to assume $\sum_{k=1}^{\infty}\eta_{k}(\omega, E) <+\infty$ a.s. or possibly just $\sum_{k=1}^{\infty}\bar{\eta_{k}}(E) <+\infty$ for the mean measures of $\eta_{k}$. If they are signed measured, possibly an assumption on the TV norms could be enough, like $ \sum_{k=1}^{\infty} ||\bar{\eta}_{k}||_{TV} = \sup_{A\in \mathcal{E}}|\bar{\eta_{k}}(A)| < +\infty$, which implies $\lim_{ m\to \infty } \sum_{k=m}^{\infty} ||\bar{\eta}_{k}||_{TV}  = 0 $. \\
%If $\eta_{k}$ is a constant, then assuming summability $\sum_{k=1}^{\infty} \eta_{k}<+\infty$ is enough for $\lim_{m\to \infty} \sum_{k=m}^{\infty} \eta_{k} = 0$. Similarly, for $\beta_{k}$.\\
Using this upper bound in \eqref{eq:proof.tight1},
\begin{align*}
    \mathbb{P}\left(\{\omega \in \Omega : \exists m \in \mathbb{N} \ \text{s.t.} \ \forall k\geq m, \ \alpha_{k}(\omega, \cdot) \in \Theta_{m} \} \right)    \geq 1- \sum_{j\in \mathbb{N}} \epsilon 2^{-j} = 1 - \epsilon,
\end{align*}
and, because $\epsilon$ is arbitrary, this implies \eqref{eq:proof.tight0}, hence it concludes the proof.
\end{proof} \vspace{0.4cm}

\begin{proof}[Proof of Corollary~\ref{corollary1}]
From Theorem~\ref{thm:asym.exch}, for all $\omega$ in a set of probability one, we have that for all $f:E\to\mathbb{R}$ continuous and bounded,
\begin{equation}\label{proof.corollary}
   \lim_{n\to\infty} \mathbb{E}(f(X_{n+1})|\mathcal{G}_{n})(\omega) \to \int_{\mathbb{R}}f(x) \alpha(\omega,dx),
\end{equation}
where   $\mathbb{E}(f(X_{n+1})|\mathcal{G}_{n})(\omega) = \int_{\mathbb{R}}f(x_{n+1})\mathbb{P}(X_{n+1}\in dx_{n+1} |\mathcal{G}_{n})(\omega)$. On the same set of probability one, we have that for all $f:E\to\mathbb{R}$ continuous and bounded,
\begin{align*}
    \lim_{M\to \infty} \mathbb{E}\left(f(X_{n+M})  |\mathcal{G}_{n}\right) & = \lim_{M\to \infty} \mathbb{E}\left[ \mathbb{E}\left[ f(X_{n+M}) | \mathcal{G}_{n+M-1}\right] \bigg| \mathcal{G}_{n}\right]  \\
    & = \mathbb{E}\left[ \lim_{M\to \infty}\mathbb{E}\left[ f(X_{n+M}) | \mathcal{G}_{n+M-1}\right] \bigg| \mathcal{G}_{n}\right] \\
    &= \mathbb{E}\left[ \int_{\mathbb{R}}f(x) \alpha(\cdot,dx) \bigg| \mathcal{G}_{n}\right], 
\end{align*}  
%\red{[I have a big doubt: in the second row, can you bring in the limit? or least, shouldn't you account for the fact that he integration distribution is also changing with $M$?? If you use in expectation everything works, but the integral is changing as well, no? I mean the expectation outside depends on $M$.]}
where the first equality follows from the fact that $\mathcal{G}_{n}\subset\mathcal{G}_{n+M-1}$, the second one from the conditional bounded convergence theorem, and the last one from \eqref{proof.corollary}. Hence, 
 $\mathbb{P}\left(X_{n+M}\in \cdot  |\mathcal{G}_{n}\right) \stackrel{w}{\to} \mathbb{E}[ \alpha(\cdot)|\mathcal{G}_{n}]$ almost surely.
\end{proof}\vspace{0.4cm}

%\subsection{Proofs of Section~\ref{sec:examples}}

\subsection{Proofs of Section~\ref{sec:param_boot}}\label{sec:proof_param_boot}

Proposition~\ref{prop:param_boot} proof is done in two steps. First we show that $(\hat{\theta})_{n\geq 1}$ converge $a.s.$; then we prove that Definition~\ref{def:acid.pred} holds.  The following Lemma is trivial and surely known but we did not find a reference for it. 

\begin{lemma}\label{eq:summability} Let $(Z_n)_{n\geq 1}$ be a sequence of a.s. finite r.v.s. such that $Z_n \stackrel{\text{a.s.}}{\to} Z$, for some a.s. finite r.v. $Z$. Let $(\eta_n)_{n \geq 1}$ be a deterministic positive summable sequence $\sum_{n=1}^\infty \eta_n <\infty$. Then, it holds that $\sum_{n=1}^\infty \eta_n Z_n < \infty$ a.s. 
\end{lemma}

\begin{proof} Without loss of generality, we can assume that all $Z_{n}$ are a.s. non-negative. Choose $\omega'\in \{\omega \in \Omega: Z_n(\omega) \to Z(\omega), Z(\omega)<\infty\}$. Then, for every $\epsilon>0$, there exists $N_{\omega'}$ such that for all $n>N_{\omega'}$, $Z_n(\omega')< Z(\omega')+\epsilon$. Therefore,
\begin{equation*}
\sum^{\infty}_{n=N_{\omega'}} \eta_n Z_n(\omega) < (Z(\omega')+\epsilon)\sum^{\infty}_{n=N_{\omega'}} \eta_n <\infty, 
\end{equation*}
because $\epsilon$ is arbitrary, $(\eta_n)_{n \geq 1}$ summable, and $Z(\omega')<\infty$, the statement follows.
\end{proof}

\begin{proof}[Proof of Proposition~\ref{prop:param_boot}] First note that $(\hat{\theta})_{n\geq 1}$ is a martingale
\[
\mathbb{E}(\hat{\theta}_n| \mathcal{G}_{n-1}) = \hat{\theta}_{n-1} + \eta_n \mathcal{I}(\hat{\theta}_{n-1})^{-1} \mathbb{E}( s(x_n, \hat{\theta}_{n-1})| \mathcal{G}_{n-1}) = \hat{\theta}_{n-1},
\]  
where the second equality follows from standard properties of the score function. In addition, 
\begin{align*}
\lim_{n \to \infty} \mathbb{E}(\hat{\theta}^2_n) & = \lim_{n \to \infty} \sum_{i=1}^n \mathbb{E}(\mathbb{E}((\hat{\theta}_i-\hat{\theta}_{i-1})^2| \mathcal{G}_{i-1}))= \sum_{i=1}^\infty \eta^2_i \mathbb{E}(\mathbb{E}(Z(x_i, \hat{\theta}_{i-1})^2| \mathcal{G}_{i-1})) \\
& = \sum_{i=1}^\infty \eta^2_i \mathbb{E}( Z(x_i, \hat{\theta}_{i-1})^2) <\infty,
\end{align*}
where the first equality follows by martingale properties, the last inequality by Assumption~\ref{ass:param_acid} c). Hence, $(\hat{\theta})_{n\geq 1}$ converges a.s. by Doob's martingale convergence theorem. 

\noindent Now, we show that $(X_n)_{n \geq 1}$ is a.c.i.d. Recall that $\alpha_{n}(A) = \int_{A} p(x;\hat{\theta}_{n})\di x$. 
\begin{align}\label{eq:param1}
  \mathbb{E}(\alpha_{n+1}(A) -\alpha_{n}(A)  | \mathcal{G}_{n}) =& \mathbb{E}\left( 
  \int_{A} p(x;\hat{\theta}_{n+1})  - p(x;\hat{\theta}_{n})\di x| \mathcal{G}_{n}\right)\nonumber \\ 
  =  &  \mathbb{E}\left( 
  \int_{A}p(x;\hat{\theta}_{n}) \left( \frac{p(x;\hat{\theta}_{n+1}) }{p(x;\hat{\theta}_{n})}-1\right)\di x| \mathcal{G}_{n}\right).
\end{align}

\noindent Let us look at the integrand. A Taylor expansion of $p(x;\theta)$ around $\hat{\theta}_{n}$ is 
\[
p(x;\theta)= p(x;\hat{\theta}_{n}) + p'(x;\hat{\theta}_{n}) (\theta - \hat{\theta}_{n}) + \frac{p''(x;\nu_n)}{2}(\theta - \hat{\theta}_{n})^2,
\]
where $\nu_n$ is a point between $\theta$ and $\hat{\theta}_{n}$. We can rewrite the integrand as 
\begin{align*}
\frac{p(x;\hat{\theta}_{n+1}) }{p(x;\hat{\theta}_{n})}-1 &=(\hat{\theta}_{n+1}-\hat{\theta}_{n})  \frac{p'(x;\hat{\theta}_{n})}{p(x;\hat{\theta}_{n})}+ \frac{(\hat{\theta}_{n+1}-\hat{\theta}_{n})^2}{2}\frac{p''(x;\nu_{n})}{p(x;\hat{\theta}_{n})} \\
&= (\hat{\theta}_{n+1}-\hat{\theta}_{n}) s(x;\hat{\theta}_{n}) + \frac{(\hat{\theta}_{n+1}-\hat{\theta}_{n})^2}{2}\frac{p''(x;\nu_{n})}{p(x;\hat{\theta}_{n})}.
\end{align*}

\noindent Plugging in the above in \eqref{eq:param1} results in two separate integrals. Recalling that $(\hat{\theta}_{n+1}-\hat{\theta}_{n})=\eta_{n+1} Z (x_{n+1}; \hat{\theta}_n)$, the first term can be rewritten as
\begin{align*}
\mathbb{E} & \Big( 
\int_{A}p(x;\hat{\theta}_{n}) \eta_{n+1} Z(x_{n+1};\hat{\theta}_{n}) s(x;\hat{\theta}_{n}) \di x| \mathcal{G}_{n}\Big)= \\
& =\eta_{n+1} \int_{\mathbb{R}} p(x_{n+1};\hat{\theta}_{n})\int_{A}p(x;\hat{\theta}_{n}) \mathcal{I}(\hat{\theta}_{n})^{-1}s(x_{n+1};\hat{\theta}_{n}) s(x;\hat{\theta}_{n}) \di x \di x_{n+1}\\
&=\eta_{n+1}\mathcal{I}(\hat{\theta}_{n})^{-1}\int_{A} p(x;\hat{\theta}_{n}) s(x;\hat{\theta}_{n}) \int_{\mathbb{R}} p(x_{n+1};\hat{\theta}_{n})  s(x_{n+1};\hat{\theta}_{n})  \di x_{n+1} \di x =0,
\end{align*}
where the last equality follows from $\int_{\mathbb{R}} p(x_{n+1};\hat{\theta}_{n})  s(x_{n+1};\hat{\theta}_{n})  \di x_{n+1}=0$. The second integral can be written as 
\begin{align*}
\mathbb{E} & \Big( 
\int_{A}p(x;\hat{\theta}_{n}) \frac{(\hat{\theta}_{n+1}-\hat{\theta}_{n})^2}{2}\frac{p''(x;\nu_{n})}{p(x;\hat{\theta}_{n})} \di x| \mathcal{G}_{n}\Big)=\\
&=\int_{\mathbb{R}} p(x_{n+1};\hat{\theta}_{n}) \int_A p(x;\hat{\theta}_{n}) \frac{(\hat{\theta}_{n+1}-\hat{\theta}_{n})^2}{2}\frac{p''(x;\nu_{n})}{p(x;\hat{\theta}_{n})} \di x \, \di x_{n+1}\\
&= \int_{\mathbb{R}} p(x_{n+1};\hat{\theta}_{n}) \frac{(\hat{\theta}_{n+1}-\hat{\theta}_{n})^2}{2} \int_A p(x;\hat{\theta}_{n}) \frac{p''(x;\nu_{n})}{p(x;\hat{\theta}_{n})} \di x \, \di x_{n+1}\\
&\leq C \int_{\mathbb{R}} p(x_{n+1};\hat{\theta}_{n}) \frac{(\hat{\theta}_{n+1}-\hat{\theta}_{n})^2}{2} \di x_{n+1},
\end{align*}
where $C<\infty$ is a constant  that exists by Assumption~\ref{ass:param_acid} d). 

\noindent $(X_n)_{n \geq 1}$ is a.c.i.d. with $\xi_n$ given  by 
\begin{equation}\label{eq:param_boot_xi_n}
C \frac{\eta_{n+1}^2}{2}\int_{\mathbb{R}} p(x_{n+1};\hat{\theta}_{n}) Z^2(x_{n+1};\hat{\theta}_{n}) \di x_{n+1} = C \frac{\eta_{n+1}^2}{2}\mathcal{I}(\hat{\theta}_n)^{-1} \leq  C \frac{\eta_{n+1}^2}{2\epsilon}:=\xi_n 
\end{equation}
where $\epsilon$ is from Assumption~\ref{ass:param_acid} b). Because $\eta_{n}^{2}$ is summable, so is $\xi_{n}$.
%\noindent Now, it holds that $\sum_{n=0}^{\infty} \xi_n< \infty$ a.s. To see this, note that, by continuous mapping theorem, $\mathcal{I}(\hat{\theta}_n)^{-1} \stackrel{\text{a.s.}}{\to} \mathcal{I}(\theta)^{-1}$ , for some possibly random $\mathcal{I}(\theta)^{-1}$, and with Assumption~\ref{ass:param_acid} b) guaranteeing that it does not diverge. The summability follows by Lemma~\ref{eq:summability}. The results follows by Proposition~\ref{thm:asym.exch}.
\end{proof}

\begin{proof}(Proof of Proposition~\ref{prop:param_boot_constrained}). It is not necessary to show that $(\tilde{\theta}_{n})_{n\geq 1}$ converges almost surely. We follow the same steps as in the proof of Proposition~\ref{prop:param_boot}, up to the definition of $\xi_n$ in \eqref{eq:param_boot_xi_n}. At this point, the information function $\mathcal{I}(\theta)$ is assumed to be continuous and greater than $\epsilon$ (Assumption ~\ref{ass:param_acid} b)), so it achieves a minimum, greater than $\epsilon$, on the compact set $\Theta=[\underline{\theta},\bar{\theta}]$. Therefore, we can upper bound equation~\eqref{eq:param_boot_xi_n} by $C \frac{\eta_{n+1}^2}{2\epsilon}$, which is summable under the assumptions of Proposition~\ref{prop:param_boot_constrained}. %it still holds that $\sum_{n=1}^{\infty} \xi_n < \infty$ almost surely, since $\epsilon < \mathcal{I}(\theta) < \infty$ for all $\theta \in [\underline{\theta},\bar{\theta}]$, and invoking the continuous mapping theorem is unnecessary. The remainder of the argument proceeds identically.
\end{proof}

\begin{proof}(Proof of Example~\ref{ex:studentt}) 
(Student-t distribution)  Let $P_{\theta}$ be a location-scale Student-$t$ distribution with location $\theta$ and known scale $\tau$ and degrees of freedom $\nu$, with density $p(x;\theta)$  with density
\begin{equation*} %\label{eq:student0}
p(x;\theta ) = \frac{\Gamma\left(\frac{\nu+1}{2}\right)}{\sqrt{\nu \pi} \tau \Gamma\left(\frac{\nu}{2}\right)} \left(1 + \frac{(x - \theta)^2}{\nu \tau^2} \right)^{-\frac{\nu+1}{2}}.
\end{equation*}
%defined in equation~\eqref{eq:student0}. 
We need to show that $\int p''(x;\theta)\di x <\infty$ to show Assumption~\ref{ass:param_acid} d).
%\eqref{eq:student0} is 
%\begin{align} \label{eq:student0}
%p_{\theta}(x )  = \frac{\Gamma\left(\frac{\nu+1}{2}\right)}{\sqrt{\nu \pi} \tau \Gamma\left(\frac{\nu}{2}\right)} \left(1 + \frac{(x - \theta)^2}{\nu \tau^2} \right)^{-\frac{\nu+1}{2}}.
%\end{align}
\noindent The first derivative $p'(x;\theta)$ of $p(x;\theta)$ is
\begin{align*} %\label{eq:student1}
  \frac{d p(x;\theta)}{d \theta}&  %\frac{\Gamma\left(\frac{\nu+1}{2}\right)}{\sqrt{\nu \pi} \tau \Gamma\left(\frac{\nu}{2}\right)}\left(\frac{\nu+1}{2}\right)\left(1 + \frac{(x - \theta)^2}{\nu \tau^2} \right)^{-\frac{\nu+1}{2}-1} \frac{2(x-\theta)}{\nu\tau^2} \\
  = p(x;\theta) \left(\frac{\nu+1}{\nu \tau^2}\right)\left(1 + \frac{(x - \theta)^2}{\nu \tau^2} \right)^{-1} (x-\theta).
\end{align*}
Differentiating again and simplifying, the second derivative $p''(x;\theta)$ is 
\begin{align*}
\frac{d^{2} p(x;\theta)}{d^{2}\theta} % & = \frac{d p(x;\theta)}{d \theta} \left(\frac{\nu+1}{\nu \tau^2}\right)\left(1 + \frac{(x - \theta)^2}{\nu \tau^2} \right)^{-1} (x-\theta) \\
%& \ \ \ + p(x;\theta) \left(\frac{\nu+1}{\nu \tau^{2}}\right)  \left[  \left( 1 + \frac{(x - \theta)^2}{\nu \tau^2} \right)^{-2} \frac{2 (x - \theta)^2 }{\nu\tau^{2}} 
%  - \left( 1 + \frac{(x - \theta)^2}{\nu \tau^2} \right)^{-1} \right] \\
%& =  p(x;\theta) \left(\frac{\nu+1}{\nu \tau^2}\right)^{2}\left(1 + \frac{(x - \theta)^2}{\nu \tau^2} \right)^{-2} (x-\theta)^{2} \\
%& \ \ \ + p(x;\theta) \left(\frac{\nu+1}{\nu \tau^{2}}\right)  \left[  \frac{2 (x - \theta)^2 }{\nu\tau^{2}} \left( 1 + \frac{(x - \theta)^2}{\nu \tau^2} \right)^{-2} - \left( 1 + \frac{(x - \theta)^2}{\nu \tau^2} \right)^{-1} \right] \\
& =  p(x;\theta) \left(\frac{(\nu+1)(\nu+3)}{(\nu \tau^2)^{2}}\right)\left(1 + \frac{(x - \theta)^2}{\nu \tau^2} \right)^{-2} (x-\theta)^{2} +\\
& \hspace{1cm} - p(x;\theta) \left(\frac{\nu+1}{\nu \tau^{2}}\right)   \left( 1 + \frac{(x - \theta)^2}{\nu \tau^2} \right)^{-1}.
\end{align*}
In order to show that $\int p''(x;\theta) \di x < C < \infty$, for some constant $C$, we can just show that, after dropping constants and changing variable to $y=x-\theta$, the following two integrals are both bounded, 
\begin{align*}
  \int_{-\infty}^{\infty}   p(y;\theta) \left(1 + \frac{y^2}{\nu \tau^2} \right)^{-2} y^{2} dy & \propto \int_{-\infty}^{\infty}   p(y;\theta) \frac{y^{2}}{(\nu \tau^2 +y^{2})^{2}} \di y \\
  &\leq \frac{1}{4\nu \tau^{2}}\int_{-\infty}^{\infty}   p(y;\theta)  \di y = \frac{1}{4\nu \tau^{2}},
\end{align*}
where the inequality follows from the fact that the function $f(y) =\frac{y^{2}}{(\nu \tau^2 +y^{2})^{2}}$ is positive and bounded, with maxima at $y=\pm\sqrt{\nu \tau^2}$. The second integral,
\begin{align*}
    \int_{-\infty}^{\infty}  p(y;\theta)    \left( 1 + \frac{y^2}{\nu \tau^2} \right)^{-1}  \di y \propto \int_{-\infty}^{\infty}   p(y;\theta)     \frac{1}{y^2+ \nu \tau^2}   \di y \leq   \int_{-\infty}^{\infty}   p(y;\theta) \frac{1}{\nu \tau^2}       \di y = \frac{1}{\nu \tau^2}.
\end{align*} 
\end{proof}

\subsection{Proofs of Section~\ref{sec:kernel}}

To prove Proposition~\ref{prop:kernel.cid}, we first consider the sequence of densities $(p_n)_{n\geq 0}$ corresponding to the sequence of measures $(\alpha_n)_{n \geq 0}$ defined in \eqref{eq:kernel_Fn} such that $\alpha_n(A) =\int_A p_n(x) \di x$ holds for all $A\in \mathcal{B}(\mathbb{R})$ and $n\geq 0$.  Existence of such densities is guaranteed by construction. $p_n(x)$ can be similarly defined recursively: 
\begin{equation} \label{eq:kernel_densities}
p_{n}(x)=\frac{1}{n}\sum_{i=1}^{n-1} K_i(x-X_i) =\frac{n-1}{n}p_{n-1}(x)+\frac{1}{n} K_n(x-X_{n}),
\end{equation}
where recall that $K_n(x)=\frac{1}{h_n}K(\frac{x}{h_n})$. %The following result is the counterpart of \cite{west1991kernel} Theorem 1 to a recursive kernel estimator.

\begin{lemma}\label{lem:kernel.density} Consider the sequence $(p_n)_{n \geq 0}$ defined by \eqref{eq:kernel_densities}. Suppose that $\mathbb{E}(p_{n+1}(x)|\mathcal{G}_{n})=p_{n}(x)$ holds $a.s.$ for all $x \in \mathbb{R}$, then it must be that $K(x)=\delta(x)$, the Dirac delta function. 
\end{lemma}

\begin{proof} Fix $x \in \mathbb{R}$
		\begin{align}
 \mathbb{E}\left( p_{n+1}(x) |\mathcal{G}_{n}\right) 
		&=  \mathbb{E}\left( \frac{n}{n+1}p_{n}(x)+\frac{1}{n+1} K_{n+1}\left(x-X_{n+1}\right) |\mathcal{G}_{n}\right) \nonumber \\
		&=  p_{n}(x)+\frac{1}{n+1}  \left[\mathbb{E}\left( K_{n+1}\left(x-X_{n+1}\right)   |\mathcal{G}_{n}\right) -p_{n}(x)\right]. \nonumber
	\end{align}
For $\mathbb{E}(p_{n+1}(x)|\mathcal{G}_{n})=p_{n}(x)$ to hold, it must be that $\mathbb{E}\left( K_{n+1}\left(x-X_{n+1}\right)   |\mathcal{G}_{n}\right)=p_{n}(x)$. In other words,
\begin{align*}
    \mathbb{E} & \left( K_{n+1}\left(x-X_{n+1}\right)   |\mathcal{G}_{n}\right)-p_{n}(x) = \\
    & =\int K_{n+1}\left(x-x_{n+1}\right)\left( \frac{1}{n} \sum_{i=1}^n K_i (x_{n+1} -X_i) \right)\di x_{n+1}-p_{n}(x) \\
    &=\frac{1}{n}\sum_{i=1}^n \left[  \int K_{n+1}\left(x-x_{n+1}\right) K_i (x_{n+1} -X_i) \di x_{n+1}-K_i(x-X_i)\right].
\end{align*}
By the same argument in the proof of  \cite{west1991kernel} Theorem 1, the term above equals zero if and only if all the terms in the square bracket equal zero. The reason is that it must hold for all $n$ a.s., so for all samples $x_{1:n}$ except for some measure zero set. Letting subsets of the data go to infinity, the corresponding kernel density goes to $0$, which implies that all the summands must be equal to zero. We need to show that
\begin{align*}
\int & K_{n+1}\left(x-x_{n+1}\right)  K_i (x_{n+1} -X_i) \di x_{n+1}-K_i (z) 
=\int K_{n+1}(z-y) K_i(y) \di y - K_i (z)=0,
\end{align*}
where the first equality follows by substitution. The convolution of densities above has as characteristic function the product of characteristic function. The equation can be rewritten in terms of characteristic functions corresponding to the densities above
\[
\gamma_{n+1}(t)\gamma_i(t) -\gamma_i(t)=0.
\]
$\gamma_i(t)$ cannot be equal to $0$ for all $t$, which implies that $\gamma_{n+1}(t)=1$. This completes the proof.
\end{proof}

\begin{proof}[Proof of Proposition~\ref{prop:kernel.cid}]
	In order to prove that the sequence is c.i.d., we need to check that $(\alpha_{n}(A))_{n\in \mathbb{N}}$ is a martingale for all $A \in \mathcal{E}$. It is actually sufficient to check the martingale property for sets of the form $A= (-\infty,t]$ for all $t \in \mathbb{R}$. Writing $ \alpha_{n+1}((-\infty,t])=F_{n+1}(t)$, we have that
	\begin{align}\label{eq:kernel_acid}
 \mathbb{E}\left( F_{n+1}(t) |\mathcal{G}_{n}\right) 
		&=  \mathbb{E}\left( \frac{n}{n+1}F_{n}(t)+\frac{1}{(n+1) h_{n+1}} \int_{-\infty}^{t} K\left(\frac{x-X_{n+1}}{h_{n+1}}\right) \di x |\mathcal{G}_{n}\right) \nonumber \\
		& =   \frac{n}{n+1}F_{n}(t)+\frac{1}{(n+1) h_{n+1}} \mathbb{E}\left(\int_{-\infty}^{t} K\left(\frac{x-X_{n+1}}{h_{n+1}}\right) \di x |\mathcal{G}_{n}\right)\nonumber \\
		&=  F_{n}(t)+\frac{1}{(n+1)}  \left[ \mathbb{E}\left(\int_{-\infty}^{t} \frac{1}{h_{n+1}} K\left(\frac{x-X_{n+1}}{h_{n+1}}\right)  \di x |\mathcal{G}_{n}\right) -F_{n}(t)\right]. \nonumber
	\end{align}
	In order to be a martingale, the second term needs to be equal to zero for all $n\in \mathbb{N}$. Therefore,
	\begin{align*}
	\mathbb{E}\left(\int_{-\infty}^{t} 	\frac{1}{h_{n+1}}  K\left(\frac{x-X_{n+1}}{h_{n+1}}\right) \di x |\mathcal{G}_{n}\right)  &= F_{n}(t)= \mathbb{P}(X_{n+1}\leq t|\mathcal{G}_{n}) \\
    &=\mathbb{E}(\mathbb{I}(X_{n+1}\leq t)|\mathcal{G}_{n}). 
	\end{align*}
	The integral on the left hand side is
	\begin{equation*}
	\int_{-\infty}^{t} 	\frac{1}{h_{n+1}} K\left(\frac{x-X_{n+1}}{h_{n+1}}\right) \di x = \int_{-\infty}^{\frac{t-X_{n+1}}{h_{n+1}}} K\left(u \right) \di u.
	\end{equation*}
	Taking the limit $h_{n+1} \to 0^{+}$, we see that if $t-X_{n+1}>0$, $\lim_{h_{n+1}\to 0^{+}} \int_{-\infty}^{\frac{t-X_{n+1}}{h_{n+1}}} K\left(u \right) \di u = \int_{-\infty}^{+\infty} K\left(u \right) \di u =1$. Instead, if $t-X_{n+1}<0$ $\lim_{h_{n+1}\to 0^{+}} \int_{-\infty}^{\frac{t-X_{n+1}}{h_{n+1}}} K\left(u \right) \di u = \int_{-\infty}^{-\infty} K\left(u \right) \di u =0$. Therefore, 
	\begin{equation*}
		\lim_{h_{n+1}\to 0^{+}} \int_{-\infty}^{\frac{t-X_{n+1}}{h_{n+1}}} K\left(u \right) \di u = \mathbb{I}(X_{n+1}\leq t).
	\end{equation*}
	a.s. with respect to $\mathbb{P}(\cdot|\mathcal{G}_{n})$ (given that the event $\{X_{n+1} = t\}$ has probability zero under this conditional distribution). Because this needs to hold for all $n\in \mathbb{N}$, the sequence of r.v.s is c.i.d. if for all $n\in\mathbb{N}$ $h_{n}\to0^+$, which corresponds to the case of empirical distribution function, $F_{n}(t)= \frac{1}{n}\sum_{i=1}^{n} \mathbb{I}(X_{i}\leq t)$.

\noindent To prove necessity, note that 
\begin{align*}
 \mathbb{E}\left( \alpha_{n+1}(A) |\mathcal{G}_{n}\right) 
		&= \alpha_{n}(A)+\frac{1}{(n+1)}  \left[ \mathbb{E}\left(\int_A K_{n+1}\left(x-X_{n+1}\right)  \di x |\mathcal{G}_{n}\right) -
        \alpha_{n}(A)\right]\\ 
        &=\alpha_{n}(A)+\frac{1}{(n+1)}  \left[ \int_A \mathbb{E}\left( K_{n+1}\left(x-X_{n+1}\right) |\mathcal{G}_{n}\right)\di x -
       \int_A p_{n}(x) \di x \right].
	\end{align*} 
    In order for the term in the squared bracket to be equal to zero for all $A \in \mathcal{B}(\mathbb{R})$, $$\mathbb{E}\left( K_{n+1}\left(x-X_{n+1}\right) |\mathcal{G}_{n}\right)=p_{n}(x)$$ must holds almost surely. The result follows by Lemma~\ref{lem:kernel.density}.
\end{proof}

\begin{proof}[Proof of Proposition~\ref{prop:kernel.acid}] First notice that the predictive distribution in \eqref{eq:kernel_Fn} can be written non-recursively as
	\begin{equation} \label{ex:kernel.predictive}
		\alpha_{n}(A)= \frac{1}{n h_{i}}\sum_{i=1}^{n}\int_{A} K\left( \frac{x- X_{i}}{h_{i}} \right)dx = \frac{1}{n}\sum_{i=1}^{n} \mu_i(A;X_i).
	\end{equation}
	Moreover, for any $A\in \mathcal{B}(\mathbb{R})$, by \eqref{eq:kernel_Fn}, we have that
	\begin{align*}
		\mathbb{E} (\alpha_{n+1}(A)-\alpha_{n}(A)|\mathcal{G}_n)= \frac{1}{n+1} \left[\mathbb{E} \left( \int_A \frac{1}{h_{n+1}} K\left(\frac{x-X_{n+1}}{h_{n+1}}\right) \di x \bigg|\mathcal{G}_n \right) - \alpha_{n}(A)\right].
	\end{align*}
	Because $X_{n+1}|\mathcal{G}_{n}\sim \alpha_{n}$ and then plugging \eqref{ex:kernel.predictive}, the expectation on the right hand side equals to
	\begin{align*}
		\mathbb{E} &\left( \int_A \frac{1}{h_{n+1}} K\left(\frac{x-X_{n+1}}{h_{n+1}}\right) \di x \bigg|\mathcal{G}_n \right)= \\
        &\hspace{1cm}=\int \int_A \frac{1}{h_{n+1}} K\left(\frac{x-x_{n+1}}{h_{n+1}}\right) \di x  \, \alpha_n(\di x_{n+1}) \\
		&\hspace{1cm}=\int \int_A \frac{1}{h_{n+1}} K\left(\frac{x-x_{n+1}}{h_{n+1}}\right) \frac{1}{n}\sum_{i=1}^n \frac{1}{h_i}  K\left(\frac{x_{n+1}-X_i}{h_{i}}\right) \di x \, \di x_{n+1} \\
		&\hspace{1cm}= \frac{1}{n}\sum_{i=1}^n  \int_A \int  \frac{1}{h_{n+1}} K\left(\frac{x-x_{n+1}}{h_{n+1}}\right) \frac{1}{h_i}  K\left(\frac{x_{n+1}-X_i}{h_{i}}\right) \di x_{n+1} \, \di x\\
		&\hspace{1cm}=\frac{1}{n}\sum_{i=1}^n \mu_{n+1*i}(A;X_i).
	\end{align*}
	Hence,
	\begin{align*}
		\mathbb{E} (\alpha_{n+1}(A)|\mathcal{G}_n)-\alpha_{n}(A)=& \frac{1}{n+1} \left(\frac{1}{n}\sum_{i=1}^{n} \mu_{n+1*i}(A;X_i) - \frac{1}{n}\sum_{i=1}^{n} 	\mu_{i}(A;X_i) \right)\\
		= & \frac{1}{n+1} \frac{1}{n}\sum_{i=1}^{n} \left( \mu_{n+1*i}(A;X_i) - 
		\mu_{i}(A;X_i) \right)\\
		\leq & \frac{1}{n+1} \frac{1}{n}\sum_{i=1}^n C \frac{h^d_{n+1}}{h^d_{i}}= \xi_{n},
	\end{align*}
where the last inequality follows by  \eqref{eq:kernel_tv}. This shows that $(X_n)_{n\geq 1}$ is $\mathcal{G}$-a.c.i.d.  with $\xi_n$ given in the last line above.
	
	\noindent The summability of $(\xi_n)_{n\geq 1}$ follows from  Theorem 3.29 and Theorem 3.25 in \cite{rudin1964principles}, which implies asymptotically exchangeability of $(X_n)_{n\geq 1}$ by Theorem~\ref{thm:asym.exch}.
\end{proof}

In addition, we include the details of an additional example (uniform kernel), which we briefly mentioned in the main manuscript.

\begin{example} (Uniform kernel) Let $K(x)=\frac{1}{2}$ for $|x| \leq1$, which implies $K_n(x)=\mathcal{U}(x;X_n-h_n,X_n+h_n)$, a uniform density with lower and upper bounds $(X_n-h_n,X_n+h_n)$. The convolution of uniform kernels $(K_{n+1} * K_{i})(x;X_i,h)$ can be computed analytically. For $h_{n+1}<h_i$,
		\begin{equation*}
			(K_{n+1} * K_{i})(x;X_i,h)=\begin{cases}
				\frac{x-X_{i} +h_{n+1}+h_i}{4h_{n+1} h_i} & \  X_i-(h_i+h_{n+1})\leq x < X_i-(h_i-h_{n+1})\\
				\frac{1}{2 h_i} & \ X_i-(h_i-h_{n+1})\leq x < X_i+(h_i-h_{n+1}) \\
				\frac{X_{i}-x+h_{n+1}+h_i}{4h_{n+1} h_i} & \ X_i+(h_i-h_{n+1})\leq x <  X_i+(h_i+h_{n+1}).
			\end{cases}
		\end{equation*}
		and $0$ if $x< X_{i}-(h_{i}+h_{n+1})$ or $x \geq X_{i} +(h_{i}+h_{n+1})$. This is the density of a trapezoidal distribution centered at $X_i$. 
	
	\noindent The total variation can be computed exactly, as it is given by a simple geometric argument, by integrating over the region in which the trapezoidal density is over the uniform one, or viceversa. Therefore, by symmetry of both densities around $X_{i}$,
	\begin{align*}
		TV(\mu_{n+1*i}(\cdot,X_i), \mu_{i} (\cdot,X_i)) = 2 \int_{X_{i}-(h_{i}+h_{n+1})}^{X_{i}-h_{i}} \frac{x-X_{i} +h_{n+1}+h_i}{4h_{n+1} h_i} \di x = \frac{1}{4} \frac{h_{n+1}}{h_i}.
	\end{align*}
\end{example}

\section{Simulations}

All code to reproduce the simulation study is available online at \url{https://github.com/lorenzocapp/acid_paper}. 
We include here some details for completeness. 

\subsection{Simulations: Data Generating Mechanism}\label{sec:data_generation}

\noindent \textit{Section~\ref{sec:ill_density}.} We repeat $100$ times the following steps: i) sample a random distribution $P$ (as described in the next paragraph), ii) given $P$, sample $n$ $i.i.d.$ values from $P$. 

In Step i), we generate a mixture as follows: a) sample number of components $n_{comp}\sim \mathcal{U}\mathcal{D}\{2,3\}$ ($\mathcal{U}\mathcal{D}$ stands for uniform discrete), b) sample the means of the kernel components, $(\mu_1,\ldots,\mu_{n_{comp}})$, each with a proposal distribution $\mathcal{U}(-3,3)$ and acceptance rule $|\mu_i-\mu_j| \geq 1.5$ for all $i\neq j$, c) sample the scale parameter of each mixture component $\sigma^2_i \sim \mathcal{U}(0.5,2)$ for $i=1,\ldots,n_{comp}$, d) sample the kernel type $K_i\sim \mathcal{U}\mathcal{D}\{\text{Gaussian},\text{Student-t (df=5)}\}$, for $i=1,\ldots,n_{comp}$, e) sample $u_i \sim \mathcal{U}(0,1)$ for $i=1,\ldots,n_{comp}$ and define weights of the mixture $(w_1,\ldots,w_{n_{comp}})=(u_1/\sum u_i, \ldots, u_{n_{comp}}/\sum u_i)$.

\noindent \textit{Section~\ref{sec:ill_regression}.} There are three data generating mechanisms. For all of them, we repeat $100$ times the following: i) generate a regression function $f$, ii) generate covariates in training data $i.i.d.$ $X_i \iidsim \mathcal{U}(0,5)$, $i=1,\ldots, n_{train}$, define covariates in test data on a regular grid $[0,4.95]$ of $100$ points, iii) given $f$, $Y_i = f(x_i) + \epsilon_i$, for $i=1,\ldots, n_{train}+n_{test}$.

\textit{Data generating mechanism 1.} $f\sim \mathcal{GP} (0, K)$, where $K$ has entries $k(x,y)=\exp\{||x-y||_2\}$ when $x\neq y$ and $k(x,x)=1+10^{-8}$ ($10^{-8}$ is the nugget). $\epsilon_i \iidsim \text{N}(0,1)$. 

\textit{Data generating mechanism 2.} $f\sim \mathcal{GP} (0, K)$, where $K$ has entries $k(x,y)=\exp\{||x-y||_2\}$ when $x\neq y$ and $k(x,x)=1+10^{-8}$ ($10^{-8}$ is the nugget). $\epsilon_i \iidsim \text{Student-t (df=5)}$. 

\textit{Data generating mechanism 3.} $f\sim \mathcal{GP} (0, K)$, where $K$ has entries $k(x,y)=0.8\exp\{||x-y||_2\}+0.2\exp\{||x-y||_1\}$ when $x\neq y$ and $k(x,x)=1+10^{-8}$ ($10^{-8}$ is the nugget). $\epsilon_i \iidsim \text{N}(0,1)$. 

\subsection{Implementation: Alternative Methodologies}\label{sec:alternatives}

\noindent \textit{Gaussian Copula (GauC) \citep{hahn18}.} We use the predictive resampling implementation of \cite{fong2021martingale}, available at \url{https://github.com/edfong/MP}, which includes a dedicated Python module (we refer explicitly to the relevant files below). In Section~\ref{sec:ill_density}, the implementation is based on the file \texttt{1\_univariate\_copula.py}, with a minor modification: we impose a lower bound on the copula correlation parameter $\rho$, which in the authors' code is set to a value $\rho_{auto}$ following an automatic selection procedure. This change addresses the cases in which the automatic selection procedure yields $\rho_{auto} \approx 0$, effectively reducing the model to the independence copula, which results in no updates. We considered a grid of lower bounds $\rho_{lower}\in\{0.4, 0.5, 0.6\}$, set $\rho_{\text{selected}} = \max\{\rho_{\text{auto}}, \rho_{lower}\}$, and adopted the setting with the best empirical performance: $0.4$. Specifically, we set $\rho_{\text{selected}} = \max\{\rho_{\text{auto}}, 0.4\}$. This simple adjustment leads to substantial improvements over the default choice $\rho_{auto}$. In Section~\ref{sec:ill_data}, we used \texttt{1\_univariate\_copula.py} for the galaxy data, and \texttt{2\_bivariate\_copula.py} for the air quality data (no modifications). For further implementation details, we refer to \cite{fong2021martingale}.

\noindent \textit{Dirichlet Process Mixture (DPM) \citep{lo84}.} We use the \texttt{R} package \texttt{dirichletprocess} of \cite{ross2018dirichletprocess}, following the code provided at \url{https://github.com/edfong/MP}, which was designed to compare DPM with GauC (we refer to the exact scripts below). In Section~\ref{sec:ill_density}, we use \texttt{1\_gmm\_dpmm.R} for density estimation with a Gaussian location-scale Dirichlet Process mixture. For the Gaussian-Inverse Gamma base measure, we set the hyperparameters to $c(0, 1, 0.1, 0.1)$. In Section~\ref{sec:ill_regression}, we estimate a joint bivariate density using a multivariate Gaussian location-scale DPM. We set the hyperparameters of the multivariate Gaussian-Wishart base measure to $c(\textbf{0}, \textbf{I}, \textbf{I}, \textbf{I})$. From the posterior, we numerically approximate the regression function. In Section~\ref{sec:ill_data}, we use \texttt{1\_galaxy\_dpmm.R} for the galaxy data (Gaussian DPM with unknown mean and variance) and \texttt{4\_lidar\_DPMM.R} for the LIDAR data, applying the same regression function used in Section~\ref{sec:ill_regression}.

\noindent \textit{Gaussian Processes.} We use the \texttt{R} package \texttt{laGP} \citep{gramacy2016lagp}. In Section~\ref{sec:ill_regression}, we define an anisotropic GP object via \texttt{newGPsep}, optimize kernel parameters using \texttt{mleGPsep}, and predict on the test set using the mean from \texttt{predGPsep}.

\subsection{Implementation: Predictive Resampling for Kernels }\label{sec:implementation_details}
Algorithm~\ref{alg:KerP} outlines the implementation of predictive resampling with kernel methods (KerP) used in Section~\ref{sec:simulations}. Given a statistical estimator $f$, the predictive resampling component of the algorithm (lines $15$–$21$) follows a standard bootstrap-like scheme. What distinguishes KerP is its specific considerations regarding parameter selection. Apart from the choice of kernel (a comparison is included in the main manuscript), the only parameter to be specified is the bandwidth sequence $(h_n)_{n \geq 1}$. While bandwidth selection has been widely studied in the contexts of kernel density estimation and kernel regression (see \cite{jones1996brief} for a review), its application within a bootstrap (predictive resampling) framework is novel. We leave the development of a dedicated method for bandwidth selection tailored to predictive resampling to future work. Here, we propose a heuristic approach  (lines $1$–$14$) that combines existing bandwidth selection methods with the theoretical results presented in this paper. We describe the rationale for this proposal below.

First, observe that it is not necessary to choose an initial distribution $\alpha_0(\cdot)$ or to choose a full sequence $h_{1:n+M}$ in advance. Given a sample $x_{1:n}$, the predictive resampling procedure does not require the full sequence $(\alpha_i)_{0:n-1}$. We decompose the procedure into two stages: (1) estimate the predictive distribution $\alpha_n(\cdot)$, and (2) define a sequence $h_{n:n+M}$ to iteratively update $\alpha_n(\cdot)$ via \eqref{eq:kernel_Fn} and perform sampling.

In the first stage, we estimate the density $p_n(x)$ associated with $\alpha_n(\cdot)$ using a ``standard" kernel density estimator, where all observations are assigned the same bandwidth. This design choice allows us to leverage existing bandwidth selection methods. Specifically, choose a method $g$ that given a sample outputs a bandwidth (see next subsection for a discussion of the methods $g$ considered): $h_n = g(x_{1:n})$, where the subscript emphasizes the sample size dependency—common in most methods. Compute $h_n=g(x_{1:n})$ and then set $h_i = h_n$ for all $i \leq n$. This is a natural choice, as it corresponds to a standard kernel density estimate and ensures that all observed data points have equal resampling weights. As opposed to the ``synthetic" data points generated in predictive resampling which will have decaying weights. %Consequently, each observation contributes equally to the predictive distributions $(\alpha_m)_{m \geq n}$.

To construct the sequence $h_{n:n+M}$, we refer to Proposition~\ref{prop:kernel.acid}, which requires that $(h_n)_{n \geq 1}$ satisfies condition \eqref{eq:kernel_bandwidth_cond}. One way to ensure this is to use an exponentially decaying sequence; see the discussion in Section~\ref{sec:kernel}. We therefore define $h_{n+i} = b_1 e^{-b_2 i}$, where $b_1$ and $b_2$ are chosen such that the sequence interpolates between $h_n$ and $h_{n+M}$. How to obtain $h_n$ was described in the previous paragraph. To determine $h_{n+M}$, we draw a synthetic data set of size $M$ via the Bayesian bootstrap and evaluate $h_{n+M} = g(x_{1:n+M})$. %Although we do not perform density estimation at this step, we use $g$ to obtain the corresponding bandwidth. 
To reduce the variance of $h_{n+M}$ (it is a function of the synthetic data set generated through the Bayesian bootstrap), we repeat this step ten times and average the results. We then set $b_1 = h_n$ and compute $b_2 = \log(h_n / h_{n+M}) / M$. In numerical studies, we noticed that the bandwidth sequence can benefit from a scaling to avoid trajectories that are too smooth or too bumpy. We added a global scaling $c$ for this. Results in the simulation study are with $c=1$ (except in the data analysis section).

%\begin{algorithm}[!t]
%\caption{KerP predictive resampling (with bandwidth selection)}\label{alg:KerP}
%\footnotesize
%  Inputs: $x_{1:n}$, $K$ (kernel), $g$ (bandwidth selection method), $c$ (scaling factor bandwidth), $f$ (estimator of interest) \\
%\begin{algorithmic}[1]
%  \State Compute $h_n=f(x_{1:n})$. \Comment{Initialize bandwidth}
%  \State Set $h_{1:n-1}=h_n$.
%  \For {$j=1$ \textbf{to}  $10$} \Comment{Determine endpoint of the bandwidth sequence}
%   \State Sample $X_{n+1:M}$ given $x_{1:n}$ with the Bayesian Boostrap.
%   \State Compute $h^{(j)}_{n+M}=g(x_{1:n+M})$.
%  \EndFor
%  \State  $h_{n+M}=\frac{1}{10} \sum_{j=1}^{10} h^{(j)}_{n+M}$
%  \State Set $b_1=h_n$ and $b_2 = \log (h_n/h_{n+M})/M$ \Comment{Define exponentially decaying bandwidth sequence}
%  \For {$i=1$ \textbf{to}  $M-1$}
%  \State Compute $h_{n+i}=b_1 e^{-b_2 i}$
%  \EndFor
%  \State Scale $h_{1:n+M}=c \, h_{1:n+M}$
%  \State Estimate $p_n(x)$ via kernel density estimate with $K$ and $h_{1:n}$ (e.g., \eqref{eq:kernel_Fn}).  \Comment{Starting distribution for resampling}
%  \State Compute $\alpha_n(\cdot)$.
%  \For {$b=1$ \textbf{to}  $B$} \Comment{Predictive resampling}
%  \For{$i=0$ \textbf{to}  $M-1$} 
%  \State Sample $X_{n+i}\sim \alpha_n$.
%  \State Update $\alpha_{n+i}$ via \eqref{eq:kernel_Fn} with $x_{n+i}$ and $h_{n+i}$.
%  \EndFor
%  \State Compute $\hat{\theta}^{(b)}_{n+M}=f(x_{1:n+M})$.
%  \EndFor
%\end{algorithmic}
%  \Return{Martingale Posterior : $\hat{\theta}^{(1)}_{n+M},\ldots, \hat{\theta}^{(B)}_{n+M}$}.
%\end{algorithm}

\begin{algorithm}[!t]
\caption{KerP predictive resampling (with bandwidth selection)}\label{alg:KerP}
\footnotesize
  Inputs: $x_{1:n}$, $K$ (kernel), $g$ (bandwidth selection method), $c$ (scaling factor bandwidth), $f$ (estimator of interest) 
\begin{algorithmic}[1]
  \State Compute $h_n=f(x_{1:n})$. \Comment{Initialize bandwidth}
  \State Set $h_{1:n-1}=h_n$.
  \For {$j=1$ \textbf{to}  $10$} \Comment{Determine endpoint of the bandwidth sequence}
   \State Sample $X_{n+1:M}$ given $x_{1:n}$ with the Bayesian Bootstrap.
   \State Compute $h^{(j)}_{n+M}=g(x_{1:n+M})$.
  \EndFor
  \State  $h_{n+M}=\frac{1}{10} \sum_{j=1}^{10} h^{(j)}_{n+M}$
  \State Set $b_1=h_n$ and $b_2 = \log (h_n/h_{n+M})/M$ \Comment{Define exponentially decaying bandwidth sequence}
  \For {$i=1$ \textbf{to}  $M-1$}
  \State Compute $h_{n+i}=b_1 e^{-b_2 i}$
  \EndFor
  \State Scale $h_{1:n+M}=c \, h_{1:n+M}$
  \State Estimate $p_n(x)$ via kernel density estimate with $K$ and $h_{1:n}$ (such as \eqref{eq:kernel_Fn}).  \Comment{Starting distribution for resampling}
  \State Compute $\alpha_n(\cdot)$.
  
  \For {$b=1$ \textbf{to}  $B$} \Comment{Predictive resampling}
  \For{$i=0$ \textbf{to}  $M-1$} 
  \State Sample $X_{n+i}\sim \alpha_n$.
  \State Update $\alpha_{n+i}$ via \eqref{eq:kernel_Fn} with $x_{n+i}$ and $h_{n+i}$.
  \EndFor
  \State Compute $\hat{\theta}^{(b)}_{n+M}=f(x_{1:n+M})$.
  \EndFor
\end{algorithmic}
  \Return{Martingale Posterior : $\hat{\theta}^{(1)}_{n+M},\ldots, \hat{\theta}^{(B)}_{n+M}$}.
\end{algorithm}

\subsection{Sensitivity to Bandwidth Selection Method}\label{sec:sensitivity_bw}

There are many bandwidth selection functions $g$ available; see \cite{jones1996brief}. In this context, there is no difference between density estimation and regression. Recall that even when we use KerP for regression, we need a multivariate density to sample the synthetic data (see Section~\ref{sec:kernel_regr}). However, the bandwidth selection methods differ depending on whether the problem is univariate or multivariate. We explore here the sensitivity of KerP to several bandwidth selection methods available in standard \texttt{R} packages.

In the univariate case, we consider methods implemented in the \texttt{density} function from the \texttt{stats} package, which serves as the base kernel estimator in \texttt{R}: Silverman's ``rule of thumb" (\texttt{nrd0})\citep{silverman2018density}, Scott's variation of Silverman's ``rule of thumb" (\texttt{nrd}) \cite{scott2015multivariate},
the plug-in method of \cite{sheather1991reliable} (\texttt{SJ}) and unbiased cross-validation (\texttt{ucv}). For multivariate settings, we use the \texttt{ks} package \citep{duong2024package}, using the plug-in method of \cite{wand1994multivariate} (\texttt{pi}) and smoothed cross-validation (\texttt{scv}). The corresponding \texttt{R} function names are indicated in parentheses.

To compare univariate methods, we replicate the numerical study from Section~\ref{sec:ill_density}. Table~\ref{tab:bw_density} summarizes the median metrics from that same simulation study (univariate cumulative distribution function estimation). While the main manuscript discusses differences between kernel choices, here we focus on comparing bandwidth selection strategies. Overall, performance across SJ and UCV is fairly similar, with the plug-in approach of \cite{sheather1991reliable} slightly outperforming cross-validation. Notably, scaling the bandwidth sequence using the parameter $c$ (see line 14 of Algorithm~\ref{alg:KerP}) can improve performance. For example, the Gaussian kernel with $c=0.9$ outperforms $c=1$—the default shown in the manuscript—by yielding lower bias, higher coverage, and similar width. However, this improvement is inconsistent across kernels: the uniform kernel performs better with $c=1.1$. In this context, Silverman's rule of thumb does not seem to perform well, nor does Scott's variation of Silverman's rule of thumb. 

\begin{table}[!t] \centering 
  \label{tab:bw_density}
\scalebox{0.75}{\begin{tabular}{@{\extracolsep{5pt}} cl|ccc|ccc} 
	\\[-1.8ex]\hline 
& & \multicolumn{3}{c }{n = 50 }& \multicolumn{3}{ c}{n = 200 } \\
Ker & Bw & $env$ & $dev$ & $awd$& $env$ & $dev$ & $awd$ \\ 
\hline 
\multirow{12}{*}{\rotatebox{90}{Gaussian}}   
 & 0.9nrd & $0.784$ & $0.025$ & $0.090$ & $0.725$ & $0.013$ & $0.043$ \\ 
  & 0.9nrd0 & $0.828$ & $0.022$ & $0.089$ & $0.772$ & $0.011$ & $0.043$ \\ 
  & 0.9SJ & $0.831$ & $0.020$ & $0.092$ & $0.772$ & $0.011$ & $0.044$ \\ 
  & 0.9ucv & $0.841$ & $0.020$ & $0.093$ & $0.756$ & $0.010$ & $0.046$ \\ 
  & 1nrd & $0.700$ & $0.029$ & $0.091$ & $0.669$ & $0.016$ & $0.044$ \\ 
  & 1nrd0 & $0.800$ & $0.024$ & $0.089$ & $0.762$ & $0.013$ & $0.043$ \\ 
  & 1SJ & $0.825$ & $0.021$ & $0.093$ & $0.769$ & $0.011$ & $0.044$ \\ 
  & 1ucv & $0.812$ & $0.022$ & $0.094$ & $0.759$ & $0.011$ & $0.046$ \\ 
  & 1.1nrd & $0.588$ & $0.034$ & $0.091$ & $0.600$ & $0.018$ & $0.043$ \\ 
  & 1.1nrd0 & $0.756$ & $0.027$ & $0.090$ & $0.716$ & $0.014$ & $0.043$ \\ 
  & 1.1SJ & $0.819$ & $0.023$ & $0.093$ & $0.731$ & $0.012$ & $0.044$ \\ 
  & 1.1ucv & $0.781$ & $0.025$ & $0.095$ & $0.769$ & $0.012$ & $0.046$ \\ 
 \hline
\multirow{12}{*}{\rotatebox{90}{Laplace}}   & 0.9nrd & $0.359$ & $0.039$ & $0.091$ & $0.425$ & $0.022$ & $0.044$ \\ 
  & 0.9nrd0 & $0.575$ & $0.032$ & $0.092$ & $0.600$ & $0.017$ & $0.044$ \\ 
  & 0.9SJ & $0.753$ & $0.026$ & $0.095$ & $0.712$ & $0.014$ & $0.045$ \\ 
  & 0.9ucv & $0.700$ & $0.029$ & $0.097$ & $0.747$ & $0.013$ & $0.046$ \\ 
  & 1nrd & $0.269$ & $0.044$ & $0.091$ & $0.303$ & $0.025$ & $0.045$ \\ 
  & 1nrd0 & $0.441$ & $0.036$ & $0.090$ & $0.481$ & $0.020$ & $0.044$ \\ 
  & 1SJ & $0.691$ & $0.029$ & $0.095$ & $0.666$ & $0.016$ & $0.045$ \\ 
  & 1ucv & $0.637$ & $0.032$ & $0.099$ & $0.669$ & $0.014$ & $0.047$ \\ 
  & 1.1nrd & $0.231$ & $0.050$ & $0.090$ & $0.216$ & $0.028$ & $0.044$ \\ 
  & 1.1nrd0 & $0.334$ & $0.041$ & $0.091$ & $0.359$ & $0.023$ & $0.044$ \\ 
  & 1.1SJ & $0.637$ & $0.032$ & $0.095$ & $0.547$ & $0.018$ & $0.045$ \\ 
  & 1.1ucv & $0.562$ & $0.036$ & $0.098$ & $0.653$ & $0.016$ & $0.047$ \\ 
\hline
 \multirow{12}{*}{\rotatebox{90}{Uniform}}   & 0.9nrd & $0.781$ & $0.019$ & $0.088$ & $0.719$ & $0.010$ & $0.043$ \\ 
  & 0.9nrd0 & $0.728$ & $0.018$ & $0.088$ & $0.706$ & $0.010$ & $0.044$ \\ 
  & 0.9SJ & $0.697$ & $0.018$ & $0.090$ & $0.684$ & $0.009$ & $0.044$ \\ 
  & 0.9ucv & $0.713$ & $0.018$ & $0.091$ & $0.678$ & $0.009$ & $0.045$ \\ 
  & 1nrd & $0.800$ & $0.020$ & $0.088$ & $0.744$ & $0.011$ & $0.044$ \\ 
  & 1nrd0 & $0.762$ & $0.020$ & $0.088$ & $0.716$ & $0.010$ & $0.044$ \\ 
  & 1SJ & $0.725$ & $0.018$ & $0.090$ & $0.706$ & $0.010$ & $0.044$ \\ 
  & 1ucv & $0.741$ & $0.018$ & $0.091$ & $0.688$ & $0.009$ & $0.045$ \\ 
  & 1.1nrd & $0.831$ & $0.020$ & $0.088$ & $0.747$ & $0.011$ & $0.043$ \\ 
  & 1.1nrd0 & $0.803$ & $0.019$ & $0.088$ & $0.725$ & $0.010$ & $0.043$ \\ 
  & 1.1SJ & $0.750$ & $0.018$ & $0.090$ & $0.725$ & $0.010$ & $0.044$ \\ 
  & 1.1ucv & $0.772$ & $0.019$ & $0.092$ & $0.713$ & $0.010$ & $0.045$ \\ 
\end{tabular} }
 \caption{Simulations: sensitivity to bandwidth in cumulative distribution function estimation. Median across $100$ data sets for different sample sizes (n). $env$ is a measure of coverage (the higher the better), $dev$ is an L1 bias (the lower the better), and $awd$ is the average width of the credible bands (the lower the better). In KerP, column Ker refers to the kernel, column bw to the bandwidth: nrd0 stands for \cite{silverman2018density} rule of thumb, nrd for \cite{scott2015multivariate} modification of Silverman's rule of thumb, SJ stands for the bandwidth selection of \cite{sheather1991reliable}, and ucv for unbiased cross-validation. The number in front of the bw method refers to the scaling factor $c$.}
\end{table} 

To compare multivariate methods, we replicated the simulation study from Section~\ref{sec:ill_regression}. %In this setting, we first estimate a bivariate density, followed by the computation of a univariate regression function. 
Table~\ref{tab:bw_regression} summarizes the median metrics considering the three data-generating mechanisms (regression function estimation). Overall, the performance is comparable across bandwidth selection methods (\texttt{pi}, \texttt{cv}) and scaling values ($c \in \{0.9, 1, 1.1\}$). A possible exception is the case of the Gaussian kernel with $c = 0.9$, which exhibits lower bias. Although the performance of the Gaussian kernel is inferior to the other kernel types.

Lastly, note that we discussed the possibility of using one sequence of bandwidths for generating synthetic data and another for estimation. The latter can be set according to the standard rates given by consistency results. We considered both options.

\begin{table}[!t] \centering 
  \label{tab:bw_regression}
	\scalebox{0.75}{\begin{tabular}{@{\extracolsep{5pt}} cl|cccc|cccc|cccc} 
\\[-1.8ex]\hline 
& & \multicolumn{4}{c}{DGM1: Gaussian $\epsilon$} & \multicolumn{4}{c}{DGM2: Student-t $\epsilon$} &\multicolumn{4}{c}{DGM3: Gauss-Exp Kernel}\\
Ker & Bw & $env$ & $dev$ & $awd$ & $MSE$ & $env$ & $dev$ & $awd$ & $MSE$ & $env$ & $dev$ & $awd$ & $MSE$ \\ 
\hline \\[-1.8ex] 
\multirow{6}{*}{\rotatebox{90}{Gaussian}} & 0.9pi & $0.845$ & $0.183$ & $0.677$ & $1.045$ & $0.850$ & $0.173$ & $0.662$ & $1.044$ & $0.880$ & $0.202$ & $0.806$ & $1.593$ \\ 
  & 0.9scv & $0.830$ & $0.184$ & $0.661$ & $1.045$ & $0.830$ & $0.174$ & $0.647$ & $1.047$ & $0.860$ & $0.210$ & $0.773$ & $1.597$ \\ 
  & 1pi & $0.810$ & $0.186$ & $0.662$ & $1.047$ & $0.820$ & $0.180$ & $0.646$ & $1.048$ & $0.860$ & $0.209$ & $0.781$ & $1.596$ \\ 
  & 1scv & $0.800$ & $0.191$ & $0.640$ & $1.049$ & $0.810$ & $0.181$ & $0.630$ & $1.049$ & $0.830$ & $0.216$ & $0.743$ & $1.599$ \\ 
  & 1.1pi & $0.795$ & $0.194$ & $0.647$ & $1.054$ & $0.800$ & $0.181$ & $0.631$ & $1.048$ & $0.850$ & $0.216$ & $0.765$ & $1.601$ \\ 
  & 1.1scv & $0.770$ & $0.195$ & $0.631$ & $1.053$ & $0.790$ & $0.187$ & $0.618$ & $1.052$ & $0.815$ & $0.226$ & $0.729$ & $1.602$ \\ 
  \hline
 \multirow{6}{*}{\rotatebox{90}{Laplace}} & 0.9pi & $0.910$ & $0.166$ & $0.736$ & $1.033$ & $0.910$ & $0.164$ & $0.722$ & $1.026$ & $0.920$ & $0.205$ & $0.882$ & $1.586$ \\ 
  & 0.9scv & $0.900$ & $0.163$ & $0.712$ & $1.028$ & $0.895$ & $0.161$ & $0.698$ & $1.026$ & $0.910$ & $0.201$ & $0.847$ & $1.579$ \\ 
  & 1pi & $0.900$ & $0.167$ & $0.712$ & $1.030$ & $0.900$ & $0.164$ & $0.701$ & $1.029$ & $0.920$ & $0.198$ & $0.861$ & $1.580$ \\ 
  & 1scv & $0.890$ & $0.164$ & $0.694$ & $1.033$ & $0.890$ & $0.161$ & $0.682$ & $1.032$ & $0.910$ & $0.200$ & $0.831$ & $1.579$ \\ 
  & 1.1pi & $0.890$ & $0.165$ & $0.689$ & $1.030$ & $0.880$ & $0.161$ & $0.680$ & $1.034$ & $0.910$ & $0.195$ & $0.833$ & $1.582$ \\ 
  & 1.1scv & $0.880$ & $0.168$ & $0.682$ & $1.035$ & $0.870$ & $0.164$ & $0.665$ & $1.033$ & $0.900$ & $0.204$ & $0.800$ & $1.584$ \\ 
  \hline
 \multirow{6}{*}{\rotatebox{90}{Uniform}} & 0.9pi & $0.890$ & $0.185$ & $0.769$ & $1.043$ & $0.900$ & $0.180$ & $0.756$ & $1.040$ & $0.870$ & $0.234$ & $0.943$ & $1.606$ \\ 
  & 0.9scv & $0.905$ & $0.175$ & $0.741$ & $1.042$ & $0.895$ & $0.173$ & $0.725$ & $1.041$ & $0.870$ & $0.224$ & $0.897$ & $1.603$ \\ 
  & 1pi & $0.910$ & $0.177$ & $0.737$ & $1.042$ & $0.900$ & $0.174$ & $0.726$ & $1.039$ & $0.865$ & $0.222$ & $0.914$ & $1.606$ \\ 
  & 1scv & $0.910$ & $0.169$ & $0.713$ & $1.036$ & $0.900$ & $0.166$ & $0.701$ & $1.041$ & $0.885$ & $0.211$ & $0.867$ & $1.594$ \\ 
  & 1.1pi & $0.910$ & $0.170$ & $0.719$ & $1.035$ & $0.900$ & $0.166$ & $0.698$ & $1.036$ & $0.870$ & $0.218$ & $0.880$ & $1.596$ \\ 
  & 1.1scv & $0.900$ & $0.165$ & $0.684$ & $1.035$ & $0.900$ & $0.160$ & $0.671$ & $1.035$ & $0.880$ & $0.204$ & $0.836$ & $1.593$ \\ 
\hline 
\end{tabular}}
 \caption{Simulations: sensitivity to bandwidth in regression function estimation. Median across $100$ data sets for three different data generating mechanisms ($n_{train}=200$,$n_{test}=100)$. $env$ is a measure of coverage (where higher is the better), $dev$ is an L1 bias (the lower the better), and $awd$ is the average width of the credible bands (the lower the better), $MSE$ is the sample mean squared error on test data (the lower the better). In KerP, column Ker refers to the kernel, column bw to the bandwidth:  pi stands for the bandwidth selection of \cite{wand1994multivariate}, and scv for smoothed cross-validation. The number in front of the bw method refers to the scaling factor $c$. }
\end{table} 

\clearpage
\section{Data analysis}

\subsection{Air Quality Data}

Figure~\ref{fig:airquality_sd} depicts the posterior standard deviation of the density estimates obtained with KerP and GauC. KerP is run with a Gaussian kernel with $(h_n)_{n \geq 1}$ chosen with the bandwidth multivariate method of \cite{wand1994multivariate}, $M=2000$, $B=500$, and $c=1.2$.

\begin{figure}[!t]
        \begin{center}
            \includegraphics[scale=0.9]{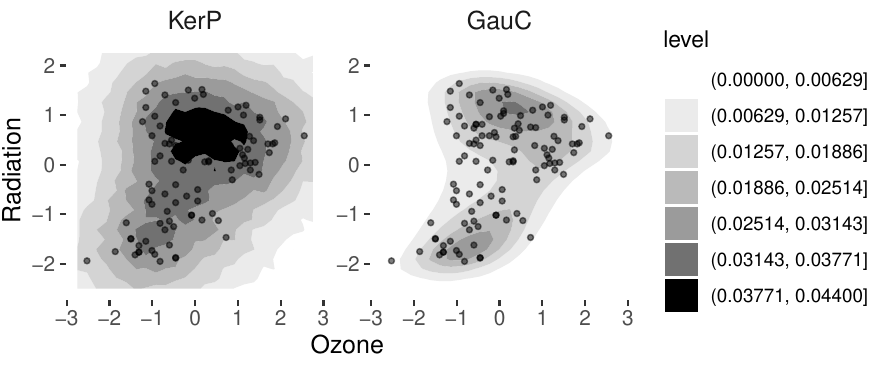}
        \end{center}
        \caption{Air quality data: posterior standard deviations. In each panel: contour plot of the posterior standard deviations. Dots depict the observations. The two panels differ by the method: KerP with Gaussian kernel and bandwidth via the multivariate plug-in method of  \cite{wand1994multivariate} (left),  and Gaussian copula with the same parameters as in \cite{fong2021martingale}(right)}
        \label{fig:airquality_sd}
    \end{figure}

\subsection{LIDAR}

Figures~\ref{fig:lidar_density0} and \ref{fig:lidar_density-3} depict the condition density $p(y|x=0)$ and $p(y|x=-3)$ of the LIDAR regression data set. KerP is run with a Gaussian kernel with $(h_n)_{n \geq 1}$ chosen with the bandwidth multivariate method of \cite{wand1994multivariate}, $M=1000$, $B=600$, and $c=5$. The regression function in the manuscript is obtained with the same parameterization except for $c=2$. The same setup is used for the out-of-sample prediction exercise. 

  \begin{figure}[!t]
        \begin{center}
            \includegraphics[scale=0.8]{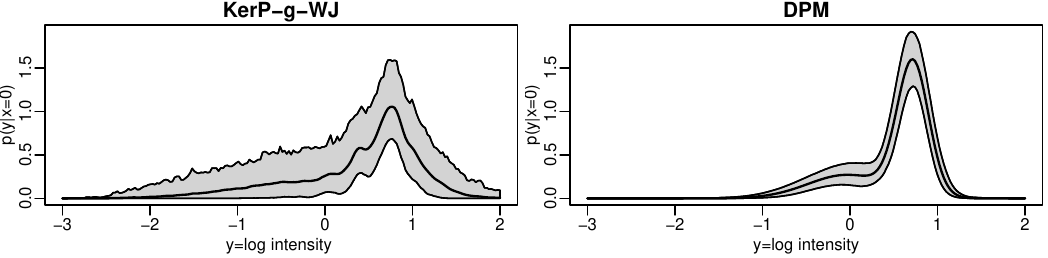}
        \end{center}
        \caption{LIDAR: conditional density $p(y|x=0$). In each panel: posterior mean $\hat{p}$ (solid line), and $95\%$ credible regions (gray shaded area). The two panels differ by the method: KerP with Gaussian kernel and bandwidth via the multivariate plug-in method of \cite{wand1994multivariate} (left), and DPM (right).}
        \label{fig:lidar_density0}
    \end{figure}

    \begin{figure}[!t]
        \begin{center}
            \includegraphics[scale=0.8]{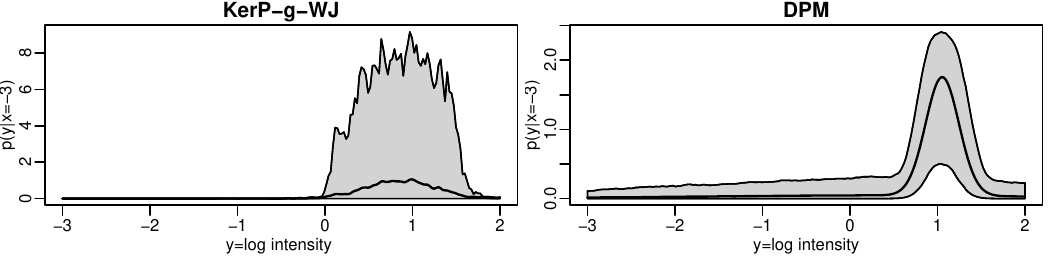}
        \end{center}
        \caption{LIDAR: conditional density $p(y|x=-3$). In each panel: posterior mean $\hat{p}$ (solid line), and $95\%$ credible regions (gray shaded area). The two panels differ by the method: KerP with Gaussian kernel and bandwidth via the multivariate plug-in method of \cite{wand1994multivariate} (left), and DPM (right). Note: the $y$-axis in the two panels is on a different scale.}
        \label{fig:lidar_density-3}
    \end{figure}

\end{document}